\documentclass[12pt,twoside,a4paper]{article}

\usepackage[utf8]{inputenc}
\usepackage{amsmath, amssymb, mathtools, enumerate}
\usepackage{mathrsfs} 
\usepackage{amsfonts, amsthm, stmaryrd, fullpage}
\usepackage{authblk}
\usepackage{tikz-cd}
\usepackage[german, english]{babel}
\usepackage[nottoc,numbib]{tocbibind}

\usepackage{color}
\usepackage{hyperref}
\hypersetup{colorlinks=true,
linktoc=all,
citecolor=blue,
filecolor=black,
linkcolor=blue,
urlcolor=blue}

\newtheorem{thm}{Theorem}[section]
\newtheorem{lemma}[thm]{Lemma}
\newtheorem{cor}[thm]{Corollary}
\newtheorem{prop}[thm]{Proposition}
\newtheorem{prop-def}[thm]{Proposition-Definition}

\theoremstyle{definition} 
\newtheorem{mydef}[thm]{Definition}
\newtheorem{notation}[thm]{Notation}
  
\newtheorem{example}[thm]{Example}

\theoremstyle{remark}
\newtheorem{rmk}[thm]{Remark}

\newcommand\ad{{\rm ad}}
\newcommand\Aff{{\rm Aff}}

\newcommand\Bl{{\rm Bl}}

\newcommand\id{{\rm id}}

\newcommand\Proj{{\rm Proj}}

\newcommand\rig{{\rm rig}}

\newcommand\spc{{\rm sp}}
\newcommand\Sp{{\rm Sp}}
\newcommand\Spa{{\rm Spa}}

\newcommand\Spec{{\rm Spec}} 
\newcommand\Spf{{\rm Spf}}

\begin{document}

\title{Formal models for relative adic spaces}
\author{Dimitri Dine \\ 
\small University of California San Diego \\ La Jolla, United States \\
\text{ddine@ucsd.edu} \\ \small \href{https://orcid.org/0000-0002-8340-5358}{ORCiD: 0000-0002-8340-5358}}
\date{}
\maketitle

\begin{abstract}We extend Raynaud's theory of formal models from rigid-analytic spaces over a nonarchimedean field to uniform qcqs adic spaces $X$, with no finite-type assumptions, over an arbitrary Tate affinoid base $S$. The key new ingredient is the notion of a normalized formal blow-up which takes on the role played by admissible formal blow-ups in the classical theory. \end{abstract}

\textit{2020 Mathematics Subject Classification:} 14G22, 14G45

\textit{Keywords:} formal schemes, adic spaces, perfectoid spaces, formal models, generic fiber, admissible formal blow-up

\tableofcontents

\section{Introduction}\label{sec:introduction}

\subsection{Motivation}

In his 1974 report \cite{RaynaudReport}, Raynaud introduced new ideas into rigid-analytic geometry which allow us to view rigid geometry as a ``birational geometry of formal schemes''. In the decades that followed, Raynaud's theory of formal models established itself as an indispensable tool, with applications to many topics of interest in algebraic geometry and number theory: Néron models of abelian varieties \cite{Bosch-Xarles}, rigid-analytic Picard varieties \cite{HL00}, classification of bounded smooth rigid-analytic groups \cite{Luetkebohmert}, proper morphisms in rigid geometry \cite{L90}, descent for coherent sheaves on rigid-analytic spaces \cite{Bosch-Goertz}, \cite{ConradDescent}, canonical subgroups and overconvergent modular forms \cite{Andreatta-Gasbarri}, \cite{Conrad06-2}, \cite{Conrad06-3}, \cite{Kisin-Lai}, \cite{Bijakowski-Pilloni-Stroh}, \cite{Pilloni20}, nonarchimedean Arakelov theory \cite{Gubler98}, \cite{Gubler-Kuennemann17} (via formal metrics on line bundles), motivic integration and complex singularities \cite{Loeser-Sebag}, \cite{Nicaise-Sebag} (via weak Néron models), Raynaud's and Harbater's proof of Abhyankar's conjecture \cite{Raynaud94}, \cite{Harbater94}, mod-$p$ Poincaré duality on smooth rigid-analytic varieties \cite{Zavyalov25}, rigid-analytic $D$-modules \cite{Ardakov-Wadsley}, Zariski-constructible sheaves on rigid-analytic varieties \cite{Bhatt-Hansen} -- the list could go on. On the other hand, since the introduction of perfectoid spaces in \cite{Scholze}, modern arithmetic geometry abounds with important examples of analytic adic spaces which do not satisfy any Noetherian or finite type assumptions: Shimura varieties at infinite level (also known as perfectoid Shimura varieties), as in the work of Scholze \cite{Scholze15} and Caraiani-Scholze \cite{CS17}, generic fibers of Rapoport-Zink spaces at infinite level \cite{SW13}, relative Fargues-Fontaine curves (\cite{Kedlaya-Liu}, \S8.7) over non-zero-dimensional perfectoid spaces of characteristic $p$, to name just a few. Thus it appears desirable to have an analog of Raynaud's theory of formal models for such spaces.

An important step towards this goal was taken in the book of Fujiwara and Kato \cite{FK}, where the authors define the category of qcqs rigid spaces as the category of qcqs adic formal schemes of finite ideal type localized by admissible formal blow-ups (more general rigid spaces are then defined by gluing qcqs ones) and where they prove many basic properties of admissible formal blow-ups for general adic formal schemes of finite ideal type, with no finiteness or Noetherian assumptions. At the same time, their result on comparison with adic spaces (\cite{FK}, Appendix~A to Ch.~II, Theorem~A.5.2) still relies both on a Noetherian assumption on the base space and on a (topologically) finite-type assumption on the adic spaces and formal schemes involved.

In \cite{AGV}, Corollary~1.2.7, Ayoub, Gallauer and Vezzani constructed an embedding of the category of uniform analytic adic spaces, with no finite-type assumptions, into the category of rigid spaces in the sense of Fujiwara-Kato, thus establishing the existence of formal models for qcqs uniform analytic adic spaces. However, this result still does not describe the category of (qcqs) uniform adic spaces as a localization of a category of formal models. Indeed, the embedding functor of loc.~cit.~is given on affinoids by \begin{equation*}\Spa(A, A^{+})\mapsto \Spf(A^{+})^{\rig}\end{equation*}(so all rigid spaces in the essential image have formal models which are integrally closed in their generic fiber), but an admissible formal blow-up of the formal scheme $\Spf(A^{+})$ can fail to be integrally closed in its generic fiber (consider the admissible formal blow-up in any finitely generated open ideal of $A^{+}$ which is not a normal ideal).

Hence, thus far, there has been no general analog of Raynaud theory for qcqs uniform analytic adic spaces satisfying no Noetherian or finite-type conditions, such as (qcqs) perfectoid spaces. The aim of this paper is to produce such a theory.

\subsection{Main theorems}       

Our set-up is as follows. We fix a complete adic ring $R$ whose ideal of definition is generated by a single non-zero-divisor and non-unit $\varpi$ and we assume that the Tate ring $R[\varpi^{-1}]$, with pair of definition $(R, \varpi)$, is sheafy, so that $S=\Spa(R[\varpi^{-1}], \overline{R})$ (where $\overline{R}$ denotes the integral closure of $R$ in $R[\varpi^{-1}]$) is an affinoid adic space. We consider adic spaces $X\to S$ over $S$, with $X$ quasi-compact quasi-separated. We ask for a theory of formal models of $X\to S$ over $\Spf(R)$ which would work under mild assumptions on $X$ (such as qcqs and uniform) and with no finite-type or Noetherian assumptions. To begin with, let us introduce the relevant category of formal schemes in which our formal models should live.
\begin{mydef}[Definition~\ref{Locally rig-sheafy formal schemes}, Definition~\ref{Locally rig-sheafy formal schemes 2}, Definition~\ref{Locally stably uniform adic formal scheme}]Let $(R, \varpi)$ be as above. \begin{enumerate}[(1)]\item An adic formal $R$-scheme is a formal scheme $\mathfrak{X}$ equipped with an adic morphism $\mathfrak{X}\to\Spf(R)$, i.e., for every affine open subset $\mathfrak{U}=\Spf(A)$ an ideal of definition of the adic ring $A$ is generated by the image of $\varpi$. \item An adic formal $R$-scheme $\mathfrak{X}$ is called locally rig-sheafy if it can be covered by affine open subsets $\mathfrak{U}=\Spf(A)$ such that the Tate ring $A[\varpi^{-1}]$, with pair of definition $(A, \varpi)$, is sheafy. \item An adic formal $R$-scheme $\mathfrak{X}$ is called locally stably uniform if it can be covered by affine open subsets $\mathfrak{U}=\Spf(A)$ such that the Tate ring $A[\varpi^{-1}]$ is stably uniform. \end{enumerate}\end{mydef}
By a theorem of Buzzard-Verberkmoes-Mihara \cite{BV}, \cite{Mihara}, stably uniform Tate rings are sheafy, so locally stably uniform adic formal $R$-schemes are also locally rig-sheafy. In the body of the paper (Definition~\ref{Locally rig-sheafy formal schemes}, Definition~\ref{Locally rig-sheafy formal schemes 2}), we also introduce locally rig-sheafy formal $R$-schemes which are not necessarily adic over $R$; however, the adic case is sufficient for formulating our main results on the existence of formal models. As in the classical theory of Raynaud, we associate to every locally rig-sheafy adic formal $R$-scheme $\mathfrak{X}$ an analytic adic space $\mathfrak{X}_{\eta}^{\ad}$, called the adic analytic generic fiber of $\mathfrak{X}$. For $\mathfrak{X}=\Spf(A)$ affine, this is simply the affinoid adic space \begin{equation*}X=\Spa(A[\varpi^{-1}], \overline{A}),\end{equation*}where the topology on the Tate ring $A[\varpi^{-1}]$ is given by the pair of definition $(A, \varpi)$ and where $\overline{A}$ denotes the integral closure of $A$ in $A[\varpi^{-1}]$ (note that the locally rig-sheafy assumption is necessary to ensure that $X$ is indeed an adic space). For general locally rig-sheafy adic formal $R$-schemes $\mathfrak{X}$, the adic space $\mathfrak{X}_{\eta}^{\ad}$ is obtained by gluing the adic analytic generic fibers of affine open subsets of $\mathfrak{X}$. The assignment \begin{equation*}\mathfrak{X}\mapsto \mathfrak{X}_{\eta}^{\ad}\end{equation*}is functorial; we denote by $f_{0\eta}$ the generic fiber of a morphism $f_{0}$. We obtain a corresponding notion of formal $R$-models. 
\begin{mydef}[Formal models, Definition~\ref{Definition of formal models}]Let $(R, \varpi)$ be as before and let \begin{equation*}S=\Spa(R[\varpi^{-1}], \overline{R}).\end{equation*}For an adic space $X\to S$ over $S$, a formal $R$-model $\mathfrak{X}$ is an adic formal $R$-scheme $\mathfrak{X}$ such that \begin{equation*}\mathfrak{X}_{\eta}^{\ad}\cong X.\end{equation*}For a morphism $f: X\to X'$ of adic spaces over $S$, a formal $R$-model of $f$ is a morphism of adic formal $R$-schemes $f_{0}: \mathfrak{X}\to\mathfrak{X}'$ such that $f_{0\eta}=f$. A formal $R$-model $\mathfrak{X}$ of an adic space $X$ over $S$ is called integrally closed if for every affine open subset $\mathfrak{U}=\Spf(A)$ of $\mathfrak{X}$ the ring $A$ is integrally closed in $A[\varpi^{-1}]$.\end{mydef}
Recall that we want to develop a theory of formal $R$-models for relative adic spaces $X\to S$ which avoids finite-type assumptions on the structure morphism $X\to S$ and Noetherian assumptions on $X$ or $S$. To achieve this goal whenever $X$ is uniform qcqs, we use the notion of normalization of a locally rig-sheafy adic formal $R$-scheme inside its generic fiber, which is based on the theory of adically quasi-coherent sheaves over adic formal schemes worked out by Fujiwara and Kato in \cite{FK}; we recall the definition of adically quasi-coherent sheaves in Definition~\ref{Adically quasi-coherent sheaf}. The following key proposition and definition was inspired by a result of Pilloni and Stroh (\cite{Pilloni-Stroh16}, Proposition~1.1). 
\begin{prop-def}[Proposition~\ref{Quasi-coherent algebras and formal models}, Proposition~\ref{Normalization of a formal model}, Definition~\ref{Definition of normalization}]\label{Normalization, introduction}Fix $(R, \varpi)$ as before. Let $\mathfrak{S}$ be a locally stably uniform $\varpi$-torsion-free qcqs adic formal $R$-scheme with adic analytic generic fiber $S=\mathfrak{S}_{\eta}^{\ad}$. Then $\spc_{S,\mathfrak{S}\ast}\mathcal{O}_{S}^{+}$ is an adically quasi-coherent algebra on $\mathfrak{S}$ in the sense of Fujiwara-Kato and its formal spectrum \begin{equation*}\mathfrak{X}=\Spf(\spc_{S,\mathfrak{S}\ast}\mathcal{O}_{S}^{+})\end{equation*}is an integrally closed formal $R$-model of $S$. We call $\mathfrak{X}$ (and the canonical affine morphism of formal schemes $\mathfrak{X}\to\mathfrak{S}$) the normalization of $\mathfrak{S}$ in its generic fiber.\end{prop-def}
The notion of formal spectrum of an adically quasi-coherent algebra referred to in the above proposition was introduced and studied by Fujiwara-Kato in \cite{FK}, Ch.~I, \S4.1(c). Besides the theory of adically quasi-coherent sheaves, we also rely on the work of Fujiwara-Kato for the basic properties of admissible formal blow-ups of adic formal $R$-schemes which were proved in \cite{FK} without Noetherian or topologically of finite type assumptions on the formal schemes involved, see \cite{FK}, Ch.~II, \S1. We can then define the notion of a normalized formal blow-up, which in our theory takes on the role played by admissible formal blow-ups in Raynaud's classical theory.
\begin{mydef}[Normalized formal blow-up]\label{Normalized formal blow-up, introduction}Let $\mathfrak{X}$ be a locally rig-sheafy qcqs $\varpi$-torsion-free adic formal $R$-scheme with uniform generic fiber over $\Spa(R[\varpi^{-1}], \overline{R})$. A morphism $f_{0}: \mathfrak{Z}'\to\mathfrak{X}$ is called a \textit{normalized formal blow-up} if it is the composition of a $\varpi$-torsion-free admissible formal blow-up $\mathfrak{X}'\to\mathfrak{X}$ followed by its normalization $\mathfrak{Z}'\to\mathfrak{X}'$ in the sense of Proposition-Definition~\ref{Normalization, introduction}.\end{mydef}
We can now finally state the main theorem of the paper, which provides a generalization of Raynaud theory to qcqs uniform adic spaces over any Tate affinoid base $S$, with no finite-type assumptions on the structure morphism $X\to S$.
\begin{thm}[Theorem~\ref{Analog of Raynaud theory}]\label{Main theorem}Let $(R, \varpi)$ be as before and let $S=\Spa(R[\varpi^{-1}], \overline{R})$, where $\overline{R}$ denotes the integral closure of $R$ inside $R[\varpi^{-1}]$. The functor\begin{equation*}\mathfrak{X}\mapsto \mathfrak{X}_{\eta}^{\ad}\end{equation*}gives rise to an equivalence of categories between \begin{enumerate}[(1)]\item the category of locally stably uniform $\varpi$-torsion-free quasi-compact quasi-separated adic formal $R$-schemes which are integrally closed in their generic fibers, localized by the right multiplicative system of compositions of normalized formal blow-ups, and 
\item the category of uniform quasi-compact quasi-separated adic spaces over $S$.\end{enumerate}\end{thm}
On the other hand, if the structure morphism to $S$ is assumed to be of finite type, we obtain the following more immediate analog of Raynaud theory. 
\begin{thm}[Theorem~\ref{Analog of Raynaud theory 2}]\label{Main theorem 2}Let $(R, \varpi)$ be as before, with $R$ integrally closed in $R[\varpi^{-1}]$, and let $S=\Spa(R[\varpi^{-1}], R)$. The functor\begin{equation*}\mathfrak{X}\mapsto \mathfrak{X}_{\eta}^{\ad}\end{equation*}gives rise to an equivalence of categories between \begin{enumerate}[(1)]\item the category of locally rig-sheafy $\varpi$-torsion-free quasi-separated adic formal $R$-schemes topologically of finite type over $\Spf(R)$, localized by admissible formal blow-ups, and \item the category of quasi-separated adic spaces of finite type over $S$.\end{enumerate}\end{thm}

\subsection{Outline of the proof of Theorem \ref{Main theorem}}

Our proof of Theorem~\ref{Main theorem} mostly follows the proof of Raynaud's classical theorem by Bosch and Lütkebohmert, see \cite{BL1}, proof of Theorem~4.1, replacing admissible formal blow-ups by normalized formal blow-ups. Thus, the proof consists in successively showing the following three statements: \begin{itemize}\item (Any morphism of formal models is determined by its generic fiber) Any two morphisms $f_{0}$, $g_{0}$ of $\varpi$-torsion-free, qcqs, locally stably uniform adic formal $R$-schemes with $f_{0\eta}=g_{0\eta}$ must be equal: This is Lemma~\ref{Faithfulness} and is proved using the specialization map from the generic fiber, as in the classical case.
\item (Formal models of morphisms) For any two adic formal $R$-schemes $\mathfrak{Z}$, $\mathfrak{X}$ as in the theorem and any morphism \begin{equation*}f: \mathfrak{Z}_{\eta}^{\ad}\to \mathfrak{X}_{\eta}^{\ad}\end{equation*}between their generic fibers, there exist compositions of normalized formal blow-ups $\mathfrak{Z}''\to \mathfrak{Z}'\to\mathfrak{Z}$, $\mathfrak{X}''\to \mathfrak{X}'\to \mathfrak{X}$ and a morphism $f_{0}: \mathfrak{Z}''\to\mathfrak{X}''$ which is a formal $R$-model of $f$ and is an isomorphism if $f$ is an isomorphism: This is Lemma~\ref{Fullness} and Lemma~\ref{Isomorphisms}; it is at this step that the use of (compositions of) normalized formal blow-ups instead of usual admissible formal blow-ups becomes necessary. \item (Existence of formal $R$-models for adic spaces) Every qcqs uniform adic space $X$ over $S$ has a formal $R$-model $\mathfrak{X}$, which is locally stably uniform and integrally closed in its generic fiber $X$: This appears as Theorem~\ref{Existence of formal models}. The proof of this existence statement is analogous to the classical one; it uses induction on the size of an affinoid open cover (the affinoid case being obvious) and a gluing argument which relies on Lemma~\ref{Isomorphisms}. We note that this part of the theorem also follows from \cite{AGV}, Corollary~1.2.7; however, op.~cit.~does not explicitly spell out the birational gluing argument. \end{itemize}

\subsection{Outline of the paper}

In Section~\ref{sec:formal schemes} we introduce the class of locally rig-sheafy formal schemes over $(R, \varpi)$ (in greater generality than in the introduction). In Section~\ref{sec:generic fiber}, we construct the (adic analytic) generic fiber functor and the specialization map from the generic fiber to the special fiber following classical constructions of Raynaud \cite{RaynaudReport} and Berthelot \cite{Berthelot96}. In Section \ref{sec:formal blow-ups} we collect some facts on admissible formal blow-ups, prove a global analog of a theorem of Bhatt, \cite{BhattNotes}, Theorem~8.1.2, which describes the adic analytic generic fiber as the inverse limit of admissible formal blow-ups and deduce some basic properties of the specialization map. Section~\ref{sec:normalized formal blow-ups} and Section~\ref{sec:main results} form the heart of the paper. In Section~\ref{sec:normalized formal blow-ups}, we establish Proposition-Definition~\ref{Normalization, introduction}, introduce the notion of normalized formal blow-ups and prove some useful properties of this class of morphisms. Finally, Section~\ref{sec:main results} contains the proofs of our main results, Theorem~\ref{Main theorem} and Theorem~\ref{Main theorem 2}. 

\subsection{Acknowledgements}

I would like to express my sincerest gratitude to my advisor, Kiran Kedlaya, for his advice, encouragement and support. I would like to thank Ryo Ishizuka for comments on a preliminary version of this paper and Jack J Garzella for helpful conversations. 

\subsection{AI disclosure}The results of this paper pre-date the AI-in-mathematics era. All results and proofs in the paper are the author's own. However, Claude Fable 5 was used to proofread the paper multiple times and point out some (mostly formal) mistakes, which were subsequently corrected by the author. 

\section{Locally rig-sheafy formal schemes over an adic ring}\label{sec:formal schemes}

Recall from \cite{FK}, Ch. I, Definition~1.1.16, that a formal scheme $\mathfrak{X}$ is said to be adic (respectively, adic of finite ideal type) if it admits an open cover by affine formal schemes of the form $\Spf(A)$ such that $A$ is a complete adic ring (respectively, a complete adic ring with finitely generated ideal of definition). The following special case of this notion is the most important to us. 
\begin{mydef}[Adic formal $R$-scheme]\label{Adic formal scheme}Let $R$ be a complete adic ring with finitely generated ideal of definition. An adic formal $R$-scheme is a formal scheme $\mathfrak{X}$ together with an adic morphism $\mathfrak{X}\to \Spf(R)$.\end{mydef}
If $R$ is a complete adic ring with finitely generated ideal of definition $I$, then $\mathfrak{X}\to \Spf(R)$ is an adic formal $R$-scheme if and only if $I\mathcal{O}_{\mathfrak{X}}$ is an \textit{ideal of definition} of $\mathfrak{X}$. 
\begin{mydef}[Ideal of definition: \cite{FK}, Ch.~I, Definition~1.1.18(2)]Let $\mathfrak{X}$ be an adic formal scheme. An ideal sheaf $\mathcal{I}\subseteq \mathcal{O}_{\mathfrak{X}}$ is called an ideal of definition of $\mathfrak{X}$ if there exists an affine open cover $(\mathfrak{U}_{i}=\Spf(A_{i}))_{i}$ of $\mathfrak{X}$ with each $A_{i}$ a complete adic ring such that $I_{i}=\mathcal{I}(\mathfrak{U}_{i})$ is an ideal of definition of $A_{i}$ for every $i$. \end{mydef}
By \cite{FK}, Ch.~I, Corollary~3.7.2, every adic formal scheme of finite ideal type which is quasi-compact and quasi-separated has an ideal of definition of finite type. 
\begin{notation}\label{Notation for formal schemes}Let $\mathfrak{X}$ be an adic formal scheme of finite ideal type and let $\mathcal{I}$ be a fixed ideal of definition of $\mathfrak{X}$ of finite type. Then for every integer $k\geq0$, we denote by $\mathfrak{X}_{k}$ the scheme \begin{equation*}\mathfrak{X}_{k}=(\vert \mathfrak{X}\vert, \mathcal{O}_{\mathfrak{X}}/\mathcal{I}^{k+1}\mathcal{O}_{\mathfrak{X}}).\end{equation*}\end{notation}
In particular, if we fix a complete adic ring $R$ and a finitely generated ideal of definition $I$ of $R$, then for every adic formal $R$-scheme $\mathfrak{X}$ we denote by $\mathfrak{X}_{k}$ the scheme $(\vert \mathfrak{X}\vert, \mathcal{O}_{\mathfrak{X}}/I^{k+1}\mathcal{O}_{\mathfrak{X}})$. If $\mathfrak{X}$ is an adic formal scheme of finite ideal type, we denote by $\Aff_{\mathfrak{X}}$ the family of all affine open subsets of $\mathfrak{X}$ of the form $\Spf(A)$ for $A$ a complete adic ring with a finitely generated ideal of definition. By \cite{FK}, Ch.~I, Corollary~3.7.13, $\Aff_{\mathfrak{X}}$ coincides with the family of all affine open subsets of $\mathfrak{X}$.
\begin{lemma}\label{Restriction maps are submetric}Let $\mathcal{I}$ be an ideal of definition of finite type on an adic formal scheme of finite ideal type $\mathfrak{X}$. For every inclusion $\mathfrak{V}\subseteq\mathfrak{U}$ with $\mathfrak{U}, \mathfrak{V}\in\Aff_{\mathfrak{X}}$ we have \begin{equation*}\mathcal{I}(\mathfrak{U})\mathcal{O}_{\mathfrak{X}}(\mathfrak{V})=\mathcal{I}(\mathfrak{V}).\end{equation*}\end{lemma}
\begin{proof}By \cite{FK}, Ch.~I, Proposition~1.1.20 and Proposition~1.1.22, the ideal of definition of finite type $\mathcal{I}\vert_{\mathfrak{U}}$ of $\mathfrak{U}$ is given by $\mathcal{I}\vert_{\mathfrak{U}}=\mathcal{I}(\mathfrak{U})\mathcal{O}_{\mathfrak{U}}$. The assertion follows.\end{proof}  
Recall from \cite{Huber2}, \S4, that a sufficient condition for an adic formal scheme to define an associated adic space is that the formal scheme be locally Noetherian. This yields a class of examples for the following class of formal schemes.
\begin{mydef}We call an adic formal scheme \textit{locally sheafy} if it is of finite ideal type and has an open cover by formal spectra of sheafy complete adic rings (with finitely generated ideals of definition).\end{mydef}
More recently, Zavyalov proved a sheafiness result which incorporates both the case of Noetherian complete adic rings and the case of adic rings topologically of finite type over a not necessarily discrete valuation ring of rank $1$.
\begin{mydef}[Fujiwara-Kato \cite{FK}, Ch.~0, Def.~8.4.3, and Ch.~I, Def.~2.1.7; see also Zavyalov \cite{Zavyalov22}, Definition 2.8]Let $A$ be a complete adic ring which has a finitely generated ideal of definition $I$. Then $A$ (or the pair $(A, I)$) is said to be topologically universally rigid-Noetherian (t.~u.~rigid-Noetherian) if, for every integer $n\geq 0$, the quasi-compact scheme \begin{equation*}\Spec(A\langle X_1,\dots, X_n\rangle)\setminus\mathcal{V}(IA\langle X_1,\dots,X_n\rangle)\end{equation*}is Noetherian. An adic formal scheme of finite ideal type $\mathfrak{X}$ is called locally universally rigid-Noetherian if it has an open cover $(\mathfrak{U}_{i})_{i}$ by affine formal schemes $\mathfrak{U}_{i}=\Spf(A_{i})$, where each $A_{i}$ is a t.~u.~rigid-Noetherian complete adic ring.\end{mydef}
\begin{mydef}[\cite{Zavyalov22}, Definition~2.8]\label{Strongly rigid-Noetherian}A complete Huber ring $A$ is called strongly rigid-Noetherian if it has a t.~u.~rigid-Noetherian pair of definition.\end{mydef} 
\begin{thm}[Zavyalov \cite{Zavyalov22}, Theorem~1.1]\label{Zavyalov's theorem}Every strongly rigid-Noetherian complete Huber pair $(A, A^{+})$ is sheafy. Furthermore, $H^{i}(U, \mathcal{O}_{X})=0$ for every rational subset $U$ of $\Spa(A, A^{+})$ and every $i\geq1$.\end{thm}    
In particular, every locally universally rigid-Noetherian adic formal scheme of finite ideal type is locally sheafy. 

The category of locally sheafy adic formal schemes embeds fully and faithfully into the category of adic spaces (the construction of \cite{Huber2}, \S4, extends verbatim to the case when $\mathfrak{X}$ is merely locally sheafy instead of locally Noetherian). For a locally sheafy formal scheme $\mathfrak{X}$ we denote by $\mathfrak{X}^{\ad}$ the associated adic space. We denote by \begin{equation*}\pi_{\mathfrak{X}}: \mathfrak{X}^{\ad}\to\mathfrak{X}\end{equation*}the canonical morphism of locally topologically ringed spaces. 
 
Our first order of business is to introduce a notion of analytic generic fiber for a reasonably broad class of adic formal schemes over $R$. Since we want the analytic generic fiber to be an adic space, the relevant class of formal schemes is specified by the following definition.
\begin{mydef}[Rig-sheafy complete adic ring]\label{Locally rig-sheafy formal schemes}Let $R$ be a complete adic ring which has an ideal of definition generated by a single element $\varpi$. We call a complete adic ring $A$ with a continuous (but not necessarily adic) ring map $R\to A$ \textit{rig-sheafy} over the pair $(R, \varpi)$ (or rig-sheafy over $R$, or just rig-sheafy, if $\varpi$ or $(R, \varpi)$ is understood from the context) if it has a finitely generated ideal of definition $I=(f_1,\dots, f_r)_{A}$ containing (the image of) $\varpi$ such that the Huber rings $B_{n,A}[\varpi^{-1}]$ with \begin{equation*}B_{n,A}=A\langle\frac{f_1^{n},\dots, f_r^{n}}{\varpi}\rangle,\end{equation*}are sheafy for all $n\geq1$. Here $B_{n, A}=A\langle\frac{f_1^{n},\dots,f_r^{n}}{\varpi}\rangle$ is the completion, for the $IA[\frac{f_{1}^{n},\dots,f_{r}^{n}}{\varpi}]$-adic topology, of the subring \begin{equation*}A[\frac{f_1^{n},\dots,f_r^{n}}{\varpi}]\end{equation*}of the localization $A[\varpi^{-1}]$.\end{mydef}
In the above definition we view the zero ring as a sheafy Huber ring (giving rise to the empty adic space), so an $R$-algebra $A$ which is $\varpi$-torsion is necessarily rig-sheafy. Note that if $A$ is not $\varpi$-torsion-free, the canonical map $A\to A[\frac{f_1^{n},\dots, f_r^{n}}{\varpi}]$ is not injective, for any $f_1,\dots, f_r$ and $n$. By continuity, the condition that the image of $\varpi$ is contained in $I$ can always be assumed to hold up to replacing $\varpi$ with a power of itself. However, the following lemma shows that there always exists a finitely generated ideal of definition of $A$ which contains $\varpi$ itself, not merely a power of $\varpi$.
\begin{lemma}\label{Ideals containing the pseudo-uniformizer}Let $R\to A$ be a continuous homomorphism of complete adic rings, where $R$ has a principal ideal of definition and $A$ has a finitely generated ideal of definition. Then for every $\varpi$ generating an ideal of definition of $R$ and for every finitely generated ideal of definition $I$ of $A$, the ideal $J=(I, \varpi)$ is an ideal of definition of $A$.\end{lemma}
\begin{proof}Since $J$ contains $I$ by definition, we only have to prove that $I$ contains some power of $J$. By continuity of $R\to A$, there exists an integer $n\geq1$ such that $\varpi^{n}\in I$. Suppose that $I$ is generated by $r$ elements $f_1,\dots, f_r$. Then $J^{(r+1)n}\subseteq (f_1^{n},\dots, f_r^{n}, \varpi^{n})_{A}\subseteq I$.\end{proof}
The following elementary lemma provides a link between Definition \ref{Locally rig-sheafy formal schemes} and birational geometry.
\begin{lemma}\label{Affine blow-up algebras}Let $f_1,\dots, f_r, g\in A$ be elements of a ring $A$. The subring $A[\frac{f_1,\dots,f_r}{g}]$ of $A[g^{-1}]$ coincides with the affine blow-up algebra $A[\frac{(f_1,\dots, f_r)_{A}}{g}]$, cf. \cite{Stacks}, Tag~052P.\end{lemma}
\begin{proof}Both rings in question are subrings of $A[g^{-1}]$. Let $f\in A[\frac{f_1,\dots, f_r}{g}]$. Then we can write $f$ as a sum of the form \begin{equation*}f=\sum_{(i_1,\dots, i_r)}a_{i_{1},\dots, i_{r}}(\frac{f_1}{g})^{i_{1}}\cdot~\dots~\cdot (\frac{f_r}{g})^{i_{r}}\end{equation*}for some coefficients $a_{i_1,\dots, i_r}\in A$. By definition of the affine blow-up algebra $A[\frac{(f_1,\dots, f_r)_{A}}{g}]$, each $(\frac{f_{j}}{g})^{i_{j}}$ belongs to $A[\frac{(f_1,\dots, f_r)_{A}}{g}]$, since $f_{j}^{i_{j}}\in (f_1,\dots, f_r)_{A}^{i_{j}}$. Since $A[\frac{(f_1,\dots, f_r)_{A}}{g}]$ is an $A$-algebra, $f\in A[\frac{(f_1,\dots, f_r)_{A}}{g}]$. This shows that \begin{equation*}A[\frac{f_1,\dots, f_r}{g}]\subseteq A[\frac{(f_1,\dots, f_r)_{A}}{g}].\end{equation*}

Conversely, let \begin{equation*}\frac{h}{g^{m}}\in A[\frac{(f_1,\dots, f_r)_{A}}{g}]\end{equation*}be an arbitrary element of the affine blow-up algebra, with $m\in\mathbb{Z}_{\geq0}$ and $h\in (f_1,\dots, f_r)_{A}^{m}$. Let $h_{jk}=\sum_{i=1}^{r}a_{ijk}f_{i}$ ($j=1,\dots, m$) be elements of $(f_1,\dots, f_r)_{A}$ with \begin{equation*}h=\sum_{k}h_{1k}\cdot~\dots~\cdot h_{mk}.\end{equation*}Then we can write $h$ as a finite sum \begin{equation*}h=\sum_{k}\sum_{m_1+\dots+m_{r}=m}b_{m_1,\dots, m_r, k}f_{1}^{m_{1}}\cdot\dots\cdot f_r^{m_{r}}\end{equation*}with coefficients $b_{m_1,\dots, m_r, k}\in A$. Consequently, we can write $\frac{h}{g^{m}}$ as \begin{equation*}\frac{h}{g^{m}}=\sum_{k}\sum_{m_1+\dots+m_r=m}b_{m_1,\dots,m_r, k}(\frac{f_1}{g})^{m_1}\cdot\dots\cdot (\frac{f_r}{g})^{m_r}.\end{equation*}But $\frac{f_{i}^{m_{i}}}{g^{m_{i}}}\in A[\frac{f_1,\dots, f_r}{g}]$ for all $i=1,\dots, r$, so $\frac{h}{g^{m}}\in A[\frac{f_1,\dots, f_r}{g}]$. \end{proof}
To put the definition of rig-sheafy complete adic rings over $(R, \varpi)$ into context, we also need the following simple lemma.
\begin{lemma}\label{Completions are torsion-free}Let $A$ be a ring and let $\varpi$ be an element of $A$. Then the $\varpi$-adic completion of the $\varpi$-torsion-free quotient $A/A[\varpi^{\infty}]$ is given by $\widehat{A}/\widehat{A}[\varpi^{\infty}]$.\end{lemma}
\begin{proof}The quotient $A/A[\varpi^{\infty}]$ is the image of the canonical map $A\to A[\varpi^{-1}]$. The map $A\to A[\varpi^{-1}]$ is strict when $A$ is endowed with the $\varpi$-adic seminorm and $A[\varpi^{-1}]$ is endowed with the canonical extension of the $\varpi$-adic seminorm on the image of $A$ in $A[\varpi^{-1}]$, by definition. Hence the $\varpi$-adic completion of the image $A/A[\varpi^{\infty}]$ of the map $A\to A[\varpi^{-1}]$ is equal to the image of the canonical map $\widehat{A}\to A[\varpi^{-1}]^{\wedge}=\widehat{A}[\varpi^{-1}]$, by \cite{BGR}, Proposition~1.1.9/5. \end{proof}      
\begin{rmk}\label{Remark on rig-sheafiness}Several remarks are in order with regard to Definition~\ref{Locally rig-sheafy formal schemes}. Firstly, note that the complete adic ring $A\langle\frac{f_{1}^{n},\dots,f_{r}^{n}}{\varpi}\rangle$ is not the same as the identically denoted rational localization of $A$ since $\varpi$ is not invertible in $A\langle\frac{f_{1}^{n},\dots,f_{r}^{n}}{\varpi}\rangle$. In fact, we have opted to denote the Huber ring of sections $\mathcal{O}_{\Spa(A, A)}(\Spa(A, A)(\frac{f_1^{n},\dots, f_r^{n}}{\varpi}))$ of the rational subset \begin{equation*}\Spa(A, A)(\frac{f_1^{n},\dots,f_r^{n}}{\varpi})=\{\, x\in \Spa(A, A)\mid \vert f_{i}^{n}(x)\vert\leq \vert\varpi(x)\vert\neq 0 \ ~\text{for}~i=1,\dots, r\,\}\end{equation*}of the affinoid pre-adic space $\Spa(A, A)$ by $B_{n,A}[\varpi^{-1}]=A\langle\frac{f_1^{n},\dots, f_r^{n}}{\varpi}\rangle[\varpi^{-1}]$ to distinguish it from its ring of definition $B_{n,A}=A\langle\frac{f_{1}^{n},\dots,f_{r}^{n}}{\varpi}\rangle$. 

Secondly, by Lemma~\ref{Affine blow-up algebras}, the ring $A[\frac{f_1^{n},\dots, f_r^{n}}{\varpi}]$ is equal to the affine blow-up algebra $A[\frac{(f_1^{n},\dots,f_r^{n})_{A}}{\varpi}]$ (with notation as in \cite{Stacks}, Tag~052P). It is readily seen that this ring is $\varpi$-torsion-free, satisfies\begin{equation*}(f_1^{n},\dots, f_r^{n})_{A}A[\frac{f_1^{n},\dots, f_r^{n}}{\varpi}]=\varpi A[\frac{f_1^{n},\dots, f_r^{n}}{\varpi}]\end{equation*}and is characterized by the following universal property: It is the initial object in the category of $\varpi$-torsion-free $A$-algebras $B$ which satisfy $(f_{1}^{n},\dots, f_{r}^{n})_{A}B\subseteq\varpi B$. In particular, we see that \begin{equation*}(f_{1}^{n},\dots, f_{r}^{n})_{A}A\langle\frac{f_1^{n},\dots, f_r^{n}}{\varpi}\rangle=\varpi A\langle\frac{f_1^{n},\dots,f_r^{n}}{\varpi}\rangle\end{equation*}and thus the Huber rings required to be sheafy in the above definition are actually Tate rings, for all $n\geq1$. 

By Lemma~\ref{Completions are torsion-free}, the completed affine blow-up algebra $A\langle\frac{f_1^{n},\dots, f_r^{n}}{\varpi}\rangle$ is $\varpi$-torsion-free. We also observe that $A\langle\frac{f_1^{n},\dots,f_r^{n}}{\varpi}\rangle$ satisfies a universal property in the category of linearly topologized complete topological $A$-algebras, analogous to the universal property of affine blow-up algebras: It is the initial object in the category of $\varpi$-torsion-free linearly topologized complete topological $A$-algebras $B$ satisfying $(f_{1}^{n},\dots, f_{r}^{n})_{A}B\subseteq\varpi B$. To see this, let $B$ be a complete topological $A$-algebra as above; by the universal property of $A[\frac{f_1^{n},\dots, f_r^{n}}{\varpi}]$ the map $A\to B$ factors uniquely through the canonical map $A\to A[\frac{f_1^{n},\dots,f_r^{n}}{\varpi}]$. Since the map $A\to B$ is continuous, every open ideal in the linearly topologized ring $B$ contains some power of $(f_1,\dots, f_r)_{A}B$ and hence also some power of $\varpi$. Since $\varpi$ generates an ideal of definition of $A[\frac{f_1^{n},\dots,f_r^{n}}{\varpi}]$, this entails that the map $A[\frac{f_1^{n},\dots,f_r^{n}}{\varpi}]\to B$ is continuous and thus factors through a continuous ring map $A\langle\frac{f_1^{n},\dots,f_r^{n}}{\varpi}\rangle\to B$, as desired.     

Finally, recall that the affine blow-up algebra $A[\frac{f_1^{n},\dots,f_r^{n}}{\varpi}]$ is canonically isomorphic to the quotient of \begin{equation*}A[T_1,\dots,T_r]/(f_1^{n}-\varpi T_1,\dots,f_r^{n}-\varpi T_r)_{A[T_1,\dots,T_r]}\end{equation*}by the ideal of $\varpi$-power-torsion elements (\cite{Stacks}, Tag~0G8S). Consequently, by Lemma~\ref{Completions are torsion-free}, $A\langle\frac{f_1^{n},\dots,f_r^{n}}{\varpi}\rangle$ is the quotient of \begin{equation*}A\langle T_1,\dots,T_r\rangle/\overline{(f_1^{n}-\varpi T_1,\dots, f_r^{n}-\varpi T_r)}_{A\langle T_1,\dots,T_r\rangle}\end{equation*}by the ideal of $\varpi$-power-torsion elements, where the bar in the above equation denotes the topological closure of an ideal in its ambient topological ring. This shows that our rings $B_{n,A}$ coincide up to $\varpi$-power-torsion with the rings $B_{n}$ used by Berthelot (\cite{Berthelot96}, (0.2.6)) in his construction of the generic fiber of a formal scheme locally formally of finite type over a discrete valuation ring (note that in Berthelot's setting $A$ is a Noetherian adic ring, so all ideals in $A\langle T_1,\dots, T_r\rangle$ are closed).\end{rmk}
Let us now prove that the definition of rig-sheafy complete adic rings is independent of any choices.        
\begin{lemma}\label{Rig-sheafiness does not depend on generators}Let $R$ and $\varpi$ be as in Definition \ref{Locally rig-sheafy formal schemes}. The property of a complete adic ring $A$ with a continuous ring map $R\to A$ being rig-sheafy over $(R, \varpi)$ does not depend on the choice of the generators $f_1,\dots, f_r$ nor on the choice of the finitely generated ideal of definition $I$ of $A$ containing the image of $\varpi$. In particular, if $R\to A$ is adic (i.e., if $\varpi$ generates an ideal of definition in $A$), then $A$ is rig-sheafy if and only if the Tate ring $A[\varpi^{-1}]$ is sheafy.\end{lemma}
\begin{proof}Let $I$ be a finitely generated ideal of definition of $A$ containing $\varpi$ and let $f_1,\dots, f_r$ be generators of $I$ such that $A\langle\frac{f_1^{n},\dots,f_r^{n}}{\varpi}\rangle[\varpi^{-1}]$ is a sheafy Tate ring for all $n\geq1$. Let $J$ be another finitely generated ideal of definition, with generators $g_1,\dots, g_s$. We want to prove that the Tate ring \begin{equation*}A\langle\frac{g_1^{n},\dots,g_s^{n}}{\varpi}\rangle[\varpi^{-1}]\end{equation*}is sheafy for all $n\geq1$. Choose an integer $m\geq1$ such that $I^{m}\subseteq J$. In particular, $f_{i}^{m}\in J=(g_1,\dots, g_s)_{A}$ for all $i=1,\dots,r$ and thus\begin{equation*}f_{i}^{msn}\in J^{sn}\subseteq (g_1^{n},\dots, g_s^{n})_{A}\end{equation*}for all $i=1,\dots,r$. It follows that $\frac{f_{i}^{msn}}{\varpi}$ is a power-bounded element of $A\langle\frac{g_1^{n},\dots,g_s^{n}}{\varpi}\rangle[\varpi^{-1}]$ for every $i=1,\dots, r$ and every $n\geq1$. By the universal property of rational localizations (\cite{Huber2}, (1.2)), the canonical map $A\to A\langle\frac{g_1^{n},\dots,g_s^{n}}{\varpi}\rangle[\varpi^{-1}]$ (for every $n\geq1$) factors through the analogous canonical map $A\to A\langle\frac{f_1^{msn},\dots,f_r^{msn}}{\varpi}\rangle[\varpi^{-1}]$. It follows that the Tate ring $A\langle\frac{g_1^{n},\dots,g_s^{n}}{\varpi}\rangle[\varpi^{-1}]$ is a rational localization of the sheafy complete Tate ring $A\langle\frac{f_1^{msn},\dots,f_r^{msn}}{\varpi}\rangle[\varpi^{-1}]$ and thus is itself sheafy.\end{proof}
\begin{lemma}\label{Rig-sheafiness does not depend on the pseudo-uniformizer}Let $R, \varpi$ be as in Definition~\ref{Locally rig-sheafy formal schemes} and let $R\to A$ be a continuous homomorphism of complete adic rings, where $A$ has a finitely generated ideal of definition. Then the property of $A$ being rig-sheafy over $(R, \varpi)$ depends only on the ideal of definition $(\varpi)_{R}$ of $R$ generated by $\varpi$ and not on the choice of a generator $\varpi$ of that ideal of definition.\end{lemma}
\begin{proof}Let $\varpi'$ be another generator of $(\varpi)_{R}$ in $R$. By Lemma~\ref{Ideals containing the pseudo-uniformizer}, there exists an ideal of definition $I=(f_1,\dots,f_r)_{A}$ of $A$ containing both $\varpi$ and $\varpi'$. Using the universal property from Remark~\ref{Remark on rig-sheafiness}, we prove that $A\langle\frac{f_1^{n},\dots, f_r^{n}}{\varpi'}\rangle=A\langle\frac{f_1^{n},\dots, f_r^{n}}{\varpi}\rangle$ for all $n\geq1$, whence the assertion of the lemma follows. First observe that the submodule of $\varpi$-power-torsion elements of any $R$-module $M$ is equal to the submodule of $\varpi'$-power-torsion elements of $M$, since both are equal to the submodule of $(\varpi)_{R}$-power-torsion elements. In particular, $A\langle\frac{f_1^{n},\dots, f_r^{n}}{\varpi'}\rangle$ is a $\varpi$-torsion-free complete topological $A$-algebra satisfying \begin{equation*}(f_1^{n},\dots, f_r^{n})_{A}A\langle\frac{f_1^{n},\dots,f_r^{n}}{\varpi'}\rangle=\varpi'A\langle\frac{f_1^{n},\dots, f_r^{n}}{\varpi'}\rangle=\varpi A\langle\frac{f_1^{n},\dots, f_r^{n}}{\varpi'}\rangle,\end{equation*}so the map $A\to A\langle\frac{f_1^{n},\dots, f_r^{n}}{\varpi'}\rangle$ factors canonically through the map $A\to A\langle\frac{f_1^{n},\dots,f_r^{n}}{\varpi}\rangle$, by the universal property described in Remark~\ref{Remark on rig-sheafiness}. By exchanging the roles of $\varpi$ and $\varpi'$, we see that $A\langle\frac{f_1^{n},\dots,f_{r}^{n}}{\varpi'}\rangle=A\langle\frac{f_1^{n},\dots,f_{r}^{n}}{\varpi}\rangle$, for all $n\geq1$.\end{proof}
Lemma~\ref{Rig-sheafiness does not depend on generators} allows us to globalize the notion of a rig-sheafy complete adic ring over $(R, \varpi)$ as follows.
\begin{mydef}[Locally rig-sheafy adic formal scheme over $(R, \varpi)$]\label{Locally rig-sheafy formal schemes 2}Let $R$ and $\varpi$ be as in Definition~\ref{Locally rig-sheafy formal schemes}. We call an affine adic formal scheme $\mathfrak{X}=\Spf(A)$ over $\Spf(R)$ rig-sheafy over $(R, \varpi)$ (or rig-sheafy over $R$, or just rig-sheafy, if $\varpi$ or $(R, \varpi)$ is understood from the context) if $A$ is a rig-sheafy complete adic ring over $(R, \varpi)$. We call a general adic formal scheme $\mathfrak{X}$ over $\Spf(R)$ \textit{locally rig-sheafy over $(R, \varpi)$} (or just locally rig-sheafy, if $(R, \varpi)$ is understood from the context) if it has an open cover by rig-sheafy affine formal subschemes.\end{mydef} 
\begin{lemma}\label{Locally sheafy implies locally rig-sheafy}If $\mathfrak{X}$ is a locally sheafy adic formal scheme over $R$, then $\mathfrak{X}$ is also locally rig-sheafy over $(R, \varpi)$, for every choice of an element $\varpi\in R$ generating an ideal of definition of $R$.\end{lemma}
\begin{proof}It suffices to prove that a sheafy complete adic ring $A$ (with finitely generated ideal of definition) is rig-sheafy over $(R, \varpi)$ for every continuous ring map $R\to A$, where the pair $(R, \varpi)$ is as in Definition~\ref{Locally rig-sheafy formal schemes}. Let $A$ be a sheafy complete adic ring over $R$ with finitely generated ideal of definition $I=(f_1,\dots, f_r)_{A}$ containing $\varpi$ (by Lemma~\ref{Ideals containing the pseudo-uniformizer}, such an ideal of definition always exists). The Tate rings $A\langle \frac{f_1^{n},\dots, f_r^{n}}{\varpi}\rangle[\varpi^{-1}]$ are the rings of sections of the rational subsets\begin{equation*}U_{n}=\{\, x\in \Spa(A, A)\mid \vert f_{i}^{n}(x)\vert\leq \vert \varpi(x)\vert\neq 0\,\}.\end{equation*}But if the pre-adic space $\Spa(A, A)$ is an adic space, so are its rational subsets $U_{n}$. It follows that each $A\langle\frac{f_1^{n},\dots, f_r^{n}}{\varpi}\rangle[\varpi^{-1}]$ is sheafy, as claimed.\end{proof}
\begin{lemma}\label{Topology on a locally rig-sheafy formal scheme}Let $R, \varpi$ be as before and let $\mathfrak{X}$ be a locally rig-sheafy adic formal scheme over $(R, \varpi)$. The family of all rig-sheafy affine open formal subschemes of $\mathfrak{X}$ is a basis for the topology of $\mathfrak{X}$.\end{lemma}
\begin{proof}It suffices to prove that for every rig-sheafy complete adic ring $A$ over $(R, \varpi)$ the complete adic ring $A\langle g^{-1}\rangle$ is again rig-sheafy over $(R, \varpi)$ for every $g\in A$. But, choosing a finitely generated ideal of definition $I$ of $A$ containing $\varpi$ and generators $f_1,\dots, f_r$ of $I$, this follows from the equality \begin{equation*}A\langle g^{-1}\rangle\langle\frac{f_1^n,\dots, f_r^{n}}{\varpi}\rangle[\varpi^{-1}]=A\langle\frac{f_1^n,\dots, f_r^{n}}{\varpi}\rangle\langle g^{-1}\rangle[\varpi^{-1}]\end{equation*}which can be shown for all $n\geq1$ using the universal property of rational localizations.\end{proof} 
\begin{rmk}\label{Formal schemes as locally $v$-ringed spaces}In the following we endow each formal scheme $\mathfrak{X}$ with the structure of a locally v-ringed space by equipping each stalk $\mathcal{O}_{\mathfrak{X},x}$, $x\in\mathfrak{X}$, with the valuation \begin{equation*}v_{x}: \mathcal{O}_{\mathfrak{X},x}\to \{0,1\}\end{equation*}which takes the value $0$ on the maximal ideal and takes the value $1$ elsewhere. With this structure, every morphism of formal schemes $f: \mathfrak{X}\to\mathfrak{S}$ becomes a morphism of locally v-ringed spaces since the homomorphisms $\mathcal{O}_{\mathfrak{S},f(x)}\to\mathcal{O}_{\mathfrak{X},x}$ are local ring homomorphisms. This exhibits the category of formal schemes as a full subcategory of the category of locally v-ringed spaces.\end{rmk}

\section{Generic fiber and specialization map}\label{sec:generic fiber}

Following the classical idea of Berthelot (\cite{Berthelot96}, \S0.2.6), we can now define the adic analytic generic fiber $\mathfrak{X}_{\eta}^{\ad}$ of a locally rig-sheafy adic formal scheme over $R$. We begin with the case when $\mathfrak{X}$ is affine and rig-sheafy. For $R$ and $\varpi$ as before and for $\mathfrak{X}=\Spf(A)$ a rig-sheafy affine formal scheme over $(R, \varpi)$, where $A$ is a complete rig-sheafy adic ring with finitely generated ideal of definition $I=(f_1,\dots, f_r)_{A}$ containing $\varpi$ (we use Lemma~\ref{Ideals containing the pseudo-uniformizer} to ensure that such an ideal of definition exists), we again consider the completed affine blow-up algebras\begin{equation*}B_{n,A}=A\langle \frac{f_{1}^{n},\dots, f_{r}^{n}}{\varpi}\rangle, \ n\geq 1.\end{equation*}Recall from Remark~\ref{Remark on rig-sheafiness} that these satisfy \begin{equation*}(f_1^{n},\dots,f_r^{n})_{A}B_{n,A}=\varpi B_{n,A}.\end{equation*}This shows that the topology on $B_{n,A}$ is the $\varpi$-adic topology, for all $n\geq1$, so each $\Spf(B_{n,A})$ is an adic formal $R$-scheme. We can describe the Tate ring $B_{n,A}[\varpi^{-1}]$ as \begin{equation*}B_{n,A}[\varpi^{-1}]=\mathcal{O}_{\Spa(A, A)}(U_{n})\end{equation*}where $U_{n}\subseteq \Spa(A, A)$ is the rational subset \begin{equation*}U_{n}=\{\, x\in \Spa(A, A) \mid \vert f_{i}^{n}(x)\vert\leq \vert\varpi(x)\vert\neq 0 \ \text{for all}\ i=1,\dots, r\,\}.\end{equation*}For $m\geq n$, we have \begin{equation*}(f_1^{m},\dots, f_r^{m})_{A}B_{n,A}\subseteq (f_1^{n},\dots, f_r^{n})_{A}B_{n,A}=\varpi B_{n,A},\end{equation*}so the canonical map $A\to B_{n,A}$ factors uniquely through a continuous $A$-algebra map $B_{m,A}\to B_{n,A}$, by the universal property of $B_{m,A}$ discussed in Remark \ref{Remark on rig-sheafiness}. Thus, for $m\geq n$, we have natural transition morphisms $B_{m,A}\to B_{n,A}$ and it is readily seen that these transition morphisms induce the natural inclusions of rational subsets $U_{n}\subseteq U_{m}$. Since $A$ is rig-sheafy, the pre-adic spaces \begin{equation*}U_{n}=\Spa(B_{n,A}[\varpi^{-1}], \overline{B_{n,A}})\end{equation*}for $n\geq1$ are adic spaces and thus the inclusion morphisms\begin{equation*}U_{n}\hookrightarrow U_{m}\end{equation*}for $m\geq n$ are open immersions of adic spaces. We define the adic analytic generic fiber $\mathfrak{X}_{\eta}^{\ad}=\mathfrak{X}_{\eta, (R, \varpi)}^{\ad}$ of $\mathfrak{X}$ over $(R, \varpi)$ to be the increasing union \begin{equation*}\mathfrak{X}_{\eta}^{\ad}:=\bigcup_{n>0}U_{n}=\bigcup_{n>0}\Spa(B_{n,A}[\varpi^{-1}], \overline{B_{n,A}}).\end{equation*}
\begin{lemma}\label{Ideals of definition and generic fiber}For any rig-sheafy affine adic formal scheme $\mathfrak{X}=\Spf(A)$ over $(R, \varpi)$, the adic space $\mathfrak{X}_{\eta}^{\ad}$ does not depend on the choice of a finitely generated ideal of definition $I$ of $A$ containing $\varpi$ and of its generators $f_1,\dots, f_r$.\end{lemma}
\begin{proof}If $J$ is another finitely generated ideal of definition of $A$, with generators $g_1,\dots, g_s$, then there exist positive integers $m, l$ such that $I^{m}\subseteq J$ and $J^{l}\subseteq I$. Then for every $n\geq1$ we have $f_{i}^{msn}\in J^{sn}\subseteq (g_1^{n},\dots, g_s^{n})_{A}$ for all $i=1,\dots, r$ and $g_{j}^{lrn}\in I^{rn}\subseteq (f_1^{n},\dots, f_r^{n})_{A}$. Consequently, \begin{align*}\{\, x\in \Spa(A, A) \mid \vert f_{i}^{msn}(x)\vert\leq \vert\varpi(x)\vert\neq 0,  i=1,\dots, r\,\} \\ \supseteq \{\, x\in \Spa(A, A) \mid \vert g_{j}^{n}(x)\vert\leq \vert\varpi(x)\vert\neq 0, j=1,\dots, s\,\}\end{align*}and \begin{align*}\{\, x\in \Spa(A, A) \mid \vert g_{j}^{lrn}(x)\vert\leq \vert\varpi(x)\vert\neq 0, j=1,\dots, s\,\} \\ \supseteq \{\, x\in \Spa(A, A) \mid \vert f_{i}^{n}(x)\vert\leq \vert\varpi(x)\vert\neq 0, i=1,\dots, r\,\},\end{align*}whence the assertion.\end{proof} 
Observe that every adic morphism of rig-sheafy affine adic formal schemes $f_{0}: \mathfrak{X}\to \mathfrak{S}$ over $(R, \varpi)$ corresponding to an adic ring homomorphism \begin{equation*}\varphi: \mathcal{O}_{\mathfrak{S}}(\mathfrak{S})\to\mathcal{O}_{\mathfrak{X}}(\mathfrak{X})\end{equation*}gives rise to morphisms\begin{equation*}\Spf(\mathcal{O}_{\mathfrak{X}}(\mathfrak{X})\langle\frac{\varphi(f_1)^{n},\dots,\varphi(f_r)^{n}}{\varpi}\rangle)\to \Spf(\mathcal{O}_{\mathfrak{S}}(\mathfrak{S})\langle\frac{f_1^{n},\dots,f_r^{n}}{\varpi}\rangle)\end{equation*}for any finitely generated ideal of definition $I=(f_1,\dots,f_r)_{\mathcal{O}_{\mathfrak{S}}(\mathfrak{S})}$ of $\mathcal{O}_{\mathfrak{S}}(\mathfrak{S})$. Since $\varphi$ is an adic homomorphism, $\varphi(I)\mathcal{O}_{\mathfrak{X}}(\mathfrak{X})$ is a finitely generated ideal of definition of $\mathcal{O}_{\mathfrak{X}}(\mathfrak{X})$. The homomorphisms $\mathcal{O}_{\mathfrak{S}}(\mathfrak{S})\langle\frac{f_1^{n},\dots,f_r^{n}}{\varpi}\rangle \to \mathcal{O}_{\mathfrak{X}}(\mathfrak{X})\langle\frac{\varphi(f_1)^{n},\dots,\varphi(f_r)^{n}}{\varpi}\rangle$ are compatible for varying $n\geq1$, so we obtain a morphism of adic spaces $\mathfrak{X}_{\eta}^{\ad}\to \mathfrak{S}_{\eta}^{\ad}$. In this way we see that the assignment $\mathfrak{X}\mapsto \mathfrak{X}_{\eta, (R, \varpi)}^{\ad}$ is a functor from the category of rig-sheafy affine adic formal schemes over $\Spf(R)$ and adic morphisms to the category of analytic adic spaces.
\begin{example}[Open unit disk and perfectoid open unit disk]\label{Open unit disk}One of the most basic examples of the generic fiber construction is the open unit disk over a nonarchimedean field $K$ which is given as the adic analytic generic fiber $\mathfrak{X}_{\eta}^{\ad}$ over $\Spa(K, K^{\circ})$ of \begin{equation*}\mathfrak{X}=\Spf(K^{\circ}[[T]]),\end{equation*}where the formal power series ring $K^{\circ}[[T]]$ is equipped with the $(\varpi, T)$-adic topology for a pseudo-uniformizer $\varpi$ of $K$. For an example which does not belong to the realm of classical rigid geometry, consider the perfectoid open unit disk over a perfectoid field $K$, which can be described as the generic fiber (in the above sense) over $\Spa(K, K^{\circ})$ of the rig-sheafy affine formal scheme $\Spf(K^{\circ}[[T^{1/p^{\infty}}]])$. Here, the affine formal scheme is rig-sheafy since for every $n\geq1$ the Tate ring $K^{\circ}[[T^{1/p^{\infty}}]]\langle\frac{T^{n}}{\varpi}\rangle[\varpi^{-1}]$ can be identified with the rational localization $K\langle T^{1/p^{\infty}}\rangle\langle\frac{T^{n}}{\varpi}\rangle$ of the perfectoid Tate algebra $K\langle T^{1/p^{\infty}}\rangle$: Indeed, $\Spa(K^{\circ}[[T^{1/p^{\infty}}]], K^{\circ}[[T^{1/p^{\infty}}]])$ is the subspace of $\Spa(K^{\circ}\langle T^{1/p^{\infty}}\rangle, K^{\circ}\langle T^{1/p^{\infty}}\rangle)$ on which $T$ is topologically nilpotent, so the map \begin{equation*}K^{\circ}\langle T^{1/p^{\infty}}\rangle\to K^{\circ}\langle T^{1/p^{\infty}}\rangle\langle\frac{T^{n}}{\varpi}\rangle\end{equation*}factors through $K^{\circ}[[T^{1/p^{\infty}}]]$ and thus also through $K^{\circ}[[T^{1/p^{\infty}}]]\langle\frac{T^{n}}{\varpi}\rangle$, by the universal property of completed affine blow-up algebras discussed in Remark~\ref{Remark on rig-sheafiness} applied to $A=K^{\circ}[[T^{1/p^{\infty}}]]$.\end{example}
\begin{example}[Fargues-Fontaine curves]\label{Fargues-Fontaine curves}A somewhat less standard example of the generic fiber construction can be described as follows. Let $(R, R^{+})$ be a perfectoid Tate Huber pair. Let $\varpi\in R^{+}$ be a pseudo-uniformizer of $R$ contained in $R^{+}$ such that $\varpi^{p}$ divides $p$ in $R^{+}$. By \cite{BMS}, Lemma~3.8, we may assume that $\varpi$ admits a compatible system of $p$-power roots in $R^{+}$, up to multiplying $\varpi$ by a unit in $R^{+}$. Let $\varpi^{\flat}$ be an element of the tilt \begin{equation*}R^{\flat+}=R^{+\flat}=\varprojlim_{f\mapsto f^{p}}R^{+}\end{equation*}of $R^{+}$ which corresponds to such a compatible system of $p$-power roots. Denote by $\mathbb{A}_{\inf}(R^{+})$ the ring of Witt vectors $W(R^{\flat+})$ considered as a complete adic ring with the $([\varpi^{\flat}], p)$-adic topology. By \cite{Kedlaya17}, Remark~3.1.7, $\mathbb{A}_{\inf}(R^{+})$ is a sheafy Huber ring. In particular, it is rig-sheafy over the pair $(W(R^{\flat+}), p[\varpi^{\flat}])$ (note that $W(R^{\flat+})$ is also $p[\varpi^{\flat}]$-adically complete by \cite{Stacks}, Tag~090T). It is then readily seen, either by inspection or by using Lemma~\ref{Generic fibers of sheafy formal schemes} below, that the adic analytic generic fiber of $\Spf(\mathbb{A}_{\inf}(R^{+}))$ over $(W(R^{\flat+}), p[\varpi^{\flat}])$ is equal to the analytic adic space \begin{equation*}Y_{\Spa(R, R^+)}=\Spa(\mathbb{A}_{\inf}(R^{+}), \mathbb{A}_{\inf}(R^{+}))\setminus\{\, p[\varpi^{\flat}]=0\,\}\end{equation*}whose quotient $X_{\Spa(R, R^+)}=Y_{\Spa(R, R^+)}/\varphi^{\mathbb{Z}}$ by the Frobenius action is the (relative) Fargues-Fontaine curve over $\Spa(R, R^{+})$. \end{example}
For any $(R, \varpi)$ and $\mathfrak{X}$, the analytic adic space $\mathfrak{X}_{\eta}^{\ad}=\mathfrak{X}_{\eta, (R, \varpi)}^{\ad}$ also comes equipped with a specialization map $\mathfrak{X}_{\eta}^{\ad}\to\mathfrak{X}$ defined as follows.
\begin{mydef}[Specialization map in the affine case]Let $\mathfrak{X}=\Spf(A)$ be a rig-sheafy affine adic formal scheme over $(R, \varpi)$ with adic analytic generic fiber $X=\mathfrak{X}_{\eta}^{\ad}=\mathfrak{X}_{\eta, (R, \varpi)}^{\ad}$ over $(R, \varpi)$. We define a map \begin{equation*}\spc_{X,\mathfrak{X}}: \vert X\vert\to \vert\mathfrak{X}\vert\end{equation*}by viewing every $x\in X$ as a continuous valuation on \begin{equation*}\mathcal{O}_{X}(U_{n})=B_{n,A}[\varpi^{-1}]\end{equation*}for some $n\geq1$ (and, by restriction, as a continuous valuation on $A$) and setting \begin{equation*}\spc_{X,\mathfrak{X}}(x)=\{\, f\in A\mid \vert f(x)\vert<1\,\}.\end{equation*}\end{mydef}Equivalently, the map $\spc_{X,\mathfrak{X}}$ can be defined as follows: For $x\in X$ choose $n\geq1$ such that $x\in U_{n}$. Then the morphism \begin{equation*}\Spa(k(x), k(x)^{+})\to X\end{equation*}determined by $x$ factors through $U_{n}$. The restriction to $A$ of the corresponding homomorphism \begin{equation*}\mathcal{O}_{X}^{+}(U_{n})\to k(x)^{+}\end{equation*}induces a morphism of formal schemes \begin{equation*}\Spf(k(x)^{+})\to \mathfrak{X}=\Spf(A)\end{equation*}and then $\spc_{X,\mathfrak{X}}(x)$ is the image of the closed point of $\Spf(k(x)^{+})$ under this morphism $\Spf(k(x)^{+})\to \mathfrak{X}$.

For $g\in A$, the pre-image of the basic open subset $D(g)\subseteq \Spf(A)$ is the increasing union \begin{equation*}\spc_{X,\mathfrak{X}}^{-1}(D(g))=\bigcup_{n>0}U_{n}(\frac{1}{g})=X(\frac{1}{g}).\end{equation*}In particular, the map $\spc_{X,\mathfrak{X}}$ is continuous. Moreover, for every $g\in A$ and any $m\geq n$, we have a commutative triangle of continuous ring maps \begin{center}\begin{tikzcd}\mathcal{O}_{\mathfrak{X}}(D(g))=A\langle g^{-1}\rangle \arrow{r} \arrow{dr} & B_{m,A}\langle g^{-1}\rangle[\varpi^{-1}]=\mathcal{O}_{X}(U_{m}(\frac{1}{g})) \arrow{d} \\ & B_{n,A}\langle g^{-1}\rangle[\varpi^{-1}], \end{tikzcd}\end{center}whence a canonical continuous map \begin{equation*}\mathcal{O}_{\mathfrak{X}}(D(g))\to \varprojlim_{m}\mathcal{O}_{X}(U_{m}(\frac{1}{g}))=\mathcal{O}_{X}(\spc_{X,\mathfrak{X}}^{-1}(D(g))).\end{equation*}These continuous ring homomorphisms define a morphism of sheaves of complete topological rings $\mathcal{O}_{\mathfrak{X}}\to \spc_{X,\mathfrak{X}\ast}\mathcal{O}_{X}$. When the formal scheme $\mathfrak{X}$ is equipped with the structure of a locally v-ringed space as per Remark~\ref{Formal schemes as locally $v$-ringed spaces}, this morphism is seen to be compatible with the valuations on the stalks. Thus we obtain a morphism of locally v-ringed spaces $X\to\mathfrak{X}$ which we again denote by $\spc_{X,\mathfrak{X}}$ and call it the \textit{specialization morphism} from $X$ to $\mathfrak{X}$.
\begin{example}In the situation of Example~\ref{Fargues-Fontaine curves}, the specialization map is a morphism of locally v-ringed spaces \begin{equation*}Y_{\Spa(R, R^{+})}=\Spa(\mathbb{A}_{\inf}(R^{+}), \mathbb{A}_{\inf}(R^{+}))\setminus\{\, p[\varpi^{\flat}]=0\,\}\to \Spf(\mathbb{A}_{\inf}(R^{+})).\end{equation*}In particular, its underlying map of topological spaces is a continuous map \begin{equation*}\vert Y_{\Spa(R, R^{+})}\vert\to \vert \Spec(R^{\flat+}/(\varpi^{\flat}))\vert=\vert\Spec(R^{+}/(\varpi))\vert.\end{equation*}\end{example}
\begin{lemma}\label{Generic fibers and open immersions}Let $\mathfrak{U}=\Spf(A)$ be a rig-sheafy affine adic formal scheme over $(R, \varpi)$ and $\mathfrak{V}=\Spf(B)$ a rig-sheafy affine formal open subscheme. Then the morphism of adic spaces $V=\mathfrak{V}_{\eta}^{\ad}\to U=\mathfrak{U}_{\eta}^{\ad}$ induced by the inclusion map $\mathfrak{V}\hookrightarrow \mathfrak{U}$ is an open immersion of adic spaces and the diagram of locally v-ringed spaces \begin{center}\begin{tikzcd}V\arrow{r} \arrow{d}{\spc_{V,\mathfrak{V}}} & U \arrow{d}{\spc_{U,\mathfrak{U}}} \\ \mathfrak{V} \arrow[hook]{r} & \mathfrak{U} \end{tikzcd}\end{center}gives rise to an isomorphism \begin{equation*}V\tilde{\rightarrow}\spc_{U,\mathfrak{U}}^{-1}(\mathfrak{V}).\end{equation*}\end{lemma}
\begin{proof}This reduces to the case when $\mathfrak{V}$ is the basic open subset $D(g)$ for some $g\in \mathcal{O}_{\mathfrak{U}}(\mathfrak{U})$. In this case the assertion follows from the equality \begin{equation*}\mathcal{O}_{\mathfrak{U}}(\mathfrak{U})\langle g^{-1}\rangle\langle\frac{f_1^{n},\dots, f_r^{n}}{\varpi}\rangle[\varpi^{-1}]=\mathcal{O}_{\mathfrak{U}}(\mathfrak{U})\langle\frac{f_1^{n},\dots, f_r^{n}}{\varpi}\rangle[\varpi^{-1}]\langle g^{-1}\rangle,\end{equation*}which holds for every $n\geq 1$ and every finite family of elements $f_1,\dots, f_r\in \mathcal{O}_{\mathfrak{U}}(\mathfrak{U})$ generating an ideal of definition of $\mathcal{O}_{\mathfrak{U}}(\mathfrak{U})$, and from the previously observed equality of sets $\spc_{U,\mathfrak{U}}^{-1}(D(g))=U(\frac{1}{g})$.\end{proof}This lemma allows us to globalize the construction of the adic analytic generic fiber by gluing.
\begin{lemma}\label{Relative gluing}Let $S$ be a locally v-ringed space. Let $\mathcal{B}$ be a basis for the topology on $S$. Suppose we are given the following data: \begin{enumerate}[(1)] \item For every $U \in\mathcal{B}$ a morphism of locally v-ringed spaces $f_{U}: X_{U} \to U$. \item For $U, V \in \mathcal{B}$ with $V\subset U$ a morphism $\rho_{V}^{U}: X_{V} \to X_{U}$ over $U$. \end{enumerate}Assume that \begin{enumerate}[(i)]\item each $\rho_{V}^{U}$ induces an isomorphism $X_{V} \to f_{U}^{-1}(V)$ of locally v-ringed spaces over $V$. \item whenever $W, V, U \in \mathcal{B}$ with $W \subset V \subset U$ we have $\rho_{W}^{U} = \rho_{V}^{U}\circ \rho_{W}^{V}$. \end{enumerate}Then there exists a morphism $f: X \to S$ of locally v-ringed spaces and isomorphisms $i_{U}: f^{-1}(U) \to X_{U}$ over $U \in \mathcal{B}$ such that for $V, U \in \mathcal{B}$ and $V \subset U$ the composition \begin{center} \begin{tikzcd} X_{V} \arrow{r}{i_{V}^{-1}} & f^{-1}(V) \arrow[hook]{r} & f^{-1}(U) \arrow{r}{i_{U}} & X_{U} \end{tikzcd} \end{center}is the morphism $\rho_{V}^{U}$. The locally v-ringed space $X$ is unique up to unique isomorphism over $S$. Moreover, if all $X_{U}$ (for $U \in \mathcal{B}$) are adic spaces, so is $X$, and if also $S$ is an adic space and the morphisms $f_{U}$ are adic, then $f: X\to S$ is adic.\end{lemma}\begin{proof}The analog of the assertion for locally ringed spaces is well-known (cf. \cite{Stacks}, Tag~01LH, for the case of schemes). Let $X$ be the locally ringed space obtained in this way. Since the subsets $f^{-1}(U)$, where $U$ ranges over open sets $U \in \mathcal{B}$, form an open cover of $X$, we can endow the corresponding sheaf $\mathcal{O}_{X}$ on $X$ with the structure of a sheaf of complete topological rings by transporting the topological structure on the sheaves $\mathcal{O}_{X_{U}}$ along the isomorphisms of locally ringed spaces $i_{U}: f^{-1}(U) \to X_{U}$. Let $j_{U}$ be the inverse of $i_{U}$. For $x \in X_{U}$ consider the corresponding map on the stalks $j_{U, x}: \mathcal{O}_{X, j_{U}(x)} = \mathcal{O}_{f^{-1}(U), j_{U}(x)} \to \mathcal{O}_{X_{U}, x}$. If $v_{x}$ is the valuation on $\mathcal{O}_{X_{U}, x}$ which is part of the structure of locally v-ringed space carried by $X_{U}$, then $v_{j_{U}(x)}(f) = v_{x}(j_{U, x}(f))$ (for $f \in \mathcal{O}_{X, j_{U}(x)}$) defines a valuation on $\mathcal{O}_{X, j_{U}(x)}$. In this fashion we obtain a structure of locally $v$-ringed space on $X$ and it is routine to check that this locally v-ringed space has the desired properties.\end{proof}          
\begin{lemma}\label{Generic fiber via gluing}Let $R$ be a complete adic ring with principal ideal of definition and let $\varpi$ be an element generating an ideal of definition of $R$. For every locally rig-sheafy adic formal scheme $\mathfrak{X}$ over $(R, \varpi)$, there exist a locally v-ringed space over $\mathfrak{X}$ \begin{equation*}\spc_{\mathfrak{X}_{\eta}^{\ad},\mathfrak{X}}: \mathfrak{X}_{\eta}^{\ad}=\mathfrak{X}_{\eta, (R, \varpi)}^{\ad}\to \mathfrak{X}\end{equation*}and, for every rig-sheafy affine open formal subscheme $\mathfrak{U}$ of $\mathfrak{X}$, an isomorphism \begin{equation*}i_{\mathfrak{U}}: \spc_{\mathfrak{X}_{\eta}^{\ad},\mathfrak{X}}^{-1}(\mathfrak{U})\tilde{\to}\mathfrak{U}_{\eta}^{\ad}=\mathfrak{U}_{\eta, (R, \varpi)}^{\ad}\end{equation*}of locally v-ringed spaces over $\mathfrak{U}$ such that for every inclusion $\mathfrak{V}\subseteq \mathfrak{U}$ of rig-sheafy affine open formal subschemes of $\mathfrak{X}$ the composition \begin{center}\begin{tikzcd}\mathfrak{V}_{\eta}^{\ad}\arrow{r}{i_{\mathfrak{V}}^{-1}} & \spc_{\mathfrak{X}_{\eta}^{\ad},\mathfrak{X}}^{-1}(\mathfrak{V})\arrow[hook]{r} & \spc_{\mathfrak{X}_{\eta}^{\ad},\mathfrak{X}}^{-1}(\mathfrak{U}) \arrow{r}{i_{\mathfrak{U}}} & \mathfrak{U}_{\eta}^{\ad}\end{tikzcd}\end{center}is equal to the canonical morphism $\mathfrak{V}_{\eta}^{\ad}\hookrightarrow \mathfrak{U}_{\eta}^{\ad}$ induced by the inclusion $\mathfrak{V}\subseteq \mathfrak{U}$. The locally v-ringed space $\spc_{\mathfrak{X}_{\eta}^{\ad},\mathfrak{X}}: \mathfrak{X}_{\eta}^{\ad}\to \mathfrak{X}$ is the unique, up to unique isomorphism over $\mathfrak{X}$, locally v-ringed space over $\mathfrak{X}$ with the above property. Moreover, the locally v-ringed space $\mathfrak{X}_{\eta, (R, \varpi)}^{\ad}$ is an analytic adic space. If $\varpi$ is a non-zero-divisor in $R$ and the Tate ring $R[\varpi^{-1}]$ is sheafy, then $\mathfrak{X}_{\eta, (R, \varpi)}^{\ad}$ is an adic space over $\Spa(R[\varpi^{-1}], \overline{R})$.\end{lemma}
\begin{proof}Follows from Lemma~\ref{Generic fibers and open immersions} and Lemma~\ref{Relative gluing}.\end{proof}
\begin{mydef}[Generic fiber and specialization morphism, general case]\label{Generic fiber, general case}For any complete adic ring $R$ with an element $\varpi\in R$ generating an ideal of definition of $R$ and for any locally rig-sheafy adic formal scheme $\mathfrak{X}$ over $(R, \varpi)$, we call the analytic adic space $\mathfrak{X}_{\eta}^{\ad}=\mathfrak{X}_{\eta, (R, \varpi)}^{\ad}$ given by Lemma~\ref{Generic fiber via gluing} the (adic analytic) generic fiber of $\mathfrak{X}$ over $(R, \varpi)$. We call the morphism of locally $v$-ringed spaces \begin{equation*}\spc_{\mathfrak{X}_{\eta}^{\ad},\mathfrak{X}}: \mathfrak{X}_{\eta}^{\ad}\to\mathfrak{X}\end{equation*}given by that lemma the specialization morphism from $\mathfrak{X}_{\eta}^{\ad}$ to $\mathfrak{X}$. If $\varpi$ is a non-zero-divisor of $R$ and the complete Tate ring $R[\varpi^{-1}]$ is sheafy, then we also call $\mathfrak{X}_{\eta, (R, \varpi)}^{\ad}$ the (adic analytic) generic fiber of $\mathfrak{X}$ over $\Spa(R[\varpi^{-1}], \overline{R})$.\end{mydef}
As an immediate consequence of the above construction, we see the following property of the specialization map of adic formal $R$-schemes (note that this statement is no longer true without the assumption that $\mathfrak{X}$ be adic over $\Spf(R)$, as is illustrated already by the example of the open unit disk, Example~\ref{Open unit disk}). 
\begin{lemma}\label{Spectral}Let $R$ be a complete adic ring with an element $\varpi\in R$ generating an ideal of definition of $R$. For any quasi-compact locally rig-sheafy adic formal $R$-scheme $\mathfrak{X}$, the adic analytic generic fiber $\mathfrak{X}_{\eta}^{\ad}$ of $\mathfrak{X}$ is quasi-compact and the specialization map \begin{equation*}\spc_{\mathfrak{X}_{\eta}^{\ad},\mathfrak{X}}: \mathfrak{X}_{\eta}^{\ad}\to \mathfrak{X}\end{equation*}is a quasi-compact continuous map.\end{lemma}
\begin{proof}The specialization map is continuous by definition, so it suffices to show that $\spc_{\mathfrak{X}_{\eta}^{\ad}, \mathfrak{X}}$ is a quasi-compact map (in which case $\mathfrak{X}_{\eta}^{\ad}$ is also quasi-compact). To show this, it suffices to prove that the pre-image of any rig-sheafy affine open subset $\mathfrak{U}\subseteq \mathfrak{X}$ is quasi-compact. By the construction in Lemma~\ref{Generic fiber via gluing}, for every rig-sheafy affine open subset $\mathfrak{U}$ we have an isomorphism of adic spaces \begin{equation*}i_{\mathfrak{U}}: \spc_{\mathfrak{X}_{\eta}^{\ad},\mathfrak{X}}^{-1}(\mathfrak{U})\tilde{\to}\mathfrak{U}_{\eta}^{\ad}.\end{equation*}But since $\mathfrak{X}$ was assumed to be adic over $\Spf(R)$, the target $\mathfrak{U}_{\eta}^{\ad}$ of this isomorphism is an affinoid adic space (if $\mathfrak{U}=\Spf(A)$ for a complete adic ring $A$, then $\mathfrak{U}_{\eta}^{\ad}=\Spa(A[\varpi^{-1}], \overline{A})$). In particular, $\mathfrak{U}_{\eta}^{\ad}$ is quasi-compact.\end{proof}
Note that the assignment $\mathfrak{X}\mapsto\mathfrak{X}_{\eta, (R, \varpi)}^{\ad}$ in Definition~\ref{Generic fiber, general case} is functorial in $\mathfrak{X}$: This was already observed in the affine, rig-sheafy case in the paragraph following the proof of Lemma~\ref{Ideals of definition and generic fiber}, and the general case follows from this by means of gluing. For ease of notation, we usually denote the image under the functor $\mathfrak{X}\mapsto \mathfrak{X}_{\eta}^{\ad}$ of a morphism of locally rig-sheafy adic formal schemes over $(R, \varpi)$ \begin{equation*}f_{0}: \mathfrak{X}\to\mathfrak{S}\end{equation*}by the symbol $f_{0\eta}$ (as opposed to the more cumbersome $f_{0\eta}^{\ad}$).  
\begin{rmk}\label{Generic fiber and Huber's analytic adic space}Note that Definition~\ref{Generic fiber, general case} includes the ``absolute'' case of locally rig-sheafy adic formal schemes $\mathfrak{X}$ over $\mathbb{Z}[[T]]$, in which case we obtain a generic fiber of $\mathfrak{X}$ over $\Spa(\mathbb{Z}((T)), \mathbb{Z}[[T]])$. We also note that for a locally universally rigid-Noetherian adic formal $R$-scheme $\mathfrak{X}$ the adic analytic generic fiber $\mathfrak{X}_{\eta}^{\ad}$ of $\mathfrak{X}$ need not coincide with the analytic adic space $t(\mathfrak{X})_{a}$ associated to $\mathfrak{X}$ by Huber in \cite{Huber3}, \S1.9. In fact, $\mathfrak{X}_{\eta}^{\ad}$ is the open subspace of $t(\mathfrak{X})_{a}$ defined by the condition $\vert\varpi(x)\vert\neq 0$.\end{rmk}
If the locally rig-sheafy formal scheme $\mathfrak{X}$ in question is actually locally sheafy, $\mathfrak{X}_{\eta}^{\ad}$ can be described as a fiber product in the category of adic spaces.   
\begin{lemma}\label{Generic fibers of sheafy formal schemes}Let $R$ be a complete adic ring with ideal of definition generated by a non-zero-divisor $\varpi$. Suppose that the complete Tate ring $R[\varpi^{-1}]$ is sheafy. Let $\mathfrak{X}$ be a locally sheafy adic formal scheme over $R$. Then the fiber product \begin{equation*}\mathfrak{X}^{\ad}\times_{\Spa(R, R)}\Spa(R[\varpi^{-1}], \overline{R})\end{equation*}exists in the category of adic spaces and is equal to $\mathfrak{X}_{\eta, (R, \varpi)}^{\ad}$. Furthermore, the specialization morphism $\spc_{\mathfrak{X}_{\eta}^{\ad},\mathfrak{X}}$ coincides with the composition of the open immersion \begin{equation*}\mathfrak{X}^{\ad}\times_{\Spa(R, R)}\Spa(R[\varpi^{-1}], \overline{R})=\mathfrak{X}^{\ad}\setminus \{\varpi=0\}\hookrightarrow \mathfrak{X}^{\ad}\end{equation*}and the canonical morphism of locally topologically ringed spaces $\pi_{\mathfrak{X}}: \mathfrak{X}^{\ad}\to\mathfrak{X}$.\end{lemma}
\begin{proof}This reduces to the case when $\mathfrak{X}$ is affine and sheafy. Suppose first that $\mathfrak{X}$ is moreover adic over $\Spf(R)$, i.e., the adic topology on $\mathcal{O}_{\mathfrak{X}}(\mathfrak{X})$ is defined by the principal ideal generated by $\varpi$. In this case, consider the Tate ring\begin{equation*}\mathcal{O}_{\mathfrak{X}^{\ad}}(\mathfrak{X}^{\ad})\widehat{\otimes}_{R}R[\varpi^{-1}]=\mathcal{O}_{\mathfrak{X}^{\ad}}(\mathfrak{X}^{\ad})[\varpi^{-1}].\end{equation*}The pre-adic space\begin{equation*}\mathfrak{X}^{\ad}=\Spa(\mathcal{O}_{\mathfrak{X}^{\ad}}(\mathfrak{X}^{\ad}), \mathcal{O}_{\mathfrak{X}^{\ad}}(\mathfrak{X}^{\ad}))\end{equation*}being an adic space implies that its open subspace $\Spa(\mathcal{O}_{\mathfrak{X}^{\ad}}(\mathfrak{X}^{\ad})[\varpi^{-1}], \overline{\mathcal{O}_{\mathfrak{X}^{\ad}}(\mathfrak{X}^{\ad})})$, where the bar on the right denotes integral closure inside $\mathcal{O}_{\mathfrak{X}^{\ad}}(\mathfrak{X}^{\ad})[\varpi^{-1}]$, is an adic space and then it is readily seen that $\Spa(\mathcal{O}_{\mathfrak{X}^{\ad}}(\mathfrak{X}^{\ad})[\varpi^{-1}], \overline{\mathcal{O}_{\mathfrak{X}^{\ad}}(\mathfrak{X}^{\ad})})$ is the desired fiber product in the category of adic spaces.

Let us now turn to the general case, when the topology on $\mathcal{O}_{\mathfrak{X}}(\mathfrak{X})$ is not necessarily the $\varpi$-adic topology. We have to verify that $\mathfrak{X}_{\eta}^{\ad}$ satisfies the universal property of the fiber product. To this end, let \begin{center}\begin{tikzcd}X\arrow{d}{f} \arrow{r} & \Spa(R[\varpi^{-1}], \overline{R}) \arrow[hook]{d} \\ \mathfrak{X}^{\ad} \arrow{r} & \Spa(R, R) \end{tikzcd}\end{center}be a commutative square in the category of adic spaces. Let $f_1,\dots, f_r$ be generators of a finitely generated ideal of definition of $\mathcal{O}_{\mathfrak{X}}(\mathfrak{X})$ containing $\varpi$. We can write $X$ as an increasing union of the open subspaces $f^{-1}(U_{n})$, where we let \begin{equation*}U_{n}=\{\, x\in \mathfrak{X}^{\ad}\mid \vert f_{i}^{n}(x)\vert\leq\vert\varpi(x)\vert\neq0, i=1,\dots,r\,\}.\end{equation*}By definition of the adic analytic generic fiber, $\mathfrak{X}_{\eta}^{\ad}=\bigcup_{n>0}U_{n}$, so the maps $f\vert_{f^{-1}(U_{n})}: f^{-1}(U_{n})\to U_{n}$ induce a map $X\to \mathfrak{X}_{\eta}^{\ad}$. It is clear by construction that this is the only morphism $X\to \mathfrak{X}_{\eta}^{\ad}$ compatible with the map $f: X\to \mathfrak{X}^{\ad}$ and the structure morphism $X\to \Spa(R[\varpi^{-1}], \overline{R})$. This shows that $\mathfrak{X}_{\eta}^{\ad}$ is indeed the fiber product of $\mathfrak{X}^{\ad}$ and $\Spa(R[\varpi^{-1}], \overline{R})$ over $\Spa(R, R)$. The last assertion concerning the specialization map $\spc_{\mathfrak{X}_{\eta}^{\ad},\mathfrak{X}}$ follows from that map's definition and the definition of the morphism of locally topologically ringed spaces $\pi_{\mathfrak{X}}: \mathfrak{X}^{\ad}\to \mathfrak{X}$ (see \cite{Huber2}, proof of Proposition~4.1).\end{proof}
\begin{lemma}\label{Completion along an ideal}Let $A$ be a complete adic ring with finitely generated ideal of definition $I$ and let $J$ be a finitely generated ideal of $A$ containing $I$, with generators $g_1,\dots,g_s$. Let $\varpi\in I$ be some element. Then the canonical continuous map \begin{equation*}A\langle\frac{g_1^{n},\dots,g_s^{n}}{\varpi}\rangle\to \widehat{A}_{J}\langle\frac{g_1^{n},\dots,g_s^{n}}{\varpi}\rangle\end{equation*}induced by the continuous map $A\to \widehat{A}_{J}$ from $A$ to the $J$-adic completion $\widehat{A}_{J}$ of $A$ is a topological isomorphism.\end{lemma}
\begin{proof}The ring $A\langle\frac{g_1^{n},\dots,g_s^{n}}{\varpi}\rangle$ is $J$-adically complete since it is $\varpi$-adically complete and \begin{equation*}J^{sn}A\langle\frac{g_1^{n},\dots,g_s^{n}}{\varpi}\rangle\subseteq (g_1^{n},\dots,g_s^{n})_{A}A\langle\frac{g_1^{n},\dots,g_s^{n}}{\varpi}\rangle=\varpi A\langle\frac{g_1^{n},\dots,g_s^{n}}{\varpi}\rangle.\end{equation*}Thus the canonical map $A\to A\langle\frac{g_1^{n},\dots,g_s^{n}}{\varpi}\rangle$ factors through $A\to \widehat{A}_{J}$ and hence, by the universal property in Remark~\ref{Remark on rig-sheafiness}, also through $A\to \widehat{A}_{J}\langle\frac{g_1^{n},\dots,g_s^{n}}{\varpi}\rangle$.\end{proof}
\begin{lemma}\label{Completion along a closed subscheme}Let $R$ be a complete adic ring admitting a principal ideal of definition and let $\varpi$ be an element generating an ideal of definition of $R$. Let $\mathfrak{X}$ be a locally rig-sheafy $\varpi$-torsion-free adic formal scheme over $(R, \varpi)$ (not necessarily an adic formal $R$-scheme) with an ideal of definition of finite type $\mathcal{I}$ containing $\varpi\mathcal{O}_{\mathfrak{X}}$. Let $Z\subset \mathfrak{X}_{0}$ be a finitely presented closed subscheme. The completion $\widehat{\mathfrak{X}}_{Z}$ of $\mathfrak{X}$ along $Z$ is again locally rig-sheafy over $(R, \varpi)$.\end{lemma}
\begin{proof}Let $(\mathfrak{U}_{i})_{i}=(\Spf(A_{i}))_{i}$ be an open cover by rig-sheafy affine adic formal schemes over $\Spf(R)$. Let $\mathcal{J}$ be an adically quasi-coherent ideal sheaf of finite type such that $\mathcal{J}\supseteq\mathcal{I}$ and $Z=\mathcal{V}(\mathcal{J})$. For every $i$, let $J_{i}=\mathcal{J}(\mathfrak{U}_{i})$. Then an affine open cover of $\widehat{\mathfrak{X}}_{Z}$ is given by $(\Spf(\widehat{A_{i}}_{J_{i}}))_{i}$, where $\widehat{A_{i}}_{J_{i}}$ is the $J_{i}$-adic completion of $A_{i}$. Hence it suffices to prove that for every rig-sheafy complete adic topological $R$-algebra $A$ with finitely generated ideal of definition $I=(f_1,\dots,f_r)_{A}$ containing $\varpi$ and for every finitely generated ideal $J=(g_1,\dots,g_s)_{A}\subsetneq A$ containing $I$ the $J$-adic completion $\widehat{A}_{J}$ of $A$ is rig-sheafy. But the inclusion $I\subseteq J=(g_1,\dots,g_s)_{A}$ entails that \begin{equation*}(f_1^{sn},\dots, f_r^{sn})_{A}\subseteq I^{sn}\subseteq J^{sn}\subseteq (g_1^{n},\dots, g_s^{n})_{A}\end{equation*}for every $n\geq1$, whence we see that, for every $n\geq1$, the canonical map $A\to A\langle\frac{g_1^{n},\dots,g_s^{n}}{\varpi}\rangle$ factors through $A\to A\langle\frac{f_1^{sn},\dots,f_r^{sn}}{\varpi}\rangle$. Thus, for every $n\geq1$, the Tate ring $A\langle\frac{g_1^{n},\dots,g_s^{n}}{\varpi}\rangle[\varpi^{-1}]$ is equal to the rational localization\begin{equation*}A\langle\frac{f_1^{sn},\dots,f_r^{sn}}{\varpi}\rangle\langle\frac{g_1^{n},\dots,g_s^{n}}{\varpi}\rangle[\varpi^{-1}]=A\langle\frac{f_1^{sn},\dots,f_r^{sn}}{\varpi}\rangle[\varpi^{-1}]\langle\frac{g_1^{n},\dots,g_s^{n}}{\varpi}\rangle\end{equation*}of the sheafy Tate ring $A\langle\frac{f_1^{sn},\dots,f_r^{sn}}{\varpi}\rangle[\varpi^{-1}]$. It follows that $A\langle\frac{g_1^{n},\dots,g_s^{n}}{\varpi}\rangle[\varpi^{-1}]$ is a sheafy Tate ring for every $n\geq1$. We conclude by Lemma~\ref{Completion along an ideal} that $\widehat{A}_{J}$ is rig-sheafy.\end{proof}
The following proposition generalizes \cite{Berthelot96}, Proposition~0.2.7, to our situation.
\begin{prop}\label{Berthelot's proposition}Let $R$ and $\varpi$ be as before and let $\mathfrak{X}$ be a locally rig-sheafy adic formal scheme over $\Spf(R)$ (not necessarily an adic formal $R$-scheme) with an ideal of definition of finite type $\mathcal{I}$ containing $\varpi\mathcal{O}_{\mathfrak{X}}$. Let $Z$ be a finitely presented closed subscheme of $\mathfrak{X}_{0}$. The morphism of analytic adic generic fibers over $(R, \varpi)$ \begin{equation*}(\widehat{\mathfrak{X}}_{Z})_{\eta}^{\ad}\to \mathfrak{X}_{\eta}^{\ad}=X\end{equation*}arising from the canonical morphism of locally rig-sheafy formal schemes $\widehat{\mathfrak{X}}_{Z}\to\mathfrak{X}$ over $\Spf(R)$ has image in $\spc_{X,\mathfrak{X}}^{-1}(Z)$ and restricts to a bijection between the set of rank-$1$ points of $(\widehat{\mathfrak{X}}_{Z})_{\eta}^{\ad}$ and the set of rank-$1$ points of $\spc_{X,\mathfrak{X}}^{-1}(Z)$.\end{prop}
\begin{proof}Since $\spc_{X,\mathfrak{X}}^{-1}(\mathfrak{V})=\mathfrak{V}_{\eta}^{\ad}$ for every rig-sheafy affine open $\mathfrak{V}\subseteq \mathfrak{X}$, it suffices to treat the case when $\mathfrak{X}$ is affine and rig-sheafy. Let $A=\mathcal{O}_{\mathfrak{X}}(\mathfrak{X})$, let $I=(f_1,\dots,f_r)_{A}$ be a finitely generated ideal of definition containing $\varpi$ and let $J=(g_1,\dots,g_s)_{A}$ be a finitely generated ideal containing $I$ defining the finitely presented closed subscheme $Z$. Let $U_{n}$, $n\geq1$, be the rational subsets of $\Spa(A, A)$ used in the definition of $\mathfrak{X}_{\eta}^{\ad}$. By definition of the specialization map, a point \begin{equation*}x\in U_{n}=\Spa(A\langle\frac{f_1^{n},\dots,f_r^{n}}{\varpi}\rangle[\varpi^{-1}], A\langle\frac{f_1^{n},\dots,f_r^{n}}{\varpi}\rangle[\varpi^{-1}]^{+})\subseteq \mathfrak{X}_{\eta}^{\ad}\end{equation*}belongs to $\spc_{X,\mathfrak{X}}^{-1}(Z)$ if and only if $\vert g_{j}(x)\vert<1$ for all $j=1,\dots, s$. Hence we get a map \begin{equation*}\bigcup_{n, m}\Lambda_{n,m}\to\spc_{X,\mathfrak{X}}^{-1}(Z)\end{equation*}from the union of the open subspaces\begin{equation*}\Lambda_{n,m}=\{\, x\in U_{m}\mid \vert g_{j}^{n}(x)\vert\leq\vert\varpi(x)\vert\neq 0 \ \text{for all} j=1,\dots, s\,\}\end{equation*}of $X$, for $n, m\geq1$, and this map induces a bijection on rank-$1$ points. On the other hand, the inclusion $I\subseteq J$ implies that \begin{equation*}f_{i}^{sn}\in J^{sn}\subseteq (g_1^{n},\dots,g_s^{n})_{A}\end{equation*}for all $i=1,\dots, r$ and hence that the rational subset \begin{equation*}\{\, x\in \Spa(A, A)\mid \vert g_{j}^{n}(x)\vert\leq\vert\varpi(x)\vert\neq 0\ \text{for} \ j=1,\dots,s\,\}\end{equation*}is contained in $U_{sn}$ for all $n\geq1$. In this way we see that $\spc_{X,\mathfrak{X}}^{-1}(Z)$ has the same rank-$1$ points as the union of the open subspaces $\Lambda_{n,sn}$ of $X$ and that \begin{align*}\Lambda_{n,sn}=\{\, x\in \Spa(A, A)\mid \vert g_{j}^{n}(x)\vert\leq\vert\varpi(x)\vert\neq 0\ \text{for} \ j=1,\dots,s\,\} \\=\Spa(A\langle\frac{g_1^{n},\dots,g_s^{n}}{\varpi}\rangle[\varpi^{-1}], A\langle\frac{g_1^{n},\dots,g_s^{n}}{\varpi}\rangle[\varpi^{-1}]^{+})\end{align*}for all $n\geq1$. By Lemma~\ref{Completion along an ideal},\begin{equation*}\Lambda_{n,sn}=\Spa(\widehat{A}_{J}\langle\frac{g_1^{n},\dots,g_s^{n}}{\varpi}\rangle[\varpi^{-1}], \widehat{A}_{J}\langle\frac{g_1^{n},\dots,g_s^{n}}{\varpi}\rangle[\varpi^{-1}]^{+}).\end{equation*}It follows that the map $(\widehat{\mathfrak{X}}_{Z})_{\eta}^{\ad}\to X$ induces a map\begin{equation*}(\widehat{\mathfrak{X}}_{Z})_{\eta}^{\ad}=\bigcup_{n\geq1}\Spa(\widehat{A}_{J}\langle\frac{g_1^{n},\dots,g_s^{n}}{\varpi}\rangle[\varpi^{-1}], \widehat{A}_{J}\langle\frac{g_1^{n},\dots,g_s^{n}}{\varpi}\rangle[\varpi^{-1}]^{+})\to \spc_{X,\mathfrak{X}}^{-1}(Z)\end{equation*}which restricts to a bijection on rank-$1$ points.\end{proof}
With the notion of adic analytic generic fiber in hand, we can define what we mean by a formal model over $R$ (respectively, a formal $R$-model) of an adic space $X$ over $\Spa(R[\varpi^{-1}], \overline{R})$, where $R$ is a complete adic ring with a non-zero-divisor $\varpi$ such that $R[\varpi^{-1}]$ is a sheafy Tate ring. 
\begin{mydef}[Formal models of an adic space]\label{Definition of formal models}Fix a ring $R$ which is $\varpi$-adically complete and $\varpi$-torsion-free for some $\varpi\in R$ and let $\overline{R}$ be the integral closure of $R$ in the Tate ring $R[\varpi^{-1}]$. Suppose that the Tate ring $R[\varpi^{-1}]$ is sheafy. For an adic space $X$ over $\Spa(R[\varpi^{-1}], \overline{R})$, a formal model of $X$ over $R$ (respectively, a formal $R$-model of $X$) is a locally rig-sheafy adic formal scheme $\mathfrak{X}$ over $(R, \varpi)$ (respectively, a locally rig-sheafy adic formal $R$-scheme $\mathfrak{X}$) such that $\mathfrak{X}_{\eta, (R, \varpi)}^{\ad}\cong X$. For a morphism of adic spaces $f: X'\to X$ over $\Spa(R[\varpi^{-1}], \overline{R})$, a formal model of $f$ over $R$ (respectively, a formal $R$-model of $f$) is a morphism of locally rig-sheafy adic formal schemes over $(R, \varpi)$ (respectively, of locally rig-sheafy adic formal $R$-schemes) $f_{0}: \mathfrak{X}'\to\mathfrak{X}$ with $f_{0\eta}=f$.  \end{mydef}
In particular, we obtain a notion of formal models over $\mathbb{Z}[[T]]$ (respectively, of formal $\mathbb{Z}[[T]]$-models) for every adic space $X$ equipped with a global pseudo-uniformizer $\varpi$, which we call a Tate adic space, viewed as an adic space over $\Spa(\mathbb{Z}((T)), \mathbb{Z}[[T]])$ as in Remark~\ref{Tate adic spaces as relative adic spaces} below. 
\begin{mydef}[Tate adic spaces]\label{Tate adic spaces}A global pseudo-uniformizer of an adic space $X$ is an element $\varpi_{X}\in \mathcal{O}_{X}^{+}(X)$ such that for any affinoid open subspace $U$ of $X$ the image of $\varpi_{X}$ in $\mathcal{O}_{X}(U)$ is a topologically nilpotent unit of $\mathcal{O}_{X}(U)$. A Tate adic space is a pair $(X, \varpi_{X})$, where $X$ is a (necessarily analytic) adic space and $\varpi_{X}\in \mathcal{O}_{X}^{+}(X)$ is a global pseudo-uniformizer of $X$. A Tate adic space $(X, \varpi_{X})$ is said to be locally Noetherian (respectively, uniform, respectively, perfectoid, respectively, sousperfectoid, respectively, diamantine) if the underlying adic space $X$ has this property. 

A morphism $f: (Y, \varpi_{Y})\to (X, \varpi_{X})$ of Tate adic spaces is a morphism of adic spaces $f: Y\to X$ such that the image of $\varpi_{X}$ under the map $\mathcal{O}_{X}^{+}(X)\to \mathcal{O}_{Y}^{+}(Y)$ is equal to $\varpi_{Y}$.\end{mydef}
\begin{rmk}[Tate adic spaces as relative adic spaces]\label{Tate adic spaces as relative adic spaces}For every sheafy Tate Huber pair $(A, A^{+})$ with pseudo-uniformizer $\varpi\in A^{+}$ of $A$ the assignment $X\mapsto (X, \varpi)$ defines a faithful functor from the category of adic spaces over $S=\Spa(A, A^{+})$ to the category of Tate adic spaces. In particular, the category of rigid-analytic varieties over any nonarchimedean field $K$ admits a faithful functor to the category of Tate adic spaces. Conversely, the category of Tate adic spaces is the same as the category of adic spaces over the affinoid adic space\begin{equation*}S=\Spa(\mathbb{Z}((T)), \mathbb{Z}[[T]]),\end{equation*}where $\mathbb{Z}[[T]]$ is endowed with the $T$-adic topology, where $\mathbb{Z}((T))=\mathbb{Z}[[T]][T^{-1}]$ and where for a Tate adic space $(X, \varpi)$ and $U\subseteq X$ affinoid open the morphism of Huber pairs\begin{equation*}(\mathbb{Z}((T)), \mathbb{Z}[[T]])\to (\mathcal{O}_{X}(U), \mathcal{O}_{X}^{+}(U))\end{equation*}is the map sending the variable $T$ to $\varpi_{X}$.\end{rmk}
The above remark allows us to talk about formal $\mathbb{Z}[[T]]$-models of any Tate adic space $X$ and thus gives us a notion of ``absolute'' formal models for Tate adic spaces.

\section{Admissible formal blow-ups and generic fiber}\label{sec:formal blow-ups}

Recall the notion of an admissible formal blow-up of an adic formal scheme $\mathfrak{X}$ of finite ideal type from \cite{FK}, Ch.~II, Definition 1.1.1. For an adic formal scheme of finite ideal type $\mathfrak{X}$, an admissible ideal $\mathcal{J}$ of $\mathfrak{X}$ is an adically quasi-coherent ideal sheaf of finite type on $\mathfrak{X}$ such that $\mathcal{J}$ contains locally on $\mathfrak{X}$ an ideal of definition (\cite{FK}, Ch.~I, Definition~3.7.4). In other words, $\mathcal{J}$ is an adically quasi-coherent ideal sheaf of finite type such that every $x\in \mathfrak{X}$ has an affine open neighbourhood $\mathfrak{U}$ with the property that $\mathcal{J}(\mathfrak{U})\subseteq \mathcal{O}_{\mathfrak{U}}(\mathfrak{U})$ is an open ideal. If $\mathfrak{X}$ has an ideal of definition of finite type $\mathcal{I}$ and if $\mathcal{J}\subseteq \mathcal{O}_{\mathfrak{X}}$ is an admissible ideal, consider for $k\geq 0$ the projective $\mathfrak{X}_{k}$-schemes\begin{equation*}\mathfrak{X}_{k}'=\Proj(\bigoplus_{n\geq 0}\mathcal{J}^{n}\otimes \mathcal{O}_{\mathfrak{X}_{k}})\to \mathfrak{X}_{k}.\end{equation*}These form an inductive system of schemes $(\mathfrak{X}'_{k})_{k}$ whose transition morphisms are closed immersions and this inductive system satisfies the conditions of \cite{EGAIa}, (10.6.3) and (10.6.4). Thus, by loc.~cit., the inductive limit of $(\mathfrak{X}_{k}')_{k}$ is an adic formal scheme of finite ideal type $\mathfrak{X}'$, endowed with a canonical proper adic morphism \begin{center}\begin{tikzcd}(\ast) & \mathfrak{X}'=\varinjlim_{k\geq 0}\Proj(\bigoplus_{n\geq 0}\mathcal{J}^{n}\otimes \mathcal{O}_{\mathfrak{X}_{k}})\arrow{r} & \mathfrak{X}.\end{tikzcd}\end{center}
\begin{mydef}[Fujiwara-Kato \cite{FK}, Ch.~II, Definition 1.1.1]\label{Admissible formal blow-up}Let $\mathfrak{X}$ be an adic formal scheme with an ideal of definition of finite type and let $\mathcal{J}\subseteq \mathcal{O}_{\mathfrak{X}}$ be an admissible ideal. An adic morphism $\mathfrak{X}'\to \mathfrak{X}$ of adic formal schemes of finite ideal type is said to be the admissible formal blow-up along (or in) $\mathcal{J}$ if it is locally isomorphic to a morphism of the form $(\ast)$. We also sometimes call the formal scheme $\mathfrak{X}'$ itself the admissible formal blow-up along (or in) $\mathcal{J}$. 

If $\mathfrak{X}$ is an arbitrary adic formal scheme of finite ideal type with ideal of definition $\mathcal{I}$ not necessarily of finite type, cover $\mathfrak{X}$ by open formal subschemes $\mathfrak{U}_{i}$ such that $\mathcal{I}\vert_{\mathfrak{U}_{i}}$ is of finite type on $\mathfrak{U}_{i}$ and define the admissible formal blow-up $\mathfrak{X}'\to\mathfrak{X}$ of $\mathfrak{X}$ in $\mathcal{J}$ by gluing the admissible formal blow-ups $\mathfrak{U}'_{i}\to \mathfrak{U}_{i}$ of $\mathfrak{U}_{i}$ in $\mathcal{J}\vert_{\mathfrak{U}_{i}}$.

If $\mathfrak{X}$ is an affine formal scheme, where $A$ is a complete adic ring with finitely generated ideal of definition $I$ and the admissible ideal sheaf $\mathcal{J}$ corresponds to the finitely generated open ideal $J=\mathcal{J}(\mathfrak{X})$ of $A$, then we sometimes also call $\mathfrak{X}'\to \mathfrak{X}$ (or $\mathfrak{X}'$) the admissible formal blow-up of $\mathfrak{X}$ in the finitely generated open ideal $J$ of $A$. \end{mydef}
\begin{example}\label{Admissible formal blow-up in the affine case}If $\mathfrak{X}=\Spf(A)$ for a complete adic ring $A$ with finitely generated ideal of definition $I$, then, by construction, the admissible formal blow-up of $\mathfrak{X}$ in a finitely generated open ideal $J$ of $A$ is the formal completion of the usual (scheme-theoretic) blow-up \begin{equation*}\Proj(\bigoplus_{n\geq0}J^{n})\to \Spec(A)\end{equation*}of $\Spec(A)$ in $J$.\end{example}
The admissible formal blow-up of an affine adic formal scheme along an admissible ideal $\mathcal{J}$ admits the following explicit local description, which is a variant of \cite{FK}, Ch.~II, \S1.1(b), and a generalization of \cite{BL1}, Lemma 2.2.
\begin{lemma}\label{Local explicit description}Let $\mathfrak{X}=\Spf(A)$ be an affine formal scheme, where $A$ is a complete adic ring with finitely generated ideal of definition $I$. Let $f_1,\dots, f_r$ be elements of $A$ generating an open ideal $J$ of $A$ (i.e., the ideal $(f_1,\dots, f_r)_{A}$ contains some power of $I$). Then the admissible formal blow-up $\mathfrak{X}'\to \mathfrak{X}$ along (the ideal sheaf defined by) $J=(f_1,\dots, f_{r})_{A}$ has an open cover consisting of the affine formal schemes $\Spf(B_{i})$, where\begin{equation*}B_{i}=A\langle \frac{f_1,\dots, f_r}{f_i}\rangle, \ i=1,\dots, r,\end{equation*}is the $I$-adic completion of the affine blow-up algebra $A[\frac{J}{f_{i}}]$. Moreover, for every $i$, the open subset $\Spf(B_{i})$ of $\mathfrak{X}'$ is the open subset where the stalks of the ideal sheaf $J\mathcal{O}_{\mathfrak{X}'}$ are generated by $f_i$.\end{lemma}
\begin{rmk}\label{Remark on affine blow-up algebra}We note that, just as in Definition~\ref{Locally rig-sheafy formal schemes} and Remark~\ref{Remark on rig-sheafiness}, the completed affine blow-up algebra $B_{i}=A\langle\frac{f_1,\dots,f_r}{f_{i}}\rangle$ is in general not equal to the rational localization of $A$ usually denoted in the same way (since $f_{i}$ is in general not invertible in $B_{i}$) but is a ring of definition of the said rational localization.\end{rmk} 
\begin{proof}[Proof of Lemma~\ref{Local explicit description}] By Example \ref{Admissible formal blow-up in the affine case} the admissible formal blow-up $\mathfrak{X}'\to \mathfrak{X}$ is the $I$-adic completion of the scheme-theoretic blow-up $\mathscr{X}'\to \mathscr{X}$ of $\mathscr{X}=\Spec(A)$ in the ideal $J=(f_1,\dots, f_r)_{A}$. Recall (for example, from \cite{Stacks}, Tag~0804) that the scheme-theoretic blow-up has an open affine covering by spectra of the affine blow-up algebras $A[\frac{J}{f_{i}}]$, for $i=1,\dots, r$. Thus $\mathfrak{X}'$ has an open cover by formal spectra of the $I$-adic completions of $A[\frac{J}{f_{i}}]$, for $i=1,\dots, r$.\end{proof}
We also need the following immediate generalization of \cite{BL1}, Lemma~2.6.
\begin{lemma}\label{Admissible formal blow-up and affine opens}Let $\mathfrak{X}$ be a quasi-compact quasi-separated adic formal scheme of finite ideal type and let $(\mathfrak{U}_{i})_{i}$ be a finite family of quasi-compact open formal subschemes. For each $i$ consider an admissible ideal $\mathcal{J}_{i}\subseteq \mathcal{O}_{\mathfrak{U}_{i}}$ of finite type and let $\varphi_{i}: \mathfrak{U}_{i}'\to \mathfrak{U}_{i}$ be the admissible formal blow-up along $\mathcal{J}_{i}$. Then \begin{enumerate}[(1)]\item Each $\mathcal{J}_{i}$ extends to an admissible ideal $\mathcal{J}_{i}'\subseteq \mathcal{O}_{\mathfrak{X}}$ and the admissible formal blow-up $\psi_{i}: \mathfrak{X}_{i}'\to \mathfrak{X}$ of $\mathfrak{X}$ in $\mathcal{J}_{i}'$ extends $\varphi_{i}$. \item There exists an admissible formal blow-up $\psi: \mathfrak{X}'\to \mathfrak{X}$ which factors as \begin{center}\begin{tikzcd}\mathfrak{X}'\arrow{r} & \mathfrak{X}_{i}'\arrow{r}{\psi_{i}} & \mathfrak{X},\end{tikzcd}\end{center}with $\mathfrak{X}'\to \mathfrak{X}_{i}'$ an admissible formal blow-up, for every $i$.\end{enumerate}\end{lemma}
\begin{proof}Fix an ideal of definition $\mathcal{I}$ of finite type of $\mathfrak{X}$ (the existence of $\mathcal{I}$ is guaranteed by the assumption that $\mathfrak{X}$ is qcqs and \cite{FK}, Ch.~I, Corollary~3.7.12). Choose an integer $k\geq 0$ such that \begin{equation*}\mathcal{J}_{i}\supseteq \mathcal{I}^{k+1}\vert_{\mathfrak{U}_{i}}\end{equation*}for all $i$. Hence $\mathcal{J}_{i}$ defines an ideal sheaf $\overline{\mathcal{J}_{i}}$ on the quasi-compact open subscheme \begin{equation*}\mathfrak{U}_{i,k}=(\mathfrak{U}_{i}, \mathcal{O}_{\mathfrak{U}_{i}}/\mathcal{I}^{k+1}\mathcal{O}_{\mathfrak{U}_{i}})\end{equation*}of the quasi-compact quasi-separated scheme $\mathfrak{X}_{k}=(\mathfrak{X}, \mathcal{O}_{\mathfrak{X}}/\mathcal{I}^{k+1}\mathcal{O}_{\mathfrak{X}})$. By \cite{EGAIb}, \S6.9, every quasi-coherent sheaf on $\mathfrak{X}_{k}$ is a union of its quasi-coherent sub-sheaves of finite type. By \cite{EGAIa}, Cor.~9.4.3, this implies that the quasi-coherent ideal sheaf of finite type $\overline{\mathcal{J}_{i}}$ on $\mathfrak{U}_{i,k}$ can be extended to a quasi-coherent ideal sheaf of finite type $\overline{\mathcal{J}_{i}'}$ on all of $\mathfrak{X}_{k}$, for every $i$. Using \cite{FK}, Ch.~I, Corollary 3.7.3, we obtain admissible ideal sheaves $\mathcal{J}_{i}'$ on $\mathfrak{X}$ such that $\mathcal{J}_{i}'\vert_{\mathfrak{U}_{i}}=\mathcal{J}_{i}$ for every $i$. The admissible formal blow-up of $\mathfrak{X}$ in $\mathcal{J}_{i}'$ extends $\mathfrak{U}_{i}'\to \mathfrak{U}_{i}$, giving the desired $\psi_{i}$. 

Finally, by \cite{FK}, Proposition II.1.1.10 and Exercise II.1.1, the admissible formal blow-up $\psi: \mathfrak{X}'\to\mathfrak{X}$ of $\mathfrak{X}$ in the product ideal sheaf $\mathcal{J}=\prod_{i}\mathcal{J}_{i}'$ factors as\begin{center}\begin{tikzcd}\mathfrak{X}'\arrow{r} & \mathfrak{X}_{i}'\arrow{r}{\psi_{i}} & \mathfrak{X},\end{tikzcd}\end{center}where $\mathfrak{X}'\to \mathfrak{X}_{i}'$ is an admissible formal blow-up, for every $i$.\end{proof}
Admissible formal blow-ups also satisfy a universal property similar to (and derived from) the universal property of the usual, scheme-theoretic blow-ups, see \cite{FK}, Ch.~II, Proposition 1.1.4(3). This universal property can sometimes be used as a replacement for the valuative criterion of properness in the setting of formal schemes. The following definition is inspired by the formulation of \cite{BhattNotes}, Theorem~8.1.2. 
\begin{mydef}[Specialization to an admissible formal blow-up]\label{Specialization to an admissible formal blow-up}Let $R$ be a complete adic ring with ideal of definition generated by a single non-zero-divisor $\varpi$ and suppose that the Tate ring $R[\varpi^{-1}]$ is sheafy. Let $\mathfrak{X}$ be an adic formal scheme over $R$ which is locally rig-sheafy over $(R, \varpi)$, with adic analytic generic fiber $X=\mathfrak{X}_{\eta}^{\ad}$ over $(R, \varpi)$, and let $\mathfrak{X}'\to \mathfrak{X}$ be the admissible formal blow-up of $\mathfrak{X}$ along an admissible ideal sheaf $\mathcal{J}$. We define a map \begin{equation*}\spc_{X,\mathfrak{X}'}: \vert X\vert\to \vert\mathfrak{X}'\vert\end{equation*}as follows. 

For $x\in X$ choose a rig-sheafy affine open neighbourhood $\mathfrak{U}=\Spf(A)$ of the point $\spc_{X,\mathfrak{X}}(x)$ in $\mathfrak{X}$. Then the morphism of adic spaces $\Spa(k(x), k(x)^{+})\to X$ determined by $x$ factors through $\spc_{X,\mathfrak{X}}^{-1}(\mathfrak{U})=\mathfrak{U}_{\eta}^{\ad}$ and hence factors through an affinoid open subspace of the form \begin{equation*}\Spa(A\langle\frac{f_1^{n},\dots, f_r^{n}}{\varpi}\rangle[\varpi^{-1}], A\langle\frac{f_1^{n},\dots,f_r^{n}}{\varpi}\rangle[\varpi^{-1}]^{+})\end{equation*}for some elements $f_1,\dots, f_r$ generating an ideal of definition of $A$ containing $\varpi$ and some integer $n\geq1$. The restriction to $A$ of the continuous homomorphism \begin{equation*}A\langle\frac{f_1^{n},\dots,f_r^{n}}{\varpi}\rangle[\varpi^{-1}]^{+}\to k(x)^{+}\end{equation*}then defines a morphism of formal schemes $\Spf(k(x)^{+})\to \mathfrak{X}$ (which factors through $\mathfrak{U}$) such that $\spc_{X,\mathfrak{X}}(x)$ is the image under this morphism of the closed point of $\Spf(k(x)^{+})$. 

It is readily seen that the morphism $\Spf(k(x)^{+})\to \mathfrak{X}$ constructed in this way does not depend on the chosen rig-sheafy affine open neighbourhood $\mathfrak{U}$ of $\spc_{X,\mathfrak{X}}(x)$. Since the ideal sheaf of finite type $\mathcal{J}\mathcal{O}_{\Spf(k(x)^{+})}$ on $\Spf(k(x)^{+})$ is invertible, this morphism $\Spf(k(x)^{+})\to \mathfrak{X}$ lifts uniquely to a morphism of formal schemes \begin{equation*}\Spf(k(x)^{+})\to \mathfrak{X}',\end{equation*}by the universal property of admissible formal blow-ups (\cite{FK}, Ch.~II, Prop.~1.1.4(3)). We then define \begin{equation*}\spc_{X,\mathfrak{X}'}(x)\in\mathfrak{X}'\end{equation*}to be the image of the closed point of $\Spf(k(x)^{+})$ under this morphism $\Spf(k(x)^{+})\to \mathfrak{X}'$.\end{mydef}
Similarly to the case $\mathfrak{X}'=\mathfrak{X}$, we have an explicit description of inverse images of rig-sheafy affine open subsets under the map $\spc_{X,\mathfrak{X}'}$.
\begin{lemma}\label{Description of inverse images}For a locally rig-sheafy adic formal scheme $\mathfrak{X}$ over $(R, \varpi)$ with generic fiber $X=\mathfrak{X}_{\eta}^{\ad}$ and an admissible formal blow-up $\mathfrak{X}'\to \mathfrak{X}$, the pre-image $\spc_{X,\mathfrak{X}'}^{-1}(\mathfrak{V})$ of any rig-sheafy affine open subset $\mathfrak{V}\subseteq \mathfrak{X}'$ coincides with the open subspace $V=\mathfrak{V}_{\eta}^{\ad}$ of $X$. Moreover, $\spc_{X,\mathfrak{X}'}\vert_{V}=\spc_{V,\mathfrak{V}}$.\end{lemma}
\begin{proof}The definition of an admissible formal blow-up is of local nature, so, for every affine open $\mathfrak{U}\subseteq\mathfrak{X}$, the restricted morphism $\mathfrak{U}'\to \mathfrak{U}$, where $\mathfrak{U}'$ is the pre-image of $\mathfrak{U}$ in $\mathfrak{X}'$, is an admissible formal blow-up. On the other hand, the definition of the (adic analytic) generic fiber of $\mathfrak{X}$ over $(R, \varpi)$ is also local in nature. Hence, to prove the equality $\mathfrak{V}_{\eta}^{\ad}=\spc_{X,\mathfrak{X}'}^{-1}(\mathfrak{V})$ for every rig-sheafy affine open subset $\mathfrak{V}\subseteq\mathfrak{X}'$, we may assume that $\mathfrak{X}=\Spf(A)$ is itself a rig-sheafy affine formal scheme. 

Let $f_1,\dots, f_r\in A$ be elements generating an ideal of definition of $A$ as in the definition of a rig-sheafy adic ring over $(R, \varpi)$. Let $\mathfrak{V}=\Spf(B)$ be a rig-sheafy affine open subset of the admissible formal blow-up $\mathfrak{X}'$. Since admissible formal blow-ups are adic maps, $(f_1,\dots, f_r)_{B}$ is an ideal of definition of $B$ (by abuse of notation, we identify $f_1,\dots, f_r$ with their images in $B$). Let $x\in X$ and fix some $n\geq1$ such that \begin{equation*}x\in \{\, y\in X\mid \vert f_{i}^{n}(y)\vert\leq\vert\varpi(y)\vert\neq 0\,\}=\Spa(A\langle\frac{f_1^{n},\dots, f_r^{n}}{\varpi}\rangle[\varpi^{-1}], A\langle\frac{f_1^{n},\dots, f_r^{n}}{\varpi}\rangle[\varpi^{-1}]^{+}).\end{equation*}The open subspace $V=\mathfrak{V}_{\eta}^{\ad}$ of $X$ is given by the increasing union of open affinoid subspaces \begin{equation*}\Spa(B\langle\frac{f_1^{m},\dots, f_r^{m}}{\varpi}\rangle[\varpi^{-1}], B\langle\frac{f_1^{m},\dots, f_r^{m}}{\varpi}\rangle[\varpi^{-1}]^{+}), \ m\geq1.\end{equation*}If $x\in V$, then, since $\vert f_{i}^{n}(x)\vert\leq \vert\varpi(x)\vert\neq 0$, the point $x$ belongs to the affinoid open subspace $\Spa(B\langle\frac{f_1^{n},\dots, f_r^{n}}{\varpi}\rangle[\varpi^{-1}], B\langle\frac{f_1^{n},\dots, f_r^{n}}{\varpi}\rangle[\varpi^{-1}]^{+})$, which means that the continuous map \begin{equation*}A\langle\frac{f_1^{n},\dots, f_r^{n}}{\varpi}\rangle[\varpi^{-1}]^{+}\to k(x)^{+}\end{equation*}factors through the canonical map \begin{equation*}A\langle\frac{f_1^{n},\dots, f_r^{n}}{\varpi}\rangle[\varpi^{-1}]^{+}\to B\langle\frac{f_1^{n},\dots, f_r^{n}}{\varpi}\rangle[\varpi^{-1}]^{+}.\end{equation*}In particular, the restriction of $A\langle\frac{f_1^{n},\dots, f_r^{n}}{\varpi}\rangle[\varpi^{-1}]^{+}\to k(x)^{+}$ to $A$ factors through the continuous map $A\to B$ induced by $\mathfrak{X}'\to \mathfrak{X}$. We obtain a commutative triangle of morphisms of formal schemes \begin{center}\begin{tikzcd}\Spf(k(x)^{+})\arrow{r} \arrow{dr} & \Spf(B)=\mathfrak{V}\subseteq\mathfrak{X}' \arrow{d} \\ & \Spf(A)=\mathfrak{X}.\end{tikzcd}\end{center}Since, by the universal property of admissible formal blow-ups, the lift of $\Spf(k(x)^{+})\to \mathfrak{X}$ to a morphism $\Spf(k(x)^{+})\to \mathfrak{X}'$ is unique, we conclude that this lift factors through $\mathfrak{V}$; in particular, $\spc_{X,\mathfrak{X}'}(x)\in\mathfrak{V}$. This shows that $V\subseteq \spc_{X,\mathfrak{X}'}^{-1}(\mathfrak{V})$.

Conversely, fix $n\geq1$ as before and assume that $\spc_{X,\mathfrak{X}'}(x)\in \mathfrak{V}$. Unraveling the definitions, this means that the image of the closed point of $\Spf(k(x)^{+})$ under the unique lift $\Spf(k(x)^{+})\to \mathfrak{X}'$ of the morphism $\Spf(k(x)^{+})\to \mathfrak{X}$ which corresponds to $\Spa(k(x), k(x)^{+})\to X$ lies in $\mathfrak{V}$. Since $\mathfrak{V}$ is open in $\mathfrak{X}'$ this entails that $\Spf(k(x)^{+})\to \mathfrak{X}'$ factors through the open formal subscheme $\mathfrak{V}$. This, in turn, implies that the continuous ring homomorphism $A=\mathcal{O}_{\mathfrak{X}}(\mathfrak{X})\to k(x)^{+}$ corresponding to $\Spf(k(x)^{+})\to \mathfrak{X}$ factors through $B$. But, by definition of the morphism $\Spf(k(x)^{+})\to \mathfrak{X}$, this homomorphism $A\to k(x)^{+}$ is the restriction to $A$ of the homomorphism $A\langle\frac{f_1^{n},\dots,f_r^{n}}{\varpi}\rangle[\varpi^{-1}]^{+}\to k(x)^{+}$ which arises from the morphism of adic spaces \begin{equation*}\Spa(k(x), k(x)^{+})\to X\end{equation*}determined by the point $x\in X$. Since $(f_1^{n},\dots, f_{r}^{n})_{k(x)^{+}}\subseteq\varpi k(x)^{+}$, the map $B\to k(x)^{+}$ factors through $B\to B\langle\frac{f_1^{n},\dots,f_r^{n}}{\varpi}\rangle$, by Remark~\ref{Remark on rig-sheafiness}. Consequently, \begin{equation*}x\in\Spa(B\langle\frac{f_1^{n},\dots,f_r^{n}}{\varpi}\rangle[\varpi^{-1}], B\langle\frac{f_1^{n},\dots,f_r^{n}}{\varpi}\rangle[\varpi^{-1}]^{+})\subseteq V.\end{equation*}We conclude that $\spc_{X,\mathfrak{X}'}^{-1}(\mathfrak{V})\subseteq V$, as desired.

As to the last assertion of the lemma, it suffices to prove that, for any point $x\in V$, the morphism of formal schemes $\Spf(k(x)^{+})\to \mathfrak{V}$ which corresponds to the morphism of adic spaces $\Spa(k(x), k(x)^{+})\to V$ is a lift of the analogous morphism $\Spf(k(x)^{+})\to \mathfrak{X}$ (the desired equality $\spc_{X,\mathfrak{X}'}(x)=\spc_{V,\mathfrak{V}}(x)$ then follows from the uniqueness of the lift $\Spf(k(x)^{+})\to\mathfrak{X}'$ of $\Spf(k(x)^{+})\to \mathfrak{X}$). To this end, let $x\in V$ and choose a rig-sheafy affine open subset $\mathfrak{U}$ of $\mathfrak{X}$ such that $\spc_{V,\mathfrak{V}}(x)\in\mathfrak{V}\cap\mathfrak{U}'$, where $\mathfrak{U}'$ is the pre-image of $\mathfrak{U}$ in $\mathfrak{X}'$. Let $\mathfrak{W}$ be a rig-sheafy affine open neighbourhood of $\spc_{V,\mathfrak{V}}(x)$ in $\mathfrak{V}\cap\mathfrak{U}'$ and let $U=\mathfrak{U}_{\eta}^{\ad}=\spc_{X,\mathfrak{X}}^{-1}(\mathfrak{U})$, $W=\mathfrak{W}_{\eta}^{\ad}=\spc_{V,\mathfrak{V}}^{-1}(\mathfrak{W})$. Since $x\in W\subseteq V$, the canonical morphism of adic spaces $\Spa(k(x), k(x)^{+})\to V$ factors through $W$. This implies that the map $\Spf(k(x)^{+})\to \mathfrak{V}$ factors through $\mathfrak{W}$. Thus it coincides with the analogously defined map $\Spf(k(x)^{+})\to \mathfrak{W}$. In particular, it suffices to prove that $\Spf(k(x)^{+})\to \mathfrak{W}\subseteq\mathfrak{X}'$ is a lift of $\Spf(k(x)^{+})\to \mathfrak{X}$. But the map $\Spf(k(x)^{+})\to \mathfrak{X}$ coincides with the map $\Spf(k(x)^{+})\to \mathfrak{U}$ corresponding to the morphism of adic spaces $\Spa(k(x), k(x)^{+})\to U$, since the morphism $\Spa(k(x), k(x)^{+})\to X$ factors through $U$. Thus we may replace $\mathfrak{X}$ with $\mathfrak{U}$, replace $\mathfrak{V}$ with $\mathfrak{W}$ and assume that $\mathfrak{X}$ is affine and rig-sheafy. Write $\mathfrak{X}=\Spf(A)$ and $\mathfrak{V}=\Spf(B)$ for rig-sheafy complete adic rings $A$ and $B$ over $(R, \varpi)$. Choose elements $f_1,\dots, f_r\in A$ generating an ideal of definition of $A$ which contains $\varpi$, and choose some $n\geq1$ such that $\vert f_{i}^{n}(x)\vert\leq \vert\varpi(x)\vert\neq0$ for all $i=1,\dots, r$. The resulting continuous $A$-algebra map \begin{equation*}B\langle\frac{f_1^{n},\dots, f_r^{n}}{\varpi}\rangle[\varpi^{-1}]^{+}\to k(x)^{+}\end{equation*}extends the continuous $A$-algebra map $A\langle\frac{f_1^{n},\dots, f_r^{n}}{\varpi}\rangle[\varpi^{-1}]^{+}\to k(x)^{+}$. In particular, the restricted continuous map $B\to k(x)^{+}$ extends the analogous map $A\to k(x)^{+}$. This shows that the morphism $\Spf(k(x)^{+})\to \mathfrak{V}=\Spf(B)$ lifts the analogous morphism $\Spf(k(x)^{+})\to \mathfrak{X}=\Spf(A)$, as desired.\end{proof}
The last part of the above lemma in particular endows $\spc_{X,\mathfrak{X}'}$ with the structure of a morphism of locally v-ringed spaces (since the maps $\spc_{\mathfrak{V}_{\eta}^{\ad},\mathfrak{V}}$ for rig-sheafy affine open subsets $\mathfrak{V}\subseteq \mathfrak{X}'$ are morphisms of locally v-ringed spaces). In this way, every admissible formal blow-up $\mathfrak{X}'\to\mathfrak{X}$ of a locally rig-sheafy adic formal scheme $\mathfrak{X}$ over $R$ gives rise to a commutative triangle of morphisms of locally v-ringed spaces over $\Spf(R)$: \begin{center}\begin{tikzcd}X=\mathfrak{X}_{\eta}^{\ad} \arrow{r}{\spc_{X,\mathfrak{X}'}} \arrow{dr}{\spc_{X,\mathfrak{X}}} & \mathfrak{X}' \arrow{d} \\ & \mathfrak{X}.\end{tikzcd}\end{center}

As is the case in classical rigid geometry of Tate-Raynaud, one of the useful features of admissible formal blow-ups is that they provide a way to build new formal models of an analytic space from old ones. 
\begin{lemma}\label{Admissible formal blow-ups and generic fiber}Let $R$ be a complete adic ring with ideal of definition generated by a single non-zero-divisor $\varpi$ and suppose that the Tate ring $R[\varpi^{-1}]$ is sheafy. Let $\mathfrak{X}$ be a locally rig-sheafy adic formal scheme over $R$. If $\mathfrak{X}'\to \mathfrak{X}$ is an admissible formal blow-up, the induced morphism ${\mathfrak{X}'}_{\eta}^{\ad}\to \mathfrak{X}_{\eta}^{\ad}$ of adic spaces over $\Spa(R[\varpi^{-1}], \overline{R})$ is an isomorphism.\end{lemma}
\begin{proof}It suffices to prove this when $\mathfrak{X}$ is affine and rig-sheafy. Write $\mathfrak{X}=\Spf(A)$, where $A$ is a rig-sheafy complete adic ring over $R$. By Lemma~\ref{Local explicit description}, $\mathfrak{X}'$ has an affine open cover $(\mathfrak{U}_{i})_{i=1,\dots,r}$ where $\mathfrak{U}_{i}$ is the formal spectrum of the completed affine blow-up algebra $A\langle\frac{f_1,\dots, f_r}{f_i}\rangle$ for some family of elements $f_1,\dots, f_r\in A$ generating an open ideal of the adic ring $A$. Let $I$ be a finitely generated ideal of definition of $A$ and assume that $I$ contains (the image of) $\varpi$. Let $g_1,\dots, g_s$ be generators of $I$. Note that, for every $i=1,\dots, r$ and every $n\geq 1$, the Tate ring \begin{equation*}A\langle\frac{f_1,\dots, f_r}{f_{i}}\rangle\langle\frac{g_1^{n},\dots, g_s^{n}}{\varpi}\rangle[\varpi^{-1}]\end{equation*}is equal to the Tate ring \begin{equation*}A\langle\frac{g_1^{n},\dots, g_s^{n}}{\varpi}\rangle\langle\frac{f_1,\dots, f_r}{f_{i}}\rangle[\varpi^{-1}].\end{equation*}Indeed, we can use the assumption that $f_1,\dots, f_r$ generate an open ideal of $A$ to see that $f_{i}$ is invertible in both of these Tate rings, then conclude that both Tate rings satisfy the universal property of the rational localization corresponding to the rational subset \begin{equation*}\{\, x\in X\mid \vert f_{j}(x)\vert\leq \vert f_{i}(x)\vert\neq 0, \ \vert g_k^{n}(x)\vert\leq \vert\varpi(x)\vert\neq 0\ \text{for} \ j=1,\dots, r, k=1,\dots, s\,\}\end{equation*}of $X$. In particular, we see that the adic rings $A\langle\frac{f_1,\dots, f_r}{f_{i}}\rangle$, for $i=1,\dots, r$, are rig-sheafy.

The pre-images $\spc_{\mathfrak{X}_{\eta}^{\ad},\mathfrak{X}'}^{-1}(\mathfrak{U}_{i})$ form an open cover of $\mathfrak{X}_{\eta}^{\ad}$. Cover each intersection $\mathfrak{U}_{i}\cap \mathfrak{U}_{j}$ by rig-sheafy affine open subsets $\mathfrak{V}_{ijk}$. Using Lemma \ref{Description of inverse images} and the result of the previous paragraph, we have \begin{equation*}\spc_{\mathfrak{X}_{\eta}^{\ad},\mathfrak{X}'}^{-1}(\mathfrak{U}_{i})=(\mathfrak{U}_{i})_{\eta}^{\ad}\end{equation*}and \begin{equation*}\spc_{\mathfrak{X}_{\eta}^{\ad},\mathfrak{X}'}^{-1}(\mathfrak{V}_{ijk})=(\mathfrak{V}_{ijk})_{\eta}^{\ad}\end{equation*}for all $i, j, k$. Hence the generic fiber $\mathfrak{X}_{\eta}^{\ad}$ of $\mathfrak{X}$ can be described as the coequalizer (in the category of adic spaces) of the morphisms\begin{equation*}\coprod_{i,j,k}(\mathfrak{V}_{ijk})_{\eta}^{\ad}\rightrightarrows \coprod_{i}(\mathfrak{U}_{i})_{\eta}^{\ad}\end{equation*}induced by the canonical morphisms\begin{equation*}\coprod_{i,j,k}\mathfrak{V}_{ijk}\rightrightarrows \coprod_{i}\mathfrak{U}_{i}.\end{equation*}But the generic fiber ${\mathfrak{X}'}_{\eta}^{\ad}$ of $\mathfrak{X}'$ has the same description as a coequalizer in the category of adic spaces.\end{proof} 
\begin{rmk}We see from the definitions that, for every admissible formal blow-up $\mathfrak{X}'\to \mathfrak{X}$ of a $\varpi$-torsion-free locally rig-sheafy adic formal scheme $\mathfrak{X}$ over $R$, the morphism of locally v-ringed spaces $\spc_{\mathfrak{X}_{\eta}^{\ad},\mathfrak{X}'}: \mathfrak{X}_{\eta}^{\ad}\to \mathfrak{X}'$ factors as the composition of the canonical isomorphism \begin{equation*}\mathfrak{X}_{\eta}^{\ad}\cong {\mathfrak{X}'}_{\eta}^{\ad}\end{equation*}supplied by Lemma \ref{Admissible formal blow-ups and generic fiber} and $\spc_{{\mathfrak{X}'}_{\eta}^{\ad},\mathfrak{X}'}$.\end{rmk}
\begin{rmk}Note that the above lemma generalizes assertion (a) in the proof of \cite{BL1}, Theorem~4.1.\end{rmk} 
In the special case when $\mathfrak{X}$ is adic over $\Spf(R)$ we also obtain a generalization of \cite{BL1}, Lemma~4.4.
\begin{lemma}\label{Open covers and formal models}Let $R$ be a complete adic ring with ideal of definition generated by a single non-zero-divisor $\varpi$ and suppose that the Tate ring $R[\varpi^{-1}]$ is sheafy. Let $\mathfrak{X}$ be a quasi-compact quasi-separated locally rig-sheafy adic formal $R$-scheme and let $(U_{i})_{i}$ be a finite family of quasi-compact open subsets of the generic fiber $X=\mathfrak{X}_{\eta}^{\ad}$ of $\mathfrak{X}$. Then there exists an admissible formal blow-up $\mathfrak{X}'\to \mathfrak{X}$ of $\mathfrak{X}$ with $\mathfrak{X}'$ being $\varpi$-torsion-free and a family of quasi-compact open formal subschemes $(\mathfrak{U}_{i})_{i}$ of $\mathfrak{X}'$ such that $U_{i}=\spc_{X,\mathfrak{X}'}^{-1}(\mathfrak{U}_{i})$. \end{lemma}
\begin{proof}For every admissible formal blow-up $\mathfrak{X}'\to\mathfrak{X}$ of $\mathfrak{X}$ there exists an admissible formal blow-up $\mathfrak{X}''\to \mathfrak{X}'$ such that $\mathfrak{X}''$ is $\varpi$-torsion-free: Indeed, we can take $\mathfrak{X}''\to\mathfrak{X}'$ to be the admissible formal blow-up of $\mathfrak{X}'$ in the admissible ideal sheaf $\varpi\mathcal{O}_{\mathfrak{X}'}$. By \cite{FK}, Ch.~II, Proposition~1.1.10, the composition $\mathfrak{X}''\to\mathfrak{X}$ is again an admissible formal blow-up of $\mathfrak{X}$. 

Thus, in view of the commutative triangle displayed above Lemma~\ref{Admissible formal blow-ups and generic fiber}, it suffices to find any admissible formal blow-up $\mathfrak{X}'\to\mathfrak{X}$ such that each $U_{i}$ is the pre-image of a quasi-compact open formal subscheme $\mathfrak{U}_{i}$ of $\mathfrak{X}'$. Using Lemma~\ref{Admissible formal blow-up and affine opens}, we reduce to proving the assertion in the case when $\mathfrak{X}=\Spf(A)$ is affine, where $A$ is a $\varpi$-adically complete $R$-algebra, endowed with the $\varpi$-adic topology. In this case $X=\mathfrak{X}_{\eta}^{\ad}$ is an affinoid adic space. Since every quasi-compact open subset of $X$ is a finite union of rational subsets, we may assume that all $U_{i}$ are rational subsets of $X$. For every index $i$, write $U_{i}=X(\frac{f_{1,i},\dots, f_{n_{i}, i}}{g_{i}})$ for some family $f_{1,i}, \dots, f_{n_{i},i}, g_{i}$ of elements of $A[\varpi^{-1}]=\mathcal{O}_{X}(X)$ generating the unit ideal of $A[\varpi^{-1}]$. Up to multiplying all these elements by a power of $\varpi$, we may assume that $f_{1,i}, \dots, f_{n_{i},i}, g_{i}\in A$ for all $i$. For each $i$ set $J_{i}=(f_{1,i}, \dots, f_{n_{i},i}, g_{i})_{A}$, which is an open ideal of $A$. Let $\mathfrak{X}_{i}\to \mathfrak{X}$ be the admissible formal blow-up of $\mathfrak{X}$ in the ideal sheaf $\mathcal{J}_{i}$ defined by $J_{i}$ and let $\mathfrak{U}_{i}$ be the open affine formal subscheme \begin{equation*}\Spf(A\langle\frac{f_{1,i},\dots, f_{n_{i},i}}{g_{i}}\rangle)\end{equation*}of $\mathfrak{X}_{i}$ (see Lemma~\ref{Local explicit description}), where, as in Lemma~\ref{Local explicit description}, the notation \begin{equation*}A\langle\frac{f_{1,i},\dots, f_{n_{i},i}}{g_{i}}\rangle\end{equation*}means the $\varpi$-adic completion of the affine blow-up algebra $A[\frac{J_{i}}{g_{i}}]$ (and not the homonymous rational localization of $A$). Since this completed affine blow-up algebra, with its $\varpi$-adic topology, is a ring of definition of the rational localization \begin{equation*}\mathcal{O}_{X}(U_{i})=A[\varpi^{-1}]\langle\frac{f_{1,i},\dots, f_{n_{i},i}}{g_{i}}\rangle\end{equation*}of $A[\varpi^{-1}]$, we see, using Lemma~\ref{Description of inverse images}, that $U_{i}=\spc_{X,\mathfrak{X}_{i}}^{-1}(\mathfrak{U}_{i})$ for every $i$. On the other hand, blowing up in the product ideal sheaf $\mathcal{J}=\prod_{i}\mathcal{J}_{i}$ yields an admissible formal blow-up $\mathfrak{X}'\to\mathfrak{X}$ which factors through an admissible formal blow-up $\pi_{i}: \mathfrak{X}'\to \mathfrak{X}_{i}$ for each $i$ (see \cite{FK}, Exercise~II.1.1). Then, for each index $i$, the open subset $U_{i}\subseteq X$ is the pre-image of the quasi-compact open subset $\pi_{i}^{-1}(\mathfrak{U}_{i})$ of $\mathfrak{X}'$ under the specialization morphism $\spc_{X,\mathfrak{X}'}: X\to \mathfrak{X}'$.\end{proof}
There is also the following useful consequence of Lemma~\ref{Admissible formal blow-ups and generic fiber}.
\begin{lemma}\label{Empty generic fiber}Let $R$ be a complete adic ring with ideal of definition generated by a single non-zero-divisor $\varpi$ and suppose that the Tate ring $R[\varpi^{-1}]$ is sheafy. Let $\mathfrak{X}$ be a quasi-compact locally rig-sheafy adic formal $R$-scheme. The following are equivalent: 
\begin{enumerate}[(1)]\item The generic fiber $X$ of $\mathfrak{X}$ over $\Spa(R[\varpi^{-1}], \overline{R})$ is empty. \item There exists an integer $n\geq1$ such that $\varpi^{n}\mathcal{O}_{\mathfrak{X}}=0$. \item For some admissible formal blow-up $\mathfrak{X}'\to\mathfrak{X}$ of $\mathfrak{X}$ there exists an integer $n\geq1$ such that $\varpi^{n}\mathcal{O}_{\mathfrak{X}'}=0$. \item For every admissible formal blow-up $\mathfrak{X}'\to\mathfrak{X}$ of $\mathfrak{X}$ there exists an integer $n\geq1$ such that $\varpi^{n}\mathcal{O}_{\mathfrak{X}'}=0$. \end{enumerate}\end{lemma}
\begin{proof}By Lemma~\ref{Admissible formal blow-ups and generic fiber}, every admissible formal blow-up $\mathfrak{X}'\to\mathfrak{X}$ satisfies ${\mathfrak{X}'}_{\eta}^{\ad}\cong X$, so it suffices to prove the equivalence of (1) and (2). Since $\mathfrak{X}$ is assumed to be quasi-compact, we may work locally and thus we may assume that $\mathfrak{X}$ is affine and rig-sheafy. But then $X=\Spa(\mathcal{O}_{\mathfrak{X}}(\mathfrak{X})[\varpi^{-1}], \overline{\mathcal{O}_{\mathfrak{X}}(\mathfrak{X})})$, where $\overline{\mathcal{O}_{\mathfrak{X}}(\mathfrak{X})}$ is the integral closure in $\mathcal{O}_{\mathfrak{X}}(\mathfrak{X})[\varpi^{-1}]$ of the image of $\mathcal{O}_{\mathfrak{X}}(\mathfrak{X})$ in $\mathcal{O}_{\mathfrak{X}}(\mathfrak{X})[\varpi^{-1}]$. Therefore, $X$ is empty if and only if $\mathcal{O}_{\mathfrak{X}}(\mathfrak{X})[\varpi^{-1}]=0$, which is the case if and only if $\varpi^{n}\mathcal{O}_{\mathfrak{X}}=0$ for some $n$. \end{proof}
By \cite{FK}, Ch.~II, Proposition~1.1.10, the composition of two admissible formal blow-ups between quasi-compact quasi-separated adic formal schemes of finite ideal type is an admissible formal blow-up. By \cite{FK}, Ch.~II, Corollary~1.3.2, the category $\Bl(\mathfrak{X})$ of admissible formal blow-ups of a quasi-compact quasi-separated $\mathfrak{X}$ is cofiltered and essentially small. In particular, we obtain a morphism of locally v-ringed spaces \begin{equation*}\spc_{X}: (\mathfrak{X}_{\eta}^{\ad}, \mathcal{O}_{\mathfrak{X}_{\eta}^{\ad}}^{+})\to\varprojlim_{\mathfrak{X}'\in \Bl(\mathfrak{X})} (\mathfrak{X}', \mathcal{O}_{\mathfrak{X}'})\end{equation*}given by \begin{equation*}\spc_{X}(x)=(\spc_{X,\mathfrak{X}'}(x))_{\mathfrak{X}'\in\Bl(\mathfrak{X})}.\end{equation*}In fact, we can prove the following global variant of a theorem of Bhatt (\cite{BhattNotes}, Theorem~8.1.2); for the classical case of a quasi-compact quasi-separated adic space of finite type over a nonarchimedean field, see, for example, \cite{Scholze}, Theorem~2.22. Since our setting is slightly different from that in \cite{BhattNotes} (if $\mathfrak{X}=\Spf(A^{+})$ for some Tate Huber pair $(A, A^{+})$, we work with admissible formal blow-ups of $\mathfrak{X}$ instead of $\Spec(A)$-modifications of $\Spec(A^{+})$), we spell out the proof in some detail.
\begin{thm}\label{Generic fiber via formal blow-ups}Let $R$ be a complete adic ring with an ideal of definition generated by a single non-zero-divisor $\varpi\in R$ and suppose that the Tate ring $R[\varpi^{-1}]$ is sheafy. For any locally rig-sheafy, quasi-compact quasi-separated adic formal $R$-scheme $\mathfrak{X}$, the above canonical morphism is an isomorphism of locally v-ringed spaces\begin{equation*}\spc_{\mathfrak{X}_{\eta}^{\ad}}: (\mathfrak{X}_{\eta}^{\ad}, \mathcal{O}_{\mathfrak{X}_{\eta}^{\ad}}^{+})\tilde{\rightarrow}\varprojlim_{\mathfrak{X}'\in \Bl(\mathfrak{X})} (\mathfrak{X}', \mathcal{O}_{\mathfrak{X}'}).\end{equation*}\end{thm}
\begin{proof}If $\varpi^{n}\mathcal{O}_{\mathfrak{X}}=0$ for some integer $n\geq1$, then both the source and the target of $\spc_{\mathfrak{X}_{\eta}^{\ad}}$ are empty (indeed, in this case the admissible formal blow-up of $\mathfrak{X}$ in the admissible ideal $\varpi\mathcal{O}_{\mathfrak{X}}$ is the empty space), so there is nothing to prove in this case. Hence we may assume that $\varpi^{n}\mathcal{O}_{\mathfrak{X}}\neq 0$ for all $n\geq1$. Note that, by Lemma~\ref{Empty generic fiber}, under this assumption every admissible formal blow-up $\mathfrak{X}'$ of $\mathfrak{X}$ also satisfies $\varpi^{n}\mathcal{O}_{\mathfrak{X}'}\neq 0$ for all $n\geq1$.    

By Lemma~\ref{Admissible formal blow-up and affine opens}, every admissible formal blow-up $\mathfrak{U}'\to\mathfrak{U}$ of a rig-sheafy affine open subset $\mathfrak{U}$ of $\mathfrak{X}$ extends to an admissible formal blow-up $\mathfrak{X}'\to\mathfrak{X}$ of $\mathfrak{X}$. By combining Lemma~\ref{Description of inverse images} and Lemma~\ref{Admissible formal blow-ups and generic fiber} we see that this admissible formal blow-up satisfies\begin{equation*}\spc_{\mathfrak{X}_{\eta}^{\ad},\mathfrak{X}'}\vert_{\mathfrak{U}_{\eta}^{\ad}}=\spc_{\mathfrak{X}_{\eta}^{\ad},\mathfrak{X}'}\vert_{{\mathfrak{U}'}_{\eta}^{\ad}}=\spc_{{\mathfrak{U}'}_{\eta}^{\ad},\mathfrak{U}'}=\spc_{\mathfrak{U}_{\eta}^{\ad},\mathfrak{U}'}.\end{equation*}Hence, for any rig-sheafy affine open subset $\mathfrak{U}\subseteq \mathfrak{X}$, we have a commutative diagram of morphisms of locally v-ringed spaces \begin{center}\begin{tikzcd} (\mathfrak{X}_{\eta}^{\ad}, \mathcal{O}_{\mathfrak{X}_{\eta}^{\ad}}^{+}) \arrow{r}{\spc_{\mathfrak{X}_{\eta}^{\ad}}} & \varprojlim_{\mathfrak{X}'\in\Bl(\mathfrak{X})}(\mathfrak{X}', \mathcal{O}_{\mathfrak{X}'}) \\ (\mathfrak{U}_{\eta}^{\ad}, \mathcal{O}_{\mathfrak{U}_{\eta}^{\ad}}^{+}) \arrow[hook]{u} \arrow{r}{\spc_{\mathfrak{U}_{\eta}^{\ad}}} & \varprojlim_{\mathfrak{U}'\in\Bl(\mathfrak{U})}(\mathfrak{U}', \mathcal{O}_{\mathfrak{U}'}). \arrow[hook]{u} \end{tikzcd}\end{center}Consequently, it suffices to treat the case when $\mathfrak{X}=\Spf(A)$ is rig-sheafy and affine. 

In this situation the adic analytic generic fiber $X$ of $\mathfrak{X}$ has the form $X=\Spa(A[\varpi^{-1}], \overline{A})$, where, as usual, the bar denotes integral closure inside the Tate ring $A[\varpi^{-1}]$ of the image of $A\to A[\varpi^{-1}]$. We want to define an inverse $\underline{x}\mapsto v_{\underline{x}}$ to the map $\spc_{X}: \vert X\vert\to \varprojlim_{\mathfrak{X}'\in\Bl(\mathfrak{X})}\vert\mathfrak{X}'\vert$. To this end, we start with an arbitrary element \begin{equation*}\underline{x}=(x_{\mathfrak{X}'})_{\mathfrak{X}'\in\Bl(\mathfrak{X})}\in \varprojlim_{\mathfrak{X}'\in\Bl(\mathfrak{X})}\vert \mathfrak{X}'\vert \end{equation*}and consider the local ring \begin{equation*}A_{\underline{x}}=\varinjlim_{\mathfrak{X}'\in\Bl(\mathfrak{X})}\mathcal{O}_{\mathfrak{X}', x_{\mathfrak{X}'}}.\end{equation*}This ring is non-zero by the assumption that the generic fiber $\mathfrak{X}_{\eta}^{\ad}$ is not empty and Lemma~\ref{Admissible formal blow-ups and generic fiber}. Note that this local ring is $\varpi$-torsion-free. Indeed, for any $\mathfrak{X}'\in\Bl(\mathfrak{X})$ there exists an admissible formal blow-up $\mathfrak{X}''\to \mathfrak{X}'$ such that $\mathfrak{X}''$ is $\varpi$-torsion-free: Just let $\mathfrak{X}''\to\mathfrak{X}'$ be the admissible formal blow-up of $\mathfrak{X}'$ in the admissible ideal $\varpi\mathcal{O}_{\mathfrak{X}'}$ of $\mathfrak{X}'$. Note also that $\varpi$ is not a unit in the above local ring, since, if it was, then $\varpi$ would also be a unit in $\mathcal{O}_{\mathfrak{X}'}(\mathfrak{U})$ for some $\mathfrak{X}'\in\Bl(\mathfrak{X})$ and some $\mathfrak{U}\in\Aff_{\mathfrak{X}'}$, which is impossible due to $\varpi$-adic completeness.   

Following an argument of Bhatt (\cite{BhattNotes}, Proposition~8.1.3, and its proof), we first prove that the $\varpi$-adic Hausdorff quotient \begin{equation*}\mathcal{O}_{\underline{x}}=A_{\underline{x}}/\bigcap_{n}\varpi^{n}A_{\underline{x}}\end{equation*}is a valuation ring (note that $\mathcal{O}_{\underline{x}}$ is not the zero ring since otherwise $\varpi$ would be a unit in $A_{\underline{x}}$). Since valuation rings are the same as local Bézout domains, we split the proof of this assertion into two steps: Showing that all finitely generated ideals of $\mathcal{O}_{\underline{x}}$ are principal (i.e., that $\mathcal{O}_{\underline{x}}$ is a Bézout ring), and showing that $\mathcal{O}_{\underline{x}}$ is an integral domain. 

To see that $\mathcal{O}_{\underline{x}}$ is Bézout, let $\overline{f_1},\dots, \overline{f_r}$ be arbitrary non-zero elements of $\mathcal{O}_{\underline{x}}$. Let $f_1,\dots, f_r$ be elements of $A_{\underline{x}}$ lifting $\overline{f_1},\dots, \overline{f_r}$. Then there exists an integer $n\geq 1$ such that \begin{equation*}f_1,\dots, f_r\not\in \varpi^{n}A_{\underline{x}}.\end{equation*}We can choose an admissible formal blow-up $\mathfrak{X}'\to\mathfrak{X}$ and an affine open neighbourhood $\mathfrak{U}\subseteq \mathfrak{X}'$ such that (by the usual abuse of notation) $f_1,\dots, f_r\in\mathcal{O}_{\mathfrak{X}'}(\mathfrak{U})$. Consider the open ideal \begin{equation*}J_{n}=(f_1,\dots, f_r, \varpi^{n})_{\mathcal{O}_{\mathfrak{X}'}(\mathfrak{U})}.\end{equation*}Let $\mathfrak{U}''\to \mathfrak{U}$ be the admissible formal blow-up of $\mathfrak{U}$ in the admissible ideal sheaf defined by $J_{n}$. By Lemma~\ref{Admissible formal blow-up and affine opens} we can extend $J_{n}$ to an admissible ideal sheaf on $\mathfrak{X}'$ and the corresponding admissible formal blow-up $\mathfrak{X}''\to\mathfrak{X}'$ extends $\mathfrak{U}''\to \mathfrak{U}$. Consider $f_1,\dots, f_r, \varpi^{n}$ as elements of $\mathcal{O}_{\mathfrak{X}''}(\mathfrak{U}'')$. By \cite{FK}, Ch.~II, Proposition~1.1.8, $\mathfrak{U}''$ is precisely the pre-image of $\mathfrak{U}$ under $\mathfrak{X}''\to \mathfrak{X}'$. In particular, the lift $x_{\mathfrak{X}''}$ of $x_{\mathfrak{X}'}$ lies in $\mathfrak{U}''$. By Lemma~\ref{Local explicit description}, the formal scheme $\mathfrak{U}''$ has an affine open cover $(\Spf(B_{i}))_{i=0,\dots, r}$ by formal spectra of the completed affine blow-up algebras\begin{equation*}B_{i}=\mathcal{O}_{\mathfrak{X}'}(\mathfrak{U})\langle\frac{f_1,\dots, f_r, \varpi^{n}}{f_i}\rangle, \ i=1,\dots, r,\end{equation*}and \begin{equation*}B_{0}=\mathcal{O}_{\mathfrak{X}'}(\mathfrak{U})\langle\frac{f_1,\dots, f_r,\varpi^{n}}{\varpi^{n}}\rangle.\end{equation*}Choose an index $i$ such that \begin{equation*}x_{\mathfrak{X}''}\in \Spf(B_{i}).\end{equation*}If $i=0$, then $(f_1,\dots, f_r,\varpi^{n})_{B_{i}}=(\varpi^{n})_{B_{i}}$ and, a fortiori, $f_1,\dots, f_r\in \varpi^{n}A_{\underline{x}}$, a contradiction. Hence $i\geq 1$ and $(f_1,\dots, f_r,\varpi^{n})_{B_{i}}=(f_{i})_{B_{i}}$ for this $i$. In particular, $(f_1,\dots, f_r)_{B_{i}}=(f_{i})_{B_{i}}$ for this $i$ (for later use, we remark that in this case $f_{i}$ is not a zero-divisor in $A_{\underline{x}}=\varinjlim_{\mathfrak{X}'\in\Bl(\mathfrak{X})}\mathcal{O}_{\mathfrak{X}',x_{\mathfrak{X}'}}$, since $\varpi^{n}\in (f_{i})_{B_{i}}$ and $A_{\underline{x}}$ is $\varpi$-torsion-free). We conclude that the ideal of $\mathcal{O}_{\underline{x}}$ generated by the images $\overline{f_1},\dots, \overline{f_r}$ of $f_1,\dots, f_r$ is a principal ideal; since $\overline{f_1},\dots, \overline{f_r}$ were arbitrary, this proves that $\mathcal{O}_{\underline{x}}$ is a Bézout ring. 

On the other hand, we see that in the situation of the above paragraph $\varpi^{n}$ belongs to the ideal of $B_{i}$ generated by $f_1,\dots, f_r$ and hence that $\varpi^{n}$ belongs to the ideal of $\mathcal{O}_{\underline{x}}$ generated by $\overline{f_1},\dots, \overline{f_r}$. In particular, if $\overline{f}$ is an arbitrary element of $\mathcal{O}_{\underline{x}}$ such that $\overline{f}\not\in \varpi^{n}\mathcal{O}_{\underline{x}}$, then \begin{equation*}\varpi^{n}\in (\overline{f})_{\mathcal{O}_{\underline{x}}}.\end{equation*}Thus for every non-zero $\overline{f}\in \mathcal{O}_{\underline{x}}$, we can find an integer $n\geq1$ with $\varpi^{n}\in (\overline{f})_{\mathcal{O}_{\underline{x}}}$. This allows us to prove that $\mathcal{O}_{\underline{x}}$ is an integral domain as follows. Suppose that $\overline{f}$ and $\overline{g}$ are non-zero elements of $\mathcal{O}_{\underline{x}}$ whose product is zero. Choosing integers $n, m\geq 1$ such that $\varpi^{n}\in (\overline{f})_{\mathcal{O}_{\underline{x}}}$ and $\varpi^{m}\in (\overline{g})_{\mathcal{O}_{\underline{x}}}$, we see that $\varpi^{n+m}\in (\overline{f}\overline{g})_{\mathcal{O}_{\underline{x}}}=0$. Since $\mathcal{O}_{\underline{x}}$ is non-zero, this implies that the local ring $A_{\underline{x}}$ has $\varpi$-torsion, in contradiction to what we have seen above.

Thus we have proved that $\mathcal{O}_{\underline{x}}$ is a valuation ring, and the above argument also showed that every non-zero element of $\mathcal{O}_{\underline{x}}$ divides some power $\varpi^{n}$ of $\varpi$. Therefore, the Tate ring $K_{\underline{x}}=\mathcal{O}_{\underline{x}}[\varpi^{-1}]$ is equal to the fraction field of $\mathcal{O}_{\underline{x}}$ and $\mathcal{O}_{\underline{x}}$ determines a valuation on the field $K_{\underline{x}}$. Note also that we have a canonical continuous map of Tate rings $A[\varpi^{-1}]\to K_{\underline{x}}$ induced by the canonical map \begin{equation*}A\to \varinjlim_{\mathfrak{X}'\in\Bl(\mathfrak{X})}\mathcal{O}_{\mathfrak{X}', x_{\mathfrak{X}'}}\to \mathcal{O}_{\underline{x}}.\end{equation*}Hence we can define an element $v_{\underline{x}}\in X=\Spa(A[\varpi^{-1}], \overline{A})$ from $\underline{x}$ by letting $v_{\underline{x}}$ be the pullback along $A[\varpi^{-1}]\to K_{\underline{x}}$ of the valuation determined by the valuation ring $\mathcal{O}_{\underline{x}}$. 

We want to show that the maps $\spc_{X}$ and $\underline{x}\mapsto v_{\underline{x}}$ are inverse to each other. To this end, let $\underline{x}=(x_{\mathfrak{X}'})_{\mathfrak{X}'}\in \varprojlim_{\mathfrak{X}'}\vert\mathfrak{X}'\vert$ be arbitrary. Since, by definition of $\mathcal{O}_{\underline{x}}$, for every admissible formal blow-up $\mathfrak{X}'\in\Bl(\mathfrak{X})$ the natural map $\mathcal{O}_{\mathfrak{X}',x_{\mathfrak{X}'}}\to \mathcal{O}_{\underline{x}}$ is a local homomorphism of local rings and since the maximal ideal of the valuation ring $\mathcal{O}_{\underline{x}}$ is the subset where the valuation is $<1$, the image $\spc_{X,\mathfrak{X}'}(v_{\underline{x}})$ of $v_{\underline{x}}$ in $\mathfrak{X}'\in\Bl(\mathfrak{X})$ is precisely $x_{\mathfrak{X}'}$. To prove the converse, let $x\in X$ and consider the corresponding images $\spc_{X,\mathfrak{X}'}(x)$ for $\mathfrak{X}'\in\Bl(\mathfrak{X})$. Note that the local ring $A_{\spc_{X}(x)}$ can be written as\begin{equation*}A_{\spc_{X}(x)}=\varinjlim_{\mathfrak{X}'\in\Bl(\mathfrak{X})}\varinjlim_{\spc_{X,\mathfrak{X}'}(x)\in \mathfrak{U}}\mathcal{O}_{\mathfrak{X}'}(\mathfrak{U}),\end{equation*}where the second inductive limit is taken over quasi-compact open neighbourhoods $\mathfrak{U}$ of $\spc_{X,\mathfrak{X}'}(x)$ in $\mathfrak{X}'$. On the other hand, using Lemma~\ref{Description of inverse images} and Lemma~\ref{Open covers and formal models}, we see that for every quasi-compact open neighbourhood $U$ of $x$ in $X$ there exists an admissible formal blow-up $\mathfrak{X}'\in\Bl(\mathfrak{X})$ and a quasi-compact open neighbourhood $\mathfrak{U}$ of $\spc_{X,\mathfrak{X}'}(x)$ such that $U=\spc_{X,\mathfrak{X}'}^{-1}(\mathfrak{U})=\mathfrak{U}_{\eta}^{\ad}$. In particular, \begin{equation*}\mathcal{O}_{X}(U)=\mathcal{O}_{U}(U)=\mathcal{O}_{\mathfrak{U}_{\eta}^{\ad}}(\mathfrak{U}_{\eta}^{\ad})=\mathcal{O}_{\mathfrak{U}}(\mathfrak{U})[\varpi^{-1}]=\mathcal{O}_{\mathfrak{X}'}(\mathfrak{U})[\varpi^{-1}].\end{equation*}Conversely, for every quasi-compact open neighbourhood $\mathfrak{U}$ of $\spc_{X,\mathfrak{X}'}(x)$ in $\mathfrak{X}'\in\Bl(\mathfrak{X})$ the pre-image $\mathfrak{U}_{\eta}^{\ad}$ of $\mathfrak{U}$ in $X$ is a quasi-compact open neighbourhood of $x$ in $X$. Therefore,\begin{equation*}A_{\spc_{X}(x)}[\varpi^{-1}]=(\varinjlim_{\mathfrak{X}'\in\Bl(\mathfrak{X})}\varinjlim_{\spc_{X,\mathfrak{X}'}(x)\in\mathfrak{U}}\mathcal{O}_{\mathfrak{X}'}(\mathfrak{U}))[\varpi^{-1}]=\varinjlim_{x\in U}\mathcal{O}_{X}(U)=\mathcal{O}_{X,x}.\end{equation*}In this way we see that the canonical surjective ring map $A_{\spc_{X}(x)}\to \mathcal{O}_{\spc_{X}(x)}$ induces a surjective ring map $\mathcal{O}_{X,x}\to K_{\spc_{X}(x)}$, so the field $K_{\spc_{X}(x)}$ is nothing else but the residue field $\kappa(x)$ of the local ring $\mathcal{O}_{X,x}$ of $X$ in $x$. By the same argument, we have \begin{equation*}\varinjlim_{\mathfrak{X}'\in \Bl(\mathfrak{X})}\varinjlim_{\spc_{X,\mathfrak{X}'}(x)\in\mathfrak{U}}\overline{\mathcal{O}_{\mathfrak{X}'}(\mathfrak{U})}=\varinjlim_{\mathfrak{X}'\in\Bl(\mathfrak{X})}\varinjlim_{\spc_{X,\mathfrak{X}'}(x)\in\mathfrak{U}}\mathcal{O}_{X}^{+}(\mathfrak{U}_{\eta}^{\ad})=\varinjlim_{x\in U}\mathcal{O}_{X}^{+}(U)=\mathcal{O}_{X,x}^{+}.\end{equation*}We obtain an integral ring map \begin{equation*}A_{\spc_{X}(x)}=\varinjlim_{\mathfrak{X}'\in\Bl(\mathfrak{X})}\varinjlim_{\spc_{X,\mathfrak{X}'}(x)\in\mathfrak{U}}\mathcal{O}_{\mathfrak{X}'}(\mathfrak{U})\to \mathcal{O}_{X,x}^{+}\end{equation*}which induces an isomorphism upon inverting $\varpi$, and this integral map is injective since its source is $\varpi$-torsion-free. 

We claim that $A_{\spc_{X}(x)}$ is actually integrally closed in $\mathcal{O}_{X,x}=A_{\spc_{X}(x)}[\varpi^{-1}]$ and thus \begin{equation*}A_{\spc_{X}(x)}=\mathcal{O}_{X,x}^{+}\end{equation*}as subrings of $\mathcal{O}_{X,x}$. For this we proceed as in the proof of \cite{FK}, Ch.~0, Proposition~8.7.5(2). First observe that, while proving that $\mathcal{O}_{\underline{x}}$ is a Bézout ring, we actually showed that the stalk \begin{equation*}A_{\spc_{X}(x)}=(\varinjlim_{\mathfrak{X}'\in\Bl(\mathfrak{X})}\mathcal{O}_{\mathfrak{X}'})_{\spc_{X}(x)}=\varinjlim_{\mathfrak{X}'\in\Bl(\mathfrak{X})}\mathcal{O}_{\mathfrak{X}',\spc_{X,\mathfrak{X}'}(x)}\end{equation*}is a $\varpi$-valuative local ring in the sense of Fujiwara and Kato (\cite{FK}, Ch.~0, Def.~8.7.1), i.e., every finitely generated ideal of $A_{\spc_{X}(x)}$ which contains a power of $\varpi$ is generated by a single non-zero-divisor (the same assertion was proved, in slightly different language, by Fujiwara and Kato, see \cite{FK}, Ch.~II, Proposition~3.2.6). Let $c\in A_{\spc_{X}(x)}$ be an element such that $t=\frac{c}{\varpi^{m}}\in\mathcal{O}_{X,x}$ is integral over $A_{\spc_{X}(x)}$ for some integer $m\geq1$. Consider an equation of integral dependence\begin{equation*}f(t)=t^{n}+a_{1}t^{n-1}+\dots+a_{n-1}t+a_{n}=0\end{equation*}of $t$ over $A_{\spc_{X}(x)}$. Since $A_{\spc_{X}(x)}$ is a $\varpi$-valuative local ring, the ideal $(\varpi^{m}, c)_{A_{\spc_{X}(x)}}$ is generated by a single non-zero-divisor $d\in A_{\spc_{X}(x)}$. Since $d$ is a non-zero-divisor, there exist unique elements $\frac{\varpi^{m}}{d}$, $\frac{c}{d}$ of $A_{\spc_{X}(x)}$ such that $d\frac{\varpi^{m}}{d}=\varpi^{m}$ and $d\frac{c}{d}=c$. Then \begin{equation*}(\frac{\varpi^{m}}{d}, \frac{c}{d})_{A_{\spc_{X}(x)}}=(1).\end{equation*}We claim that $\frac{\varpi^{m}}{d}$ is invertible in $A_{\spc_{X}(x)}$, in which case, viewing $t$ as an element of the localization $\mathcal{O}_{X,x}[\frac{1}{d}]$, \begin{equation*}t=\frac{c/d}{\varpi^{m}/d}=(\frac{\varpi^{m}}{d})^{-1}\frac{c}{d}\in A_{\spc_{X}(x)}\end{equation*}and we are done. By way of contradiction, assume that $\frac{\varpi^{m}}{d}$ belongs to the maximal ideal of $A_{\spc_{X}(x)}$. Then the relation $1\in (\frac{\varpi^{m}}{d}, \frac{c}{d})_{A_{\spc_{X}(x)}}$ entails that $\frac{c}{d}$ is a unit in $A_{\spc_{X}(x)}$. But, on the other hand, \begin{equation*}(\frac{\varpi^{m}}{d})^{n}f(\frac{c/d}{\varpi^{m}/d})=(\frac{\varpi^{m}}{d})^{n}f(t)=0,\end{equation*}so $\frac{c^{n}}{d^{n}}\in (\frac{\varpi^{m}}{d})_{A_{\spc_{X}(x)}}$ and thus $\frac{c}{d}$ cannot be a unit, a contradiction. This concludes the proof of the claim that $A_{\spc_{X}(x)}=\mathcal{O}_{X,x}^{+}$.   

Now, $\mathcal{O}_{\spc_{X}(x)}$ is equal to the image of $A_{\spc_{X}(x)}=\mathcal{O}_{X,x}^{+}\subseteq \mathcal{O}_{X,x}$ inside the residue field $K_{\spc_{X}(x)}=\kappa(x)$ of $X$ in $x$. By \cite{ConradNotes}, Proposition~14.3.1, this image is exactly the valuation ring $\kappa(x)^{+}$. Putting everything together we obtain \begin{equation*}\mathcal{O}_{\spc_{X}(x)}=\kappa(x)^{+}.\end{equation*}Consequently, $v_{\spc_{X}(x)}=x$ as valuations on $A[\varpi^{-1}]$. This proves that the maps $\spc_{X}$ and $\underline{x}\mapsto v_{\underline{x}}$ are indeed inverse to each other. 

We now prove that the map $\spc_{X}$ is generalizing. Given $x\in X$ let $\underline{y}$ be a generization of $\spc_{X}(x)$ in $\varprojlim_{\mathfrak{X}'\in\Bl(\mathfrak{X})}\vert\mathfrak{X}'\vert$. Since we have shown that $\spc_{X}$ is bijective, there exists $y\in X$ such that \begin{equation*}\underline{y}=\spc_{X}(y).\end{equation*}Suppose that $U$ is a quasi-compact open subset of $X$ which contains $x$. By Lemma~\ref{Open covers and formal models} we find an admissible formal blow-up $\mathfrak{X}'\in \Bl(\mathfrak{X})$ and an open subset $\mathfrak{U}$ of $\mathfrak{X}'$ such that $U$ is the pre-image of $\mathfrak{U}$ in $X$. Since $U$ contains $x$, the open subset $\mathfrak{U}\subseteq \mathfrak{X}'$ contains $\spc_{X,\mathfrak{X}'}(x)$ and the pre-image of $\mathfrak{U}$ in $\varprojlim_{\mathfrak{X}'\in\Bl(\mathfrak{X})}\vert\mathfrak{X}'\vert$ contains $\spc_{X}(x)$. Since $\spc_{X}(y)$ is a generization of $\spc_{X}(x)$, this pre-image also contains $\spc_{X}(y)$. In particular, $\spc_{X,\mathfrak{X}'}(y)\in \mathfrak{U}$ and, consequently, $y\in U$. This shows that $\spc_{X}$ is a generalizing map. Since transition morphisms in $\Bl(\mathfrak{X})$ are morphisms of quasi-compact quasi-separated formal schemes, the underlying continuous maps of topological spaces are spectral maps between spectral spaces. Hence the inverse limit $\varprojlim_{\mathfrak{X}'\in\Bl(\mathfrak{X})}\vert\mathfrak{X}'\vert$ is a spectral space, by \cite{Stacks}, Tag~0A2Z. By Lemma~\ref{Spectral}, the specialization maps $\spc_{X,\mathfrak{X}'}: \vert X\vert\to \vert\mathfrak{X}'\vert$ are spectral for all $\mathfrak{X}'\in\Bl(\mathfrak{X})$, so the map to the inverse limit $\spc_{X}$ is also spectral. On the other hand, by \cite{Stacks}, Tag~09XU, any continuous and generalizing bijective spectral map between spectral spaces is a homeomorphism. It follows that the map $\spc_{X}$ is a homeomorphism on the underlying topological spaces. 

It only remains to prove that the morphism of sheaves of rings on $X$\begin{equation*}\spc_{X}^{\ast}(\varinjlim_{\mathfrak{X}'\in\Bl(\mathfrak{X})}\mathcal{O}_{\mathfrak{X}'})\to \mathcal{O}_{X}^{+}\end{equation*}induced by the map $\spc_{X}$ is an isomorphism. For this it suffices to prove that the morphism of sheaves induces isomorphisms on stalks. But we have already seen that, for every $x\in X$, the stalk\begin{equation*}(\varinjlim_{\mathfrak{X}'\in\Bl(\mathfrak{X})}\mathcal{O}_{\mathfrak{X}'})_{\spc_{X}(x)}=\varinjlim_{\mathfrak{X}'\in\Bl(\mathfrak{X})}\mathcal{O}_{\mathfrak{X}',\spc_{X,\mathfrak{X}'}(x)}=A_{\spc_{X}(x)}\end{equation*}is equal to $\mathcal{O}_{X,x}^{+}$.\end{proof}
\begin{rmk}\label{Generic fiber in the non-adic case}Let us caution the reader that the above theorem does not generalize to the case when $\mathfrak{X}$ is not adic over $\Spf(R)$. For a simple counterexample, consider the open unit disk over $\mathbb{Q}_{p}$, viewed as the adic analytic generic fiber $\mathfrak{X}_{\eta}^{\ad}$ (over $\Spa(\mathbb{Q}_{p}, \mathbb{Z}_{p})$) of the rig-sheafy affine adic formal scheme $\mathfrak{X}=\Spf(\mathbb{Z}_{p}[[T]])$ over $\Spf(\mathbb{Z}_{p})$, where $\mathbb{Z}_{p}[[T]]$ is equipped with the $(p, T)$-adic topology. In this case, $\mathfrak{X}_{\eta}^{\ad}$ is not quasi-compact while the inverse limit $\varprojlim_{\mathfrak{X}'\in\Bl(\mathfrak{X})}\vert\mathfrak{X}'\vert$ is a spectral space, by \cite{Stacks}, Tag~0A2Z.\end{rmk}
\begin{rmk}\label{Relation to Fujiwara-Kato}In the work \cite{FK} of Fujiwara and Kato coherent (formal) rigid spaces are defined as objects of the category of qcqs adic formal schemes of finite ideal type localized by admissible formal blow-ups. In particular, every adic formal scheme $\mathfrak{X}$ satisfying the assumptions of Theorem \ref{Generic fiber via formal blow-ups} gives rise to a coherent rigid space $\mathscr{X}$. The locally ringed space\begin{equation*}\varprojlim_{\mathfrak{X}'\in\Bl(\mathfrak{X})}\mathfrak{X}'\end{equation*}appearing on the right-hand side of the isomorphism in Theorem \ref{Generic fiber via formal blow-ups} is studied in Ch.~II.3 of loc.~cit. and is denoted there by $(\langle\mathscr{X}\rangle, \mathcal{O}_{\mathscr{X}}^{\textrm{int}})$, where $\mathscr{X}$ is the coherent (formal) rigid space defined by $\mathfrak{X}$ and where the topological space $\langle\mathscr{X}\rangle$ is called the Zariski-Riemann space of the rigid space $\mathscr{X}$. The sheaf $\mathcal{O}_{\mathscr{X}}^{\textrm{int}}$ is called the integral structure sheaf (\cite{FK}, Ch.~II, Proposition~3.2.15). Note that a key part of the proof of Theorem~\ref{Generic fiber via formal blow-ups} was to show (following an idea of Bhatt from \cite{BhattNotes}) that, for every $\underline{x}\in\langle\mathscr{X}\rangle$, the local ring $\mathcal{O}_{\mathscr{X},\underline{x}}^{\textrm{int}}$ is a $\varpi$-valuative ring in the sense of \cite{FK}, Ch.~0, Definition~8.7.1. This is also proved by Fujiwara and Kato, for any coherent rigid space $\mathscr{X}$, in \cite{FK}, Ch.~II, Proposition~3.2.6. 

Fujiwara and Kato also define a so-called rigid structure sheaf $\mathcal{O}_{\mathscr{X}}$ on the Zariski-Riemann space $\langle\mathscr{X}\rangle$ which in our case is just given by $\mathcal{O}_{\mathscr{X}}^{\textrm{int}}[\varpi^{-1}]$, i.e., it corresponds via the isomorphism in Theorem \ref{Generic fiber via formal blow-ups} to the usual structure sheaf of the adic space $\mathfrak{X}_{\eta}^{\ad}$. Fujiwara and Kato call the triple $(\langle\mathscr{X}\rangle, \mathcal{O}_{\mathscr{X}}^{\textrm{int}}, \mathcal{O}_{\mathscr{X}})$ the Zariski-Riemann triple associated with the coherent rigid space $\mathscr{X}$. When the rigid space $\mathscr{X}$ is locally universally Noetherian (a condition analogous to our locally Noetherian condition for adic spaces), it was already observed in Appendix~A to loc.~cit., Ch.~II, that the pair $(\mathscr{X}, \mathcal{O}_{\mathscr{X}})$ is an adic space (see \cite{FK}, Ch.~II, Theorem~A.5.1). Theorem \ref{Generic fiber via formal blow-ups} can be thought of as extending this result to coherent rigid spaces $\mathscr{X}$ which arise from a locally rig-sheafy qcqs adic formal $R$-scheme (those qcqs adic formal $R$-schemes whose associated rigid spaces satisfy the locally universally Noetherian assumption of Fujiwara and Kato are automatically locally rig-sheafy, by the result of Zavyalov discussed at the beginning of Section~\ref{sec:formal schemes}). As another consequence of Theorem~\ref{Generic fiber via formal blow-ups}, our specialization map $\spc_{X,\mathfrak{X}}$ coincides with the specialization map $\spc_{\mathfrak{X}}: \langle\mathscr{X}\rangle\to\mathfrak{X}$ on the Zariski-Riemann space $\langle\mathscr{X}\rangle$ which is defined in \cite{FK}, Ch.~II, 3.1(a). This means, in particular, that the topological results of \cite{FK}, Ch.~II.4, apply to the underlying topological space of any adic space over $\Spa(R[\varpi^{-1}], \overline{R})$ which admits a qcqs formal $R$-model.\end{rmk}   
We record the following immediate consequence of Theorem~\ref{Generic fiber via formal blow-ups}.
\begin{cor}\label{Center maps are surjective}Fix a complete adic ring $R$ with ideal of definition generated by a single non-zero-divisor $\varpi$ and suppose that the Tate ring $R[\varpi^{-1}]$ is sheafy. For any locally rig-sheafy, quasi-compact quasi-separated adic formal $R$-scheme $\mathfrak{X}$ and any $\varpi$-torsion-free admissible formal blow-up $\mathfrak{X}'\in\Bl(\mathfrak{X})$, the specialization map $\spc_{\mathfrak{X}_{\eta}^{\ad},\mathfrak{X}'}: \mathfrak{X}_{\eta}^{\ad}\to \mathfrak{X}'$ is surjective.\end{cor}
\begin{proof}Note that $\varpi$-torsion-free formal schemes in $\Bl(\mathfrak{X})$ form a cofinal system in the cofiltered category $\Bl(\mathfrak{X})$. In view of Theorem \ref{Generic fiber via formal blow-ups}, it suffices to prove that every admissible formal blow-up $\mathfrak{X}'\to \mathfrak{X}$ of a locally rig-sheafy, quasi-compact quasi-separated, $\varpi$-torsion-free adic formal $R$-scheme $\mathfrak{X}$ is surjective. For this we may assume that $\mathfrak{X}=\Spf(A)$ for some rig-sheafy $\varpi$-torsion-free and $\varpi$-adically complete $R$-algebra $A$. In this case the admissible formal blow-up $\mathfrak{X}'\to \mathfrak{X}$ in a finitely generated open ideal $J$ of $A$ is the completion of the scheme-theoretic blow-up $\pi_{J}: X'\to \Spec(A)$ of $\Spec(A)$ in $J$. The blow-up $\pi_{J}$ restricts to an isomorphism of schemes $X'\setminus\mathcal{V}(J\mathcal{O}_{X'})\tilde{\to}\Spec(A)\setminus\mathcal{V}(J)$. In particular, $\pi_{J}$ contains the basic open subset $\Spec(A[\varpi^{-1}])=\Spec(A)\setminus\mathcal{V}(\varpi)$ in its image. Since $A$ is $\varpi$-torsion-free, $\Spec(A[\varpi^{-1}])$ is dense in $\Spec(A)$, so $\pi_{J}$ has dense image. Since $\pi_{J}$ is proper, this means that $\pi_{J}$ is surjective. But then also the base change $\mathfrak{X}'_{0}=X'\times_{\Spec(A)}\Spec(A/(\varpi))\to \Spec(A/(\varpi))=\mathfrak{X}_{0}$ is surjective, so $\mathfrak{X}'\to \mathfrak{X}$ is surjective.\end{proof}
\begin{rmk}\label{History of the specialization map}The above corollary is not essentially new, with variants and special cases scattered across the literature on nonarchimedean analytic geometry. In the special case $\mathfrak{X}=\Spf(A^{\circ})$, where $A$ is an affinoid algebra in the sense of Tate over a nonarchimedean field $K$, it was already proved by Tate in \cite{Tate}, Theorem~6.4, that the restriction of $\spc_{X,\mathfrak{X}}$ to the set of rigid points $\Sp(A)$ of $\Spa(A, A^{\circ})$ contains every closed point of $\Spf(A^{\circ})$ in its image. In \cite{Berkovich}, Proposition~2.4.4, Berkovich proved that in this case also every generic point of $\mathfrak{X}=\Spf(A^{\circ})$ is in the image of $\spc_{X,\mathfrak{X}}$ (more precisely, is the image under $\spc_{X,\mathfrak{X}}$ of a rank $1$ point of $X$). For $X$ a not necessarily affinoid rigid space over $K$ and $\mathfrak{X}$ a flat formal $K^{\circ}$-model topologically of finite type (equivalently, topologically of finite presentation), surjectivity of $\spc_{X,\mathfrak{X}}$ onto the set of closed points of $\mathfrak{X}$ appears as Proposition~3.5 in \cite{BL1} and as Proposition~1.1.5 in Berthelot's work \cite{Berthelot96}. Finally, when $X=\Spa(A, A^{+})$ and $\mathfrak{X}'=\mathfrak{X}=\Spf(A^{+})$ for some Tate Huber pair $(A, A^{+})$ the corollary is an immediate consequence of \cite{BhattNotes}, Theorem~8.1.2 (of which our Theorem~\ref{Generic fiber via formal blow-ups} is a global variant), as was observed, for example, in \cite{Gleason25}, Proposition~4.2, and in Proposition~2.13.6 of the lecture notes \cite{Fargues22} of Fargues. The general statement of Corollary~\ref{Center maps are surjective} can be deduced from this affinoid case by using Lemma~\ref{Admissible formal blow-up and affine opens} and Lemma~\ref{Description of inverse images} to reduce to the rig-sheafy, affine case and then using the fact that $\Spf(\overline{\mathcal{O}_{\mathfrak{X}}(\mathfrak{X})})\to \Spf(\mathcal{O}_{\mathfrak{X}}(\mathfrak{X}))$ is surjective and closed, since $\mathcal{O}_{\mathfrak{X}}(\mathfrak{X})\hookrightarrow \overline{\mathcal{O}_{\mathfrak{X}}(\mathfrak{X})}$ is integral.\end{rmk}
Combining Corollary~\ref{Center maps are surjective} with Lemma~\ref{Open covers and formal models} we obtain the following result.
\begin{cor}\label{Open covers and formal models 2}Let $R$ be a complete adic ring with ideal of definition generated by a single non-zero-divisor $\varpi$ and suppose that $R[\varpi^{-1}]$ is a sheafy Tate ring. Let $\mathfrak{X}$ be a locally rig-sheafy quasi-compact quasi-separated adic formal $R$-scheme and let $(U_{i})_{i}$ be a finite open cover of $X=\mathfrak{X}_{\eta}^{\ad}$ by quasi-compact open subsets. Then there exists a $\varpi$-torsion-free admissible formal blow-up $\mathfrak{X}'\to\mathfrak{X}$ of $\mathfrak{X}$ and a finite open cover $(\mathfrak{U}_{i})_{i}$ of $\mathfrak{X}'$ by quasi-compact open formal subschemes such that $\spc_{X,\mathfrak{X}'}^{-1}(\mathfrak{U}_{i})=U_{i}$ for all $i$. \end{cor}
\begin{rmk}We observe that the above corollary (along with Lemma~\ref{Local explicit description}) also yields a new proof of \cite{Kedlaya-Liu}, Lemma~2.4.19, for sheafy complete Tate rings.\end{rmk}
Let us also record a few other topological consequences of Theorem~\ref{Generic fiber via formal blow-ups}.
\begin{cor}\label{The specialization map is closed}For every quasi-compact quasi-separated locally rig-sheafy adic formal $R$-scheme $\mathfrak{X}$ with adic analytic generic fiber $X$ over $(R, \varpi)$ the specialization map $\spc_{X,\mathfrak{X}}: \vert X\vert\to \vert\mathfrak{X}\vert$ is closed.\end{cor}
\begin{proof}Follows from Theorem~\ref{Generic fiber via formal blow-ups} and the fact that admissible formal blow-ups are proper and, in particular, closed. \end{proof}
\begin{cor}\label{Connectedness and generic fiber}Let $\mathfrak{X}$ be a $\varpi$-torsion-free quasi-compact quasi-separated locally rig-sheafy adic formal $R$-scheme and let $X$ be the generic fiber of $\mathfrak{X}$ over $(R, \varpi)$. If $X$ is connected, then $\mathfrak{X}$ is connected.\end{cor}
\begin{proof}If $X$ is connected, then $\mathfrak{X}$ is connected since $\spc_{X,\mathfrak{X}}$ is continuous and surjective, by Corollary~\ref{Center maps are surjective}. \end{proof}

\section{Normalized formal blow-ups}\label{sec:normalized formal blow-ups}

In this section we introduce and study the notion of normalized formal blow-ups, which is necessary for the construction of formal $R$-models of general uniform qcqs adic spaces over $\Spa(R[\varpi^{-1}], \overline{R})$. We begin by recalling the notion of a uniform analytic adic space. 
\begin{mydef}[Uniform analytic adic space]\label{Uniform adic space}An analytic adic space $X$ is called uniform if for every affinoid open subspace $U\subseteq X$ with $\mathcal{O}_{X}(U)$ a Tate ring, the Tate ring $\mathcal{O}_{X}(U)$ is stably uniform.\end{mydef}
This is a local notion.
\begin{lemma}\label{Uniformity is a local notion}If $X$ is an analytic adic space and there exists an affinoid open cover $(U_{i})_{i}$ of $X$ such that $\mathcal{O}_{X}(U_{i})$ is a stably uniform Tate ring for all $i$, then $X$ is uniform.\end{lemma}
\begin{proof}It suffices to prove that for every affinoid open subset $U\subseteq X$ with $\mathcal{O}_{X}(U)$ a Tate ring, the Tate ring $\mathcal{O}_{X}(U)$ is uniform (in this case, it follows that $\mathcal{O}_{X}(U)$ is also stably uniform by applying the same conclusion to rational subsets of $U$). Let $U$ be an arbitrary affinoid open subspace of $X$ such that $\mathcal{O}_{X}(U)$ is a Tate ring. For every $i$, we can cover $U_{i}\cap U$ by rational subsets $V_{ij}$ of $U_{i}$. Since $\mathcal{O}_{X}(U_{i})$ is stably uniform, each of its rational localizations $\mathcal{O}_{X}(V_{ij})$ is uniform. Since $U$ is quasi-compact, it is covered by finitely many of the affinoid open subsets $V_{ij}$. Since $\mathcal{O}_{X}$ is a sheaf of complete topological rings, the map \begin{equation*}\mathcal{O}_{X}(U)\hookrightarrow \prod_{i, j}\mathcal{O}_{X}(V_{ij})\end{equation*}is injective and strict, i.e., the topology on $\mathcal{O}_{X}(U)$ coincides with the restriction along this map of the product topology on the (finite) product $\prod_{i, j}\mathcal{O}_{X}(V_{ij})$. Since $\mathcal{O}_{X}(U)^{\circ}$ maps into $\prod_{i, j}\mathcal{O}_{X}(V_{ij})^{\circ}$, it is clearly bounded with respect to the restriction of the product topology. Hence $\mathcal{O}_{X}(U)^{\circ}$ is also bounded with respect to the original topology of $\mathcal{O}_{X}(U)$. That is, $\mathcal{O}_{X}(U)$ is uniform. \end{proof}
Next, we recall the following special case of \cite{FK}, Ch.~I, Def.~3.1.3.
\begin{mydef}[Fujiwara-Kato]\label{Adically quasi-coherent sheaf}Let $\mathfrak{X}$ be an adic formal scheme of finite ideal type admitting an ideal of definition of finite type $\mathcal{I}$. A sheaf of $\mathcal{O}_{\mathfrak{X}}$-modules $\mathcal{F}$ is called \textit{adically quasi-coherent}~if it is complete, i.e.,\begin{equation*}\mathcal{F}=\varprojlim_{k}\mathcal{F}/\mathcal{I}^{k}\mathcal{F},\end{equation*}and if the sheaf of $\mathcal{O}_{\mathfrak{X}}/\mathcal{I}^{k+1}\mathcal{O}_{\mathfrak{X}}$-modules $\mathcal{F}/\mathcal{I}^{k+1}\mathcal{F}$ is a quasi-coherent sheaf on the scheme $\mathfrak{X}_{k}$ for all $k\geq 0$.\end{mydef}
\begin{rmk}\label{Completeness does not depend on an ideal of definition}By \cite{FK}, Ch.~I, \S3.1, discussion preceding Definition~3.1.1, and by loc.~cit., Ch.~I, Lemma 3.1.2, the conditions defining the notion of an adically quasi-coherent sheaf do not depend on the choice of an ideal of definition of finite type.\end{rmk}
For an adically quasi-coherent sheaf of algebras $\mathcal{A}$ on an adic formal scheme of finite ideal type $\mathfrak{S}$, there is a notion of relative formal spectrum $\Spf(\mathcal{A})\to \mathfrak{S}$ similar to the notion of relative spectrum of a quasi-coherent algebra on a scheme. Moreover, similar to the scheme case, all affine morphisms of adic formal schemes of finite ideal type $f: \mathfrak{X}\to \mathfrak{S}$ can be characterized as morphisms of the form $\Spf(\mathcal{A})\to \mathfrak{S}$, namely, for any affine morphism $f: \mathfrak{X}\to \mathfrak{S}$ we can choose $\mathcal{A}=f_{\ast}\mathcal{O}_{\mathfrak{X}}$, see \cite{FK}, Ch.~I, Theorem~4.1.8.
\begin{prop}\label{Quasi-coherent algebras and formal models}Fix a complete adic ring $R$ with a non-zero-divisor $\varpi\in R$ generating an ideal of definition of $R$ and suppose that the Tate ring $R[\varpi^{-1}]$ is sheafy. Let $\mathfrak{S}$ be a locally rig-sheafy $\varpi$-torsion-free quasi-compact quasi-separated adic formal $R$-scheme whose adic analytic generic fiber $S=\mathfrak{S}_{\eta}^{\ad}$ over $\Spa(R[\varpi^{-1}], \overline{R})$ is a uniform adic space. For every admissible formal blow-up $\mathfrak{S}'\to \mathfrak{S}$ of $\mathfrak{S}$, the direct image sheaf $\spc_{S,\mathfrak{S}'\ast}\mathcal{O}_{S}^{+}$ is an adically quasi-coherent sheaf of algebras on $\mathfrak{S}'$.\end{prop}
\begin{proof}It suffices to prove, for any affine open cover $(\mathfrak{U}_{i})_{i}$ of $\mathfrak{S}'$ and $U_{i}={\mathfrak{U}_{i}}_{\eta}^{\ad}$, the sheaf\begin{equation*}\spc_{S,\mathfrak{S}'\ast}\mathcal{O}_{S}^{+}\vert_{\mathfrak{U}_{i}}=\spc_{U_{i},\mathfrak{U}_{i}\ast}\mathcal{O}_{U_{i}}^{+}\end{equation*}(where in the above equality we tacitly used Lemma~\ref{Description of inverse images}) is adically quasi-coherent for each $i$. Hence it suffices to prove the proposition in the case when $\mathfrak{S}'=\mathfrak{S}$, when $\mathfrak{S}$ is a rig-sheafy affine formal scheme (we use here that $\spc_{S,\mathfrak{S}'}^{-1}(\mathfrak{U}_{i})=U_{i}$, by Lemma~\ref{Description of inverse images}). Let $A=\mathcal{O}_{\mathfrak{S}}(\mathfrak{S})$. Since $S$ is uniform, $\mathcal{O}_{S}^{+}(S)$ is $\varpi$-adically complete. Thus it suffices to check that \begin{equation*}\mathcal{O}_{S}^{+}(S(\frac{1}{g}))=(\spc_{S,\mathfrak{S}\ast}\mathcal{O}_{S}^{+})(D(g))=\mathcal{O}_{S}^{+}(S)\widehat{\otimes}_{A}A\langle g^{-1}\rangle=\mathcal{O}_{S}^{+}(S)\langle g^{-1}\rangle\end{equation*}for every $g\in A$. Hence it suffices to prove that $\mathcal{O}_{S}^{+}(S)\langle g^{-1}\rangle$ is integrally closed inside $\mathcal{O}_{S}(S)\langle g^{-1}\rangle$. By \cite{Huber0}, Lemma~2.4.3(iv), this follows from $\mathcal{O}_{S}^{+}(S)[g^{-1}]$ being an integrally closed open subring of $\mathcal{O}_{S}(S)[g^{-1}]$. \end{proof}
In order to establish the existence of formal $R$-models for a general uniform quasi-compact quasi-separated adic space over $\Spa(R[\varpi^{-1}], \overline{R})$ we need the following notion of integrally closed models.
\begin{mydef}[Integrally closed formal models]A locally rig-sheafy $\varpi$-torsion-free adic formal $R$-scheme $\mathfrak{X}$ with adic analytic generic fiber $X$ over $\Spa(R[\varpi^{-1}], \overline{R})$ is said to be \textit{integrally closed inside its generic fiber $X$} if for every affine open subset $\mathfrak{V}$ of $\mathfrak{X}$ the ring $\mathcal{O}_{\mathfrak{X}}(\mathfrak{V})$ is integrally closed in the Tate ring $\mathcal{O}_{\mathfrak{X}}(\mathfrak{V})[\varpi^{-1}]$. In this case we also say that $\mathfrak{X}$ is an integrally closed formal $R$-model of the adic space $X$ over $\Spa(R[\varpi^{-1}], \overline{R})$. \end{mydef}
The first order of business is to verify that the above notion of integrally closed formal models is local.
\begin{lemma}\label{Integrally closed formal models}For a locally rig-sheafy $\varpi$-torsion-free adic formal $R$-scheme $\mathfrak{X}$ the following are equivalent: \begin{enumerate}[(1)]\item $\mathfrak{X}$ is integrally closed in its generic fiber. \item There exists an affine open cover $(\mathfrak{U}_{i})_{i}$ of $\mathfrak{X}$ such that $\mathcal{O}_{\mathfrak{X}}(\mathfrak{U}_{i})$ is integrally closed in $\mathcal{O}_{\mathfrak{X}}(\mathfrak{U}_{i})[\varpi^{-1}]$ for every $i$.\end{enumerate}\end{lemma}
\begin{proof}This reduces to proving that for a $\varpi$-adically complete, $\varpi$-torsion-free ring $A$ which is integrally closed in $A[\varpi^{-1}]$ and $g\in A$, the completed localization $A\langle g^{-1}\rangle$ is integrally closed in $A\langle g^{-1}\rangle[\varpi^{-1}]$, which is a consequence of \cite{Huber0}, Lemma~2.4.3(iv).\end{proof}
Using Proposition~\ref{Quasi-coherent algebras and formal models}, we can always construct an integrally closed formal $R$-model from an arbitrary one, at least for uniform adic spaces. This is a consequence of the following proposition, cf. also \cite{Pilloni-Stroh16}, Proposition~1.1, for a more classical result along these lines.
\begin{prop}\label{Normalization of a formal model}Fix a complete adic ring $R$ with a non-zero-divisor $\varpi\in R$ generating an ideal of definition of $R$ and suppose that the Tate ring $R[\varpi^{-1}]$ is sheafy. Let $\mathfrak{S}$ be a locally rig-sheafy $\varpi$-torsion-free quasi-compact quasi-separated adic formal $R$-scheme whose adic analytic generic fiber $S=\mathfrak{S}_{\eta}^{\ad}$ over $\Spa(R[\varpi^{-1}], \overline{R})$ is uniform. The affine morphism \begin{equation*}f_{0}: \mathfrak{X}=\Spf(\spc_{S,\mathfrak{S}\ast}\mathcal{O}_{S}^{+})\to \mathfrak{S}\end{equation*}satisfies $f_{0\eta}=\id$ and $\mathfrak{X}$ is an integrally closed formal $R$-model of $S$. Moreover, $f_{0}$ is the unique affine morphism of locally rig-sheafy adic formal $R$-schemes $\mathfrak{X}\to\mathfrak{S}$ with target $\mathfrak{S}$ such that $f_{0\eta}=\id$ and such that the source $\mathfrak{X}$ is $\varpi$-torsion-free and integrally closed in its generic fiber $S$.\end{prop}
\begin{proof}By Proposition~\ref{Quasi-coherent algebras and formal models}, $\spc_{S,\mathfrak{S}\ast}\mathcal{O}_{S}^{+}$ is an adically quasi-coherent algebra on $\mathfrak{S}$, so the relative formal spectrum $\mathfrak{X}$ is defined and the morphism $f_{0}$ is affine. Let $(\mathfrak{U}_{i})_{i}$ be an affine open cover of $\mathfrak{S}$ and let $U_{i}=\spc_{S,\mathfrak{S}}^{-1}(\mathfrak{U}_{i})=\mathfrak{U}_{i\eta}$, which is an affinoid open subspace of $S$, for every $i$. The fact that $f_{0\eta}=\id$ and that $\mathfrak{X}$ is integrally closed in $S$ follows from $\spc_{S,\mathfrak{S}\ast}\mathcal{O}_{S}^{+}(\mathfrak{U}_{i})=\mathcal{O}_{S}^{+}(U_{i})$ being integrally closed in $\mathcal{O}_{S}(U_{i})$ and a ring of definition of $\mathcal{O}_{S}(U_{i})$ (since $\mathcal{O}_{S}(U_{i})$ are uniform Tate rings). 

It remains to prove the uniqueness part of the proposition. To this end, let $g_{0}: \mathfrak{Z}\to \mathfrak{S}$ be another affine adic morphism of $\varpi$-torsion-free adic formal $R$-schemes such that $\mathfrak{Z}_{\eta}^{\ad}=S$, $g_{0\eta}=\id$ and $\mathfrak{Z}$ is integrally closed in its generic fiber. By \cite{FK}, Ch.~I, Theorem~4.1.8, $g_{0\ast}\mathcal{O}_{\mathfrak{Z}}$ is an adically quasi-coherent $\mathcal{O}_{\mathfrak{S}}$-algebra, $\mathfrak{Z}=\Spf(g_{0\ast}\mathcal{O}_{\mathfrak{Z}})$ and $g_{0}$ is the canonical morphism $\Spf(g_{0\ast}\mathcal{O}_{\mathfrak{Z}})\to \mathfrak{S}$. By loc.~cit., Proposition~4.1.10, the pre-image $g_{0}^{-1}(\mathfrak{U})$ of every affine open subset $\mathfrak{U}\subseteq\mathfrak{S}$ is affine. Let $\mathfrak{U}$ be an affine open subset of $\mathfrak{S}$ and let $U=\spc_{S,\mathfrak{S}}^{-1}(\mathfrak{U})$. Since $S=\mathfrak{Z}_{\eta}^{\ad}$ and $\id=g_{0\eta}$, the commutative triangle \begin{center}\begin{tikzcd}S \arrow{d}{\spc_{S,\mathfrak{Z}}}  \arrow{dr}{\spc_{S,\mathfrak{S}}} \\ \mathfrak{Z} \arrow{r}{g_{0}} & \mathfrak{S} \end{tikzcd}\end{center}(together with Lemma~\ref{Description of inverse images}) shows that $g_{0}^{-1}(\mathfrak{U})_{\eta}^{\ad}$ is equal to the affinoid open subspace \begin{equation*}U=\Spa(\mathcal{O}_{S}(U), \mathcal{O}_{S}^{+}(U))\end{equation*}of $S$. Therefore, we have \begin{equation*}\mathcal{O}_{\mathfrak{Z}}(g_{0}^{-1}(\mathfrak{U}))[\varpi^{-1}]=\mathcal{O}_{S}(U),\end{equation*}and the integral closure of $\mathcal{O}_{\mathfrak{Z}}(g_{0}^{-1}(\mathfrak{U}))$ inside $\mathcal{O}_{\mathfrak{Z}}(g_{0}^{-1}(\mathfrak{U}))[\varpi^{-1}]$ is equal to $\mathcal{O}_{S}^{+}(U)$. Since $\mathfrak{Z}$ is integrally closed in its generic fiber, we have \begin{equation*}\mathcal{O}_{\mathfrak{Z}}(g_{0}^{-1}(\mathfrak{U}))=\mathcal{O}_{S}^{+}(U).\end{equation*}Since $\mathfrak{U}$ was arbitrary, this means that \begin{equation*}g_{0\ast}\mathcal{O}_{\mathfrak{Z}}=\spc_{S,\mathfrak{S}\ast}\mathcal{O}_{S}^{+}\end{equation*}as adically quasi-coherent algebras on $\mathfrak{S}$. By \cite{FK}, Ch.~I, Theorem~4.1.8, this implies the desired equality $g_{0}=f_{0}$.\end{proof}
\begin{mydef}[Normalization of a formal model]\label{Definition of normalization}In the situation of Proposition~\ref{Normalization of a formal model} we call the formal scheme $\mathfrak{X}$ over $\mathfrak{S}$ the normalization of $\mathfrak{S}$ in its generic fiber $S$, or the normalization of the formal $R$-model $\mathfrak{S}$ of $S$.\end{mydef}
We also give a name to morphisms between formal models which induce the identity on the generic fiber. 
\begin{mydef}[Formal modification]\label{Definition of formal modification}For a locally rig-sheafy qcqs adic formal $R$-scheme $\mathfrak{X}$ with generic fiber $X$ over $\Spa(R[\varpi^{-1}], \overline{R})$, a \textit{formal modification} $f_{0}$ of $\mathfrak{X}$ is an adic morphism of locally rig-sheafy qcqs adic formal $R$-schemes $f_{0}: \mathfrak{X}'\to\mathfrak{X}$ which induces an isomorphism on the generic fibers. A morphism $h_{0}$ between two formal modifications $f_{0}: \mathfrak{X}'\to\mathfrak{X}$ and $g_{0}: \mathfrak{X}''\to\mathfrak{X}$ of $\mathfrak{X}$ is a commutative triangle \begin{center}\begin{tikzcd}\mathfrak{X}'\arrow{r}{h_{0}} \arrow{dr}{f_{0}} & \mathfrak{X}'' \arrow{d}{g_{0}} \\ & \mathfrak{X}\end{tikzcd}\end{center}of adic morphisms of formal schemes. If there exists a morphism from a formal modification $f_{0}$ to a formal modification $g_{0}$ of $\mathfrak{X}$, then we also say that $f_{0}$ dominates $g_{0}$ and $g_{0}$ is dominated by $f_{0}$.\end{mydef}
By Lemma~\ref{Admissible formal blow-ups and generic fiber}, every admissible formal blow-up is a formal modification. On the other hand, we obtain a different class of examples from Proposition~\ref{Normalization of a formal model}.
\begin{cor}\label{Normalization of a formal model 2}For every locally rig-sheafy $\varpi$-torsion-free quasi-compact quasi-separated adic formal $R$-scheme $\mathfrak{X}$ with uniform adic analytic generic fiber $X$ there exists a unique (up to isomorphism) affine formal modification $\mathfrak{Z}\to \mathfrak{X}$ of $\mathfrak{X}$ such that $\mathfrak{Z}$ is $\varpi$-torsion-free and integrally closed in its generic fiber $X$.\end{cor}
\begin{proof}This is a reformulation of Proposition~\ref{Normalization of a formal model}.\end{proof}
The assumption that $X$ be uniform in Proposition~\ref{Normalization of a formal model} and Corollary~\ref{Normalization of a formal model 2} cannot be fully omitted.
\begin{lemma}\label{Integrally closed formal models and uniform adic spaces}If $X$ is an adic space over $\Spa(R[\varpi^{-1}], \overline{R})$ which admits an integrally closed formal $R$-model $\mathfrak{X}$, then $X$ has an open cover by affinoid open subspaces $U_{i}$ such that each of the Tate rings $\mathcal{O}_{X}(U_{i})$ is uniform.\end{lemma}
\begin{proof}Covering $\mathfrak{X}$ by rig-sheafy affine open subsets, it suffices to assume that $\mathfrak{X}$ is a rig-sheafy affine adic formal $R$-scheme. Then $X$ is affinoid, namely, \begin{equation*}X=\Spa(\mathcal{O}_{\mathfrak{X}}(\mathfrak{X})[\varpi^{-1}], \overline{\mathcal{O}_{\mathfrak{X}}(\mathfrak{X})}),\end{equation*}where the topology on $\mathcal{O}_{\mathfrak{X}}(\mathfrak{X})[\varpi^{-1}]$ is the one defined by the pair $(\mathcal{O}_{\mathfrak{X}}(\mathfrak{X}), (\varpi))$. The assumption that $\mathfrak{X}$ is an integrally closed formal $R$-model of $X$ implies that \begin{equation*}\mathcal{O}_{\mathfrak{X}}(\mathfrak{X})=\overline{\mathcal{O}_{\mathfrak{X}}(\mathfrak{X})}=\mathcal{O}_{X}^{+}(X),\end{equation*}so $\mathcal{O}_{X}^{+}(X)$ is a ring of definition of $\mathcal{O}_{X}(X)$. Thus the assertion follows from the following elementary lemma, which follows, for example, from \cite{Nakazato-Shimomoto22}, Lemma~2.13(2), but whose proof we include for the reader's convenience.\end{proof}
\begin{lemma}\label{Power-bounded elements}Let $A$ be a Tate ring which has an integrally closed ring of definition $A^{+}$. Then $A$ is uniform.\end{lemma}
\begin{proof}Let $\varpi\in A^{+}$ be a non-zero-divisor which generates an ideal of definition of $A^{+}$. Let $f\in A^{\circ}$, i.e., there exists an integer $m\geq 0$ such that $f^{n}\in \varpi^{-m}A^{+}$ for all $n$. Then, in particular, $(\varpi f)^{m}\in A^{+}$. Since $A^{+}$ is integrally closed, this means that $\varpi f\in A^{+}$. It follows that $\varpi A^{\circ}\subseteq A^{+}$, so $A^{\circ}$ is bounded, as claimed.\end{proof}
Let us also give a name to those adic formal $R$-schemes whose generic fibers are uniform adic spaces (cf. \cite{Nakazato-Shimomoto22}, Definition~2.16 and Lemma~2.17).
\begin{mydef}[Locally stably uniform adic formal $R$-schemes]\label{Locally stably uniform adic formal scheme}~\begin{enumerate}[(1)]\item A complete adic ring $A$ with ideal of definition generated by a single non-zero-divisor $\varpi$ is called uniform (respectively, stably uniform) if the Tate ring $A[\varpi^{-1}]$ is uniform (respectively, stably uniform). \item An affine adic formal $R$-scheme $\mathfrak{X}$ is called uniform (respectively, stably uniform) if the adic ring $\mathcal{O}_{\mathfrak{X}}(\mathfrak{X})$ is uniform (respectively, stably uniform). An arbitrary adic formal $R$-scheme $\mathfrak{X}$ is then called locally uniform (respectively, locally stably uniform) if $\mathfrak{X}$ has an open cover by uniform (respectively, stably uniform) affine adic formal $R$-schemes.\end{enumerate}\end{mydef}
Locally stably uniform adic formal $R$-schemes are locally rig-sheafy by the well-known theorem of Buzzard-Verberkmoes \cite{BV} and Mihara \cite{Mihara}. In particular, the adic analytic generic fiber $\mathfrak{X}_{\eta}^{\ad}$ over $\Spa(R[\varpi^{-1}], \overline{R})$ of such an adic formal $R$-scheme $\mathfrak{X}$ is defined. 
\begin{lemma}\label{Uniformity is local}If $\mathfrak{X}$ is a locally uniform (respectively, locally stably uniform) adic formal $R$-scheme, then every affine open subset $\mathfrak{U}$ of $\mathfrak{X}$ is uniform (respectively, stably uniform).\end{lemma}
\begin{proof}If $\mathfrak{X}$ is locally stably uniform, then $X=\mathfrak{X}_{\eta}^{\ad}$ is defined and then the assertion follows from Lemma~\ref{Uniformity is a local notion}. If $\mathfrak{X}$ is merely locally uniform, it suffices to prove that every affine open subset is again locally uniform. Indeed, if $\mathfrak{U}$ is an arbitrary affine open subset and $(\mathfrak{U}_{i})_{i}$ is a finite affine open cover of $\mathfrak{U}$, then the injective map \begin{equation*}\mathcal{O}_{\mathfrak{X}}(\mathfrak{U})\hookrightarrow \prod_{i}\mathcal{O}_{\mathfrak{X}}(\mathfrak{U}_{i})\end{equation*}is $\varpi$-adically strict (since $\mathcal{O}_{\mathfrak{X}}$ is a sheaf of complete topological rings on $\mathfrak{X}$) and then the map of Tate rings \begin{equation*}\mathcal{O}_{\mathfrak{X}}(\mathfrak{U})[\varpi^{-1}]\hookrightarrow \prod_{i}\mathcal{O}_{\mathfrak{X}}(\mathfrak{U}_{i})[\varpi^{-1}]\end{equation*}is also strict. It then follows as in the proof of Lemma~\ref{Uniformity is a local notion} that $\mathcal{O}_{\mathfrak{X}}(\mathfrak{U})$ is uniform if each $\mathcal{O}_{\mathfrak{X}}(\mathfrak{U}_{i})$ is uniform. Thus the assertion reduces to proving that for any uniform complete adic ring $A$, with $\varpi\in A$ a non-zero-divisor generating an ideal of definition of $A$, and for any $g\in A$, the Tate ring $A\langle g^{-1}\rangle[\varpi^{-1}]=A[\varpi^{-1}]\langle g^{-1}\rangle$ is uniform. That is, we only have to prove that $A[\varpi^{-1}]\langle g^{-1}\rangle$ is a uniform Tate ring whenever $A[\varpi^{-1}]$ is one. By Lemma~\ref{Power-bounded elements} it suffices to show that $A[\varpi^{-1}]\langle g^{-1}\rangle$ has an integrally closed ring of definition. Let $A^{+}$ be an integrally closed ring of definition of $A[\varpi^{-1}]$. Then $A^{+}[g^{-1}]$ is an integrally closed ring of definition of $A[\varpi^{-1}][g^{-1}]$. It then follows from \cite{Huber0}, Lemma~2.4.3(iii) and (iv), that $A^{+}\langle g^{-1}\rangle$ is an integrally closed ring of definition of $A[\varpi^{-1}]\langle g^{-1}\rangle$.\end{proof}
In fact, locally stably uniform adic formal $R$-schemes $\mathfrak{X}$ are precisely the locally rig-sheafy adic formal $R$-schemes whose generic fiber is uniform.
\begin{lemma}\label{Uniform generic fiber}Let $X$ be an adic space over $\Spa(R[\varpi^{-1}], \overline{R})$ and let $\mathfrak{X}$ be a $\varpi$-torsion-free formal $R$-model of $X$. The following are equivalent: \begin{enumerate}[(1)] \item $X$ is a uniform adic space. \item $\mathfrak{X}$ is a locally stably uniform adic formal $R$-scheme.\end{enumerate}\end{lemma}
\begin{proof}Let $(\mathfrak{U}_{i})_{i}$ be a rig-sheafy affine open cover of $\mathfrak{X}$ and let $U_{i}=\spc_{X,\mathfrak{X}}^{-1}(\mathfrak{U}_{i})$. If $X$ is a uniform adic space, then the Tate ring $\mathcal{O}_{X}(U_{i})=\mathcal{O}_{\mathfrak{X}}(\mathfrak{U}_{i})[\varpi^{-1}]$ is stably uniform for each $i$, so $\mathfrak{X}$ is locally stably uniform. Conversely, if $\mathfrak{X}$ is locally stably uniform, then, by Lemma~\ref{Uniformity is local}, each $\mathfrak{U}_{i}$ is stably uniform and thus $\mathcal{O}_{X}(U_{i})$ is a stably uniform Tate ring for all $i$. We conclude by Lemma~\ref{Uniformity is a local notion} that $X$ is uniform. \end{proof}
In particular, every $\varpi$-torsion-free formal modification of a locally stably uniform adic formal $R$-scheme is again locally stably uniform. We are now ready to introduce the following special kind of formal modifications.
\begin{mydef}[Normalized formal blow-up]\label{Normalized formal blow-up}Let $\mathfrak{X}$ be a locally stably uniform qcqs $\varpi$-torsion-free adic formal $R$-scheme. A formal modification $f_{0}: \mathfrak{Z}'\to\mathfrak{X}$ of $\mathfrak{X}$ is called a \textit{normalized formal blow-up} if it is the composition of an admissible formal blow-up $\mathfrak{X}'\to\mathfrak{X}$ followed by its normalization $\mathfrak{Z}'\to\mathfrak{X}'$ (whose existence is guaranteed by Corollary~\ref{Normalization of a formal model 2}).\end{mydef}
The following analog of Lemma~\ref{Admissible formal blow-up and affine opens} holds true for normalized formal blow-ups.
\begin{lemma}\label{Normalized formal blow-up and affine opens}Let $\mathfrak{X}$ be a locally stably uniform qcqs $\varpi$-torsion-free adic formal $R$-scheme. Let $\mathfrak{U}$ be a quasi-compact open formal subscheme of $\mathfrak{X}$ and let $\varphi: \mathfrak{V}\to\mathfrak{U}$ be a normalized formal blow-up. Then there exists a normalized formal blow-up $\psi: \mathfrak{Z}\to \mathfrak{X}$ and an open immersion $\mathfrak{V}\hookrightarrow \mathfrak{Z}$ such that the restriction of $\psi$ to $\mathfrak{V}$ equals $\varphi$.\end{lemma}
\begin{proof}Let $\mathfrak{U}'\to\mathfrak{U}$ be an admissible formal blow-up such that $\mathfrak{V}$ is the normalization of $\mathfrak{U}'$ in its generic fiber. By Lemma~\ref{Admissible formal blow-up and affine opens}, there exists an admissible formal blow-up $\mathfrak{X}'\to\mathfrak{X}$ extending $\mathfrak{U}'\to\mathfrak{U}$. Let $g_{0}: \mathfrak{Z}\to\mathfrak{X}'$ be the normalization of $\mathfrak{X}'$ in $X$. Since $g_{0}$ is an affine formal modification, so is the restricted morphism \begin{equation*}g_{0}\vert_{g_{0}^{-1}(\mathfrak{U}')}: g_{0}^{-1}(\mathfrak{U}')\to \mathfrak{U}'.\end{equation*}Moreover, $g_{0}^{-1}(\mathfrak{U}')$ is integrally closed in its generic fiber, $\mathfrak{Z}$ being integrally closed in its generic fiber. It follows that $g_{0}^{-1}(\mathfrak{U}')$ is canonically identified with the normalization $\mathfrak{V}$ of $\mathfrak{U}'$ in its generic fiber, by the uniqueness part of Corollary~\ref{Normalization of a formal model 2}. Hence the composition $\psi$ of $\mathfrak{X}'\to\mathfrak{X}$ with $g_{0}: \mathfrak{Z}\to\mathfrak{X}'$ is a normalized formal blow-up of $\mathfrak{X}$ with the desired property.\end{proof}

\section{Proof of the main results}\label{sec:main results}

It turns out that normalized formal blow-ups play a role in the theory of formal $R$-models of a general uniform qcqs adic space over $\Spa(R[\varpi^{-1}], \overline{R})$ analogous to the role of admissible formal blow-ups in Raynaud's theory of formal models of rigid spaces.  
\begin{lemma}\label{Faithfulness}Let $f_{0}, g_{0}: \mathfrak{Z}\to\mathfrak{X}$ be morphisms of locally rig-sheafy $\varpi$-torsion-free quasi-separated adic formal $R$-schemes such that $f_{0\eta}=g_{0\eta}$. Then $f_{0}=g_{0}$.\end{lemma}
\begin{proof}It suffices to prove this when the formal schemes $\mathfrak{Z}$, $\mathfrak{X}$ are quasi-compact, in which case Corollary~\ref{Center maps are surjective} applies to them. Let $X=\mathfrak{X}_{\eta}^{\ad}$ and $Z=\mathfrak{Z}_{\eta}^{\ad}$. Consider the commutative diagram \begin{center}\begin{tikzcd}Z\arrow{r}{f_{0\eta}} \arrow{d}{\spc_{Z,\mathfrak{Z}}} & X \arrow{d}{\spc_{X,\mathfrak{X}}} \\ \mathfrak{Z}\arrow{r}{f_{0}} & \mathfrak{X} \end{tikzcd}\end{center}and the analogous commutative diagram for $g_{0}$. Let $x_{0}\in\vert\mathfrak{Z}\vert$. By Corollary~\ref{Center maps are surjective} we can choose $x\in\vert Z\vert$ such that $\spc_{Z,\mathfrak{Z}}(x)=x_{0}$. Then we have \begin{equation*}f_{0}(x_{0})=f_{0}(\spc_{Z,\mathfrak{Z}}(x))=\spc_{X,\mathfrak{X}}(f_{0\eta}(x))=\spc_{X,\mathfrak{X}}(g_{0\eta}(x))=g_{0}(x_{0}),\end{equation*}so $f_{0}$ and $g_{0}$ coincide as maps between the underlying topological spaces of $\mathfrak{Z}$ and $\mathfrak{X}$. Therefore, to prove that $f_{0}=g_{0}$, we may work locally and assume that $\mathfrak{Z}$ and $\mathfrak{X}$ are affine and rig-sheafy. In this case the assertion follows from the commutative diagram \begin{center}\begin{tikzcd}\mathcal{O}_{X}(X)\arrow{r}{f_{0\eta}^{\ast}} & \mathcal{O}_{Z}(Z) \\ \mathcal{O}_{\mathfrak{X}}(\mathfrak{X})\arrow{r}{f_{0}^{\ast}} \arrow[hook]{u} & \mathcal{O}_{\mathfrak{Z}}(\mathfrak{Z}), \arrow[hook]{u} \end{tikzcd}\end{center}where $f_{0}^{\ast}$ (respectively, $f_{0\eta}^{\ast}$) is the continuous ring map induced by $f_{0}$ (respectively, $f_{0\eta}$), and from the analogous diagram for $g_{0}$.\end{proof}
Using the above lemma, we can prove that the normalization of a locally stably uniform $\varpi$-torsion-free qcqs adic formal $R$-scheme is characterized by the following universal property. 
\begin{prop}[Universal property of normalization]\label{Universal property of normalization}Set assumptions as in Proposition~\ref{Normalization of a formal model}. Let $f_{0}: \mathfrak{Z}\to \mathfrak{S}$ be a morphism of $\varpi$-torsion-free qcqs locally rig-sheafy adic formal $R$-schemes such that $\mathfrak{Z}$ is integrally closed in its generic fiber $Z=\mathfrak{Z}_{\eta}^{\ad}$. Then $f_{0}$ factors uniquely through the normalization of $\mathfrak{S}$ in its generic fiber $S=\mathfrak{S}_{\eta}^{\ad}$. \end{prop}
\begin{proof}Let $f=f_{0\eta}: Z\to S$. Consider the normalization \begin{equation*}g_{0}: \Spf(\spc_{S,\mathfrak{S}\ast}\mathcal{O}_{S}^{+})\to \mathfrak{S}\end{equation*}of $\mathfrak{S}$ in $S$ (cf.~Proposition~\ref{Normalization of a formal model}). Since $g_{0\eta}=\id$, any two morphisms \begin{equation*}h_{0}, h'_{0}: \mathfrak{Z}\to \Spf(\spc_{S,\mathfrak{S}\ast}\mathcal{O}_{S}^{+})\end{equation*}with \begin{equation*}f_{0}=g_{0}\circ h_{0}=g_{0}\circ h'_{0}\end{equation*}induce the same morphism of adic spaces $Z\to S$, and then Lemma~\ref{Faithfulness} gives $h'_{0}=h_{0}$. To prove existence, let $(\mathfrak{U}_{i})_{i}$ be an open cover of $\mathfrak{S}$ by rig-sheafy affine open subsets and let $U_{i}$ be the pre-image of $\mathfrak{U}_{i}$ in $S$. Up to enlarging the index set, we can find a rig-sheafy affine open cover $(\mathfrak{V}_{i})_{i}$ of $\mathfrak{Z}$ such that $f_{0}(\mathfrak{V}_{i})\subseteq \mathfrak{U}_{i}$ for all $i$. Set $V_{i}=\spc_{Z, \mathfrak{Z}}^{-1}(\mathfrak{V}_{i})$. For every $i, j$, let $(\mathfrak{W}_{ijk})_{k}$ be a rig-sheafy affine open cover of $\mathfrak{V}_{i}\cap \mathfrak{V}_{j}$ and set $W_{ijk}=\spc_{Z, \mathfrak{Z}}^{-1}(\mathfrak{W}_{ijk})$. Since $\mathfrak{Z}$ is integrally closed in its generic fiber $Z$, the ring $\mathcal{O}_{\mathfrak{Z}}(\mathfrak{V}_{i})$ is integrally closed in $\mathcal{O}_{Z}(V_{i})$ for each $i$, so, for every $i$, the ring map \begin{equation*}\mathcal{O}_{\mathfrak{S}}(\mathfrak{U}_{i})\to \mathcal{O}_{\mathfrak{Z}}(\mathfrak{V}_{i})=\mathcal{O}_{Z}^{+}(V_{i})\end{equation*}factors through the integral closure $\mathcal{O}_{S}^{+}(U_{i})$ of $\mathcal{O}_{\mathfrak{S}}(\mathfrak{U}_{i})$ in $\mathcal{O}_{S}(U_{i})$. This defines morphisms of formal schemes \begin{equation*}h_{0, i}: \mathfrak{V}_{i}\to \Spf(\spc_{S, \mathfrak{S}\ast}\mathcal{O}_{S}^{+}(\mathfrak{U}_{i}))\end{equation*}such that \begin{equation*}f_{0}\vert_{\mathfrak{V}_{i}}=g_{0}\vert_{g_{0}^{-1}(\mathfrak{U}_{i})}\circ h_{0, i}\end{equation*}for every $i$. It remains to show that these morphisms glue to a morphism \begin{equation*}h_{0}: \mathfrak{Z}\to \Spf(\spc_{S, \mathfrak{S}\ast}\mathcal{O}_{S}^{+}).\end{equation*}That is, we have to show that $h_{0, i}\vert_{\mathfrak{W}_{ijk}}=h_{0, j}\vert_{\mathfrak{W}_{ijk}}$ for all $i, j, k$. But $g_{0\eta}=\id$, so \begin{equation*}h_{0, i\eta}=f\vert_{V_{i}}\end{equation*}for all $i$ as morphisms $V_{i}\to S$ and then \begin{equation*}(h_{0, i}\vert_{\mathfrak{W}_{ijk}})_{\eta}=f\vert_{W_{ijk}}=(h_{0, j}\vert_{\mathfrak{W}_{ijk}})_{\eta}\end{equation*}as morphisms $W_{ijk}\to S$. We conclude by Lemma~\ref{Faithfulness} that $h_{0, i}\vert_{\mathfrak{W}_{ijk}}=h_{0, j}\vert_{\mathfrak{W}_{ijk}}$ for all $i, j, k$, as desired. \end{proof}
Combining the above proposition with the universal property of admissible formal blow-ups, we obtain the following universal property of normalized formal blow-ups. 
\begin{prop}[Universal property of normalized formal blow-ups]\label{Universal property of normalized formal blow-up}Let $f_{0}: \mathfrak{Z}\to \mathfrak{X}$ be a morphism of locally stably uniform $\varpi$-torsion-free qcqs adic formal $R$-schemes such that $\mathfrak{Z}$ is integrally closed in its generic fiber. Let $\mathcal{J}\subseteq \mathcal{O}_{\mathfrak{X}}$ be an admissible ideal of $\mathfrak{X}$ and suppose that the ideal $(f_{0}^{-1}\mathcal{J})\mathcal{O}_{\mathfrak{Z}}$ is invertible. Then $f_{0}$ factors uniquely through the normalization of the admissible formal blow-up along $\mathcal{J}$ in its generic fiber $X=\mathfrak{X}_{\eta}^{\ad}$. \end{prop}
\begin{proof}Let $\mathfrak{X}'\to \mathfrak{X}$ be the admissible formal blow-up of $\mathfrak{X}$ along $\mathcal{J}$. By the universal property of admissible formal blow-ups (\cite{FK}, Ch.~II, Proposition~1.1.4(3)), $\mathfrak{Z}\to \mathfrak{X}$ factors uniquely through $\mathfrak{X}'\to \mathfrak{X}$. But $\mathfrak{Z}$ is integrally closed in its generic fiber, so, by Proposition~\ref{Universal property of normalization}, the resulting map $\mathfrak{Z}\to \mathfrak{X}'$ factors uniquely through the normalization of $\mathfrak{X}'$ in its generic fiber $X$. 
\end{proof}
\begin{lemma}\label{Fullness}Let $\mathfrak{Z}$, $\mathfrak{X}$ be locally stably uniform $\varpi$-torsion-free qcqs adic formal $R$-schemes and let $f: \mathfrak{Z}_{\eta}^{\ad}\to\mathfrak{X}_{\eta}^{\ad}$ be a morphism of adic spaces over $\Spa(R[\varpi^{-1}], \overline{R})$. If $\mathfrak{Z}$ is integrally closed in its generic fiber, there exists a normalized formal blow-up $\mathfrak{Z}'\to\mathfrak{Z}$ and a unique adic morphism $f_{0}: \mathfrak{Z}'\to \mathfrak{X}$ with $f_{0\eta}=f$.\end{lemma}
Note that under the assumptions of the lemma the adic space $\mathfrak{Z}_{\eta}^{\ad}$ is uniform, by Lemma~\ref{Uniform generic fiber}, so we are indeed allowed to talk about normalized formal blow-ups of $\mathfrak{Z}$. 
\begin{proof}Suppose first that $\mathfrak{Z}$ and $\mathfrak{X}$ are rig-sheafy affine formal $R$-schemes. Then $\mathfrak{Z}_{\eta}^{\ad}=\Spa(\mathcal{O}_{\mathfrak{Z}}(\mathfrak{Z})[\varpi^{-1}], \overline{\mathcal{O}_{\mathfrak{Z}}(\mathfrak{Z})})$ and similarly for $\mathfrak{X}_{\eta}^{\ad}$. In this case we establish a stronger claim: The morphism $f$ extends uniquely to an adic morphism $f_{0}: \mathfrak{Z}\to\mathfrak{X}$. Since $\mathfrak{Z}$ is integrally closed in its generic fiber, we have \begin{equation*}\mathcal{O}_{\mathfrak{Z}}(\mathfrak{Z})=\overline{\mathcal{O}_{\mathfrak{Z}}(\mathfrak{Z})}=\mathcal{O}_{\mathfrak{Z}_{\eta}^{\ad}}^{+}(\mathfrak{Z}_{\eta}^{\ad}).\end{equation*}Hence any morphism of adic spaces $f: \mathfrak{Z}_{\eta}^{\ad}\to \mathfrak{X}_{\eta}^{\ad}$ induces a continuous ring map $\mathcal{O}_{\mathfrak{X}_{\eta}^{\ad}}^{+}(\mathfrak{X}_{\eta}^{\ad})\to \mathcal{O}_{\mathfrak{Z}}(\mathfrak{Z})$. The latter restricts to a continuous ring map $\mathcal{O}_{\mathfrak{X}}(\mathfrak{X})\to \mathcal{O}_{\mathfrak{Z}}(\mathfrak{Z})$ which defines the desired morphism of formal schemes $f_{0}$ (the uniqueness of this morphism follows from Lemma~\ref{Faithfulness}). Now suppose that $\mathfrak{X}$ is also integrally closed in its generic fiber. Then $\mathcal{O}_{\mathfrak{X}_{\eta}^{\ad}}^{+}(\mathfrak{X}_{\eta}^{\ad})=\mathcal{O}_{\mathfrak{X}}(\mathfrak{X})$, so the map $\mathcal{O}_{\mathfrak{X}}(\mathfrak{X})\to \mathcal{O}_{\mathfrak{Z}}(\mathfrak{Z})$ is the canonical map \begin{equation*}\mathcal{O}_{\mathfrak{X}_{\eta}^{\ad}}^{+}(\mathfrak{X}_{\eta}^{\ad})\to \mathcal{O}_{\mathfrak{Z}_{\eta}^{\ad}}^{+}(\mathfrak{Z}_{\eta}^{\ad})\end{equation*}induced by $f$. Thus, if $f$ is an isomorphism, the above continuous ring map\begin{equation*}\mathcal{O}_{\mathfrak{X}}(\mathfrak{X})=\mathcal{O}_{\mathfrak{X}_{\eta}^{\ad}}^{+}(\mathfrak{X}_{\eta}^{\ad})\tilde{\to} \mathcal{O}_{\mathfrak{Z}_{\eta}^{\ad}}^{+}(\mathfrak{Z}_{\eta}^{\ad})=\mathcal{O}_{\mathfrak{Z}}(\mathfrak{Z})\end{equation*}is an isomorphism, so the morphism of formal schemes $f_{0}$ is an isomorphism. This settles the case when $\mathfrak{Z}$ and $\mathfrak{X}$ are affine.

In the general case, cover $\mathfrak{X}$ by rig-sheafy affine open subsets $(\mathfrak{U}_{i})_{i\in I}$. Set \begin{equation*}U_{i}=\spc_{\mathfrak{X}_{\eta}^{\ad},\mathfrak{X}}^{-1}(\mathfrak{U}_{i}).\end{equation*}Possibly after enlarging the index set $I$, we can choose an affinoid open cover $(V_{i})_{i\in I}$ of $\mathfrak{Z}_{\eta}^{\ad}$ with $f(V_{i})\subseteq U_{i}$ for all $i$. By Corollary~\ref{Open covers and formal models 2}, there exists an admissible formal blow-up $\mathfrak{T}\to\mathfrak{Z}$ of $\mathfrak{Z}$ and an open cover $(\mathfrak{T}_{i})_{i\in I}$ of $\mathfrak{T}$ such that \begin{equation*}V_{i}=\spc_{\mathfrak{Z}_{\eta}^{\ad},\mathfrak{T}}^{-1}(\mathfrak{T}_{i})\end{equation*}for all $i$. Composing $\mathfrak{T}\to\mathfrak{Z}$ with its normalization $\mathfrak{Z}'\to\mathfrak{T}$ in $\mathfrak{Z}_{\eta}^{\ad}$ and letting $\mathfrak{V}_{i}$ be the pre-image of $\mathfrak{T}_{i}$ in $\mathfrak{Z}'$, we obtain a normalized formal blow-up $\mathfrak{Z}'\to\mathfrak{Z}$ and an open cover $(\mathfrak{V}_{i})_{i}$ of $\mathfrak{Z}'$ such that \begin{equation*}V_{i}=\spc_{\mathfrak{Z}_{\eta}^{\ad},\mathfrak{Z}'}^{-1}(\mathfrak{V}_{i})\end{equation*}for all $i$. Note that, by design, $\mathfrak{Z}'$ is again integrally closed in its generic fiber. Let $(\mathfrak{V}'_{i})_{i}$ be a refinement of the open cover $(\mathfrak{V}_{i})_{i}$ of $\mathfrak{Z}'$ such that all $\mathfrak{V}'_{i}$ are affine (possibly for a larger index set). Set \begin{equation*}V'_{i}=\spc_{\mathfrak{Z}_{\eta}^{\ad},\mathfrak{Z}'}^{-1}(\mathfrak{V}'_{i})={\mathfrak{V}'_{i}}_{\eta}^{\ad}\end{equation*}and note these still form an affinoid open cover of $\mathfrak{Z}_{\eta}^{\ad}$ such that $f(V'_{i})\subseteq U_{i}$ for all $i$. For any pair of indices $i, j\in I$ let $(\mathfrak{W}_{ijk})_{k}$ be a finite affine open cover of $\mathfrak{V}'_{i}\cap\mathfrak{V}'_{j}$ and, for every $i, j, k$, set \begin{equation*}W_{ijk}=\spc_{\mathfrak{Z}_{\eta}^{\ad},\mathfrak{Z}'}^{-1}(\mathfrak{W}_{ijk}).\end{equation*}By what we established in the previous paragraph, we know that there exists, for every index $i$, a unique adic morphism $f_{0i}: \mathfrak{V}'_{i}\to\mathfrak{U}_{i}$ with the property that \begin{equation*}f_{0i\eta}=f\vert_{V'_{i}}: V'_{i}\to U_{i}.\end{equation*}The morphisms of adic spaces $(f_{0i}\vert_{\mathfrak{W}_{ijk}})_{\eta}$ and $(f_{0j}\vert_{\mathfrak{W}_{ijk}})_{\eta}$ coincide as morphisms \begin{equation*}W_{ijk}\to \mathfrak{X}_{\eta}^{\ad}\end{equation*}since they are both equal to $f\vert_{W_{ijk}}$. By Lemma~\ref{Faithfulness}, this implies \begin{equation*}f_{0i}\vert_{\mathfrak{W}_{ijk}}=f_{0j}\vert_{\mathfrak{W}_{ijk}}\end{equation*}for all $i, j, k$. It follows that the adic morphisms $f_{0i}$ for $i\in I$ glue to a unique adic morphism $f_{0}: \mathfrak{Z}'\to \mathfrak{X}$ satisfying $f_{0\eta}=f$. 
\end{proof}
\begin{lemma}\label{Isomorphisms}Set assumptions as in the previous lemma and suppose further that $\mathfrak{X}$ is also integrally closed in its generic fiber. If the map of adic spaces $f: \mathfrak{Z}_{\eta}^{\ad}\to \mathfrak{X}_{\eta}^{\ad}$ is an isomorphism, then there exist compositions of normalized formal blow-ups $\mathfrak{Z}''\to \mathfrak{Z}'\to \mathfrak{Z}$ and $\mathfrak{X}''\to \mathfrak{X}'\to \mathfrak{X}$ and an isomorphism \begin{equation*}h_{0}: \mathfrak{Z}''\cong \mathfrak{X}''\end{equation*}with $h_{0\eta}=f$. \end{lemma}
\begin{proof}We follow the proof of \cite{BoschLectures}, \S8.4, Lemma~4(d). By Lemma~\ref{Fullness}, there exist normalized formal blow-ups $\mathfrak{Z}'\to \mathfrak{Z}$ and $\mathfrak{X}'\to \mathfrak{X}$ and maps \begin{equation*}f_{0}: \mathfrak{Z}'\to \mathfrak{X}\end{equation*}and \begin{equation*}g_{0}: \mathfrak{X}'\to \mathfrak{Z}\end{equation*}such that $f_{0\eta}=f$ and $g_{0\eta}=f^{-1}$. Let $\mathcal{I}$ be an admissible ideal sheaf on $\mathfrak{Z}$ such that $\mathfrak{Z}'\to \mathfrak{Z}$ is the normalization of the admissible formal blow-up along $\mathcal{I}$ and let $\mathcal{J}$ be an admissible ideal sheaf of $\mathfrak{X}$ such that $\mathfrak{X}'\to \mathfrak{X}$ is the normalization of the admissible formal blow-up along $\mathcal{J}$. Set \begin{equation*}\mathcal{J}\mathcal{O}_{\mathfrak{Z}'}=(f_{0}^{-1}\mathcal{J})\mathcal{O}_{\mathfrak{Z}'}\end{equation*}and \begin{equation*}\mathcal{I}\mathcal{O}_{\mathfrak{X}'}=(g_{0}^{-1}\mathcal{I})\mathcal{O}_{\mathfrak{X}'}.\end{equation*}Let $\mathfrak{Z}''\to \mathfrak{Z}'$ be the normalization of the admissible formal blow-up along $\mathcal{J}\mathcal{O}_{\mathfrak{Z}'}$ and let $\mathfrak{X}''\to \mathfrak{X}'$ be the normalization of the admissible formal blow-up along $\mathcal{I}\mathcal{O}_{\mathfrak{X}'}$. By the universal property of normalized formal blow-ups (Proposition~\ref{Universal property of normalized formal blow-up}), the composite map $\mathfrak{Z}''\to \mathfrak{Z}'\to \mathfrak{X}$ factors uniquely through the normalized formal blow-up $\mathfrak{X}'\to \mathfrak{X}$. Analogously, $\mathfrak{X}''\to \mathfrak{X}'\to \mathfrak{Z}$ factors uniquely through the normalized formal blow-up $\mathfrak{Z}'\to \mathfrak{Z}$. We obtain a diagram of morphisms of formal schemes \begin{center}\begin{tikzcd}\mathfrak{Z}'' \arrow{d} \arrow{dr}[pos=0.3]{f'_{0}} & \mathfrak{X}'' \arrow{d} \arrow{dl}[pos=0.7, yshift=-7pt]{g'_{0}} \\ \mathfrak{Z}' \arrow{d} \arrow{dr}[pos=0.3]{f_{0}} & \mathfrak{X}' \arrow{d} \arrow{dl}[pos=0.7, yshift=-7pt]{g_{0}} \\ \mathfrak{Z} & \mathfrak{X} \end{tikzcd}\end{center}which is commutative by Lemma~\ref{Faithfulness} since all maps induce isomorphisms on the generic fibers. Since $\mathfrak{Z}''\to \mathfrak{Z}$ factors through the admissible formal blow-up of $\mathfrak{Z}$ along $\mathcal{I}$, the ideal sheaf $\mathcal{I}\mathcal{O}_{\mathfrak{Z}''}$ is invertible. By Proposition~\ref{Universal property of normalized formal blow-up} once again, this entails that the qcqs map $f'_{0}: \mathfrak{Z}''\to \mathfrak{X}'$ factors uniquely through a map $h_{0}: \mathfrak{Z}''\to \mathfrak{X}''$. By the analogous argument applied to the ideal $\mathcal{J}$, we see that the map $g'_{0}: \mathfrak{X}''\to \mathfrak{Z}'$ factors uniquely through a map $h'_{0}: \mathfrak{X}''\to \mathfrak{Z}''$. By construction, $h_{0\eta}=f_{0\eta}=f$ and $h'_{0\eta}=g_{0\eta}=f^{-1}$, so $(h_{0}\circ h'_{0})_{\eta}=\id$ and $(h'_{0}\circ h_{0})_{\eta}=\id$. But then Lemma~\ref{Faithfulness} implies that $h_{0}$ and $h'_{0}$ are actually inverse to each other. \end{proof}
Using the above results on integrally closed formal models, we can establish the existence of formal $R$-models for an arbitrary uniform quasi-compact quasi-separated adic space $X$ over $\Spa(R[\varpi^{-1}], \overline{R})$. Our proof is similar in structure to the classical argument from Raynaud theory (\cite{BL1}, proof of Theorem~4.1(e)), but uses normalized formal blow-ups in place of admissible formal blow-ups.
\begin{thm}\label{Existence of formal models}Let $R$ be a complete adic ring with ideal of definition generated by a non-zero-divisor $\varpi$, suppose that the Tate ring $R[\varpi^{-1}]$ is sheafy and let $X$ be a uniform quasi-compact quasi-separated adic space over $\Spa(R[\varpi^{-1}], \overline{R})$. There exists a $\varpi$-torsion-free quasi-compact quasi-separated formal $R$-model $\mathfrak{X}$ of $X$ which is integrally closed in its generic fiber $X$.\end{thm}
\begin{proof}Every uniform affinoid adic space $\Spa(A, A^{+})$ over $\Spa(R[\varpi^{-1}], \overline{R})$ admits the integrally closed $\varpi$-torsion-free affine formal $R$-model $\Spf(A^{+})$. We proceed by induction on the size of a finite affinoid open cover of $X$. Thus, suppose that the assertion holds for every qcqs uniform adic space over $\Spa(R[\varpi^{-1}], \overline{R})$ which can be covered by $n$ affinoid open subspaces. Let $X$ be a qcqs uniform adic space over $\Spa(R[\varpi^{-1}], \overline{R})$ which is covered by $n+1$ affinoid open subspaces $U_{1},\dots, U_{n}, U_{n+1}$. Set $U=U_{1}$ and $V=\bigcup_{i=2}^{n+1}U_{i}$. By the induction hypothesis, $U$ and $V$ admit $\varpi$-torsion-free qcqs formal $R$-models $\mathfrak{U}$ and $\mathfrak{V}$ which are integrally closed in their generic fibers. Set $W=U\cap V$. The hypothesis that $X$ is quasi-separated ensures that $W$ is quasi-compact. By Lemma~\ref{Open covers and formal models} there exist formal modifications $\mathfrak{U}'\to\mathfrak{U}$ and $\mathfrak{V}'\to\mathfrak{V}$ and quasi-compact open subsets $\mathfrak{W}_{1}\subseteq \mathfrak{U}'$ and $\mathfrak{W}_{2}\subseteq \mathfrak{V}'$ such that \begin{equation*}W=\spc_{U,\mathfrak{U}'}^{-1}(\mathfrak{W}_{1})=\spc_{V,\mathfrak{V}'}^{-1}(\mathfrak{W}_{2}).\end{equation*}Using Corollary~\ref{Normalization of a formal model 2} and replacing $\mathfrak{U}'$ (respectively, $\mathfrak{V}'$) with their normalizations inside $U$ (respectively, inside $V$), we may assume that $\mathfrak{U}'$ and $\mathfrak{V}'$ are $\varpi$-torsion-free and integrally closed in their generic fibers. In this case, $\mathfrak{W}_{1}$ and $\mathfrak{W}_{2}$ are also integrally closed in their generic fibers. Hence, by Lemma~\ref{Isomorphisms}, there exist compositions of normalized formal blow-ups $\mathfrak{W}''_{1}\to \mathfrak{W}'_{1}\to\mathfrak{W}_{1}$ and $\mathfrak{W}''_{2}\to \mathfrak{W}'_{2}\to \mathfrak{W}_{2}$ such that the isomorphism \begin{equation*}(\mathfrak{W}_{1})_{\eta}^{\ad}\cong(\mathfrak{W}_{2})_{\eta}^{\ad}\end{equation*}which arises from the identity $\spc_{U,\mathfrak{U}'}^{-1}(\mathfrak{W}_{1})=\spc_{V,\mathfrak{V}'}^{-1}(\mathfrak{W}_{2})$ induces an isomorphism \begin{equation*}\mathfrak{W}''_{1}\cong\mathfrak{W}''_{2}.\end{equation*}By Lemma~\ref{Normalized formal blow-up and affine opens} applied twice, there exist compositions of normalized formal blow-ups $\mathfrak{U}''\to\mathfrak{U}'$ and $\mathfrak{V}''\to \mathfrak{V}'$ such that we can regard $\mathfrak{W}''_{1}$ and $\mathfrak{W}''_{2}$ as open formal subschemes of $\mathfrak{U}''$ and $\mathfrak{V}''$, respectively. Gluing $\mathfrak{U}''$ and $\mathfrak{V}''$ along the above isomorphism $\mathfrak{W}''_{1}\cong\mathfrak{W}''_{2}$, we obtain a formal $R$-model $\mathfrak{X}$ of $X$ with the desired properties.\end{proof}
We would like to summarize our results in a form which makes the analogy with Raynaud's theory of formal models more apparent (cf. \cite{BL1}, Theorem~4.1). This formulation relies on the following lemma. 
\begin{lemma}\label{Multiplicative system}Fix a complete adic ring $R$ with non-zero-divisor $\varpi\in R$ generating an ideal of definition. The set $\mathcal{S}$ of compositions of normalized formal blow-ups is a right multiplicative system in the category $\mathcal{C}$ of locally stably uniform $\varpi$-torsion-free qcqs adic formal $R$-schemes which are integrally closed in their generic fibers. \end{lemma}
\begin{proof}We verify that $\mathcal{S}$ satisfies the axioms RMS1-3 of \cite{Stacks}, Tag~04VC. 

RMS1: Since every $\mathfrak{X}\in \mathcal{C}$ is integrally closed in its generic fiber, the identity of $\mathfrak{X}$ is a normalized formal blow-up. Moreover, $\mathcal{S}$ is closed under compositions (of composable elements) by definition. 

RMS2: Consider a solid diagram \begin{center}\begin{tikzcd}& \mathfrak{Y} \arrow{d}{s} \\ \mathfrak{Z} \arrow{r}{g_{0}} & \mathfrak{W}\end{tikzcd}\end{center}with $s\in \mathcal{S}$. We want to prove that there exist morphisms $f_{0}: \mathfrak{X}\to \mathfrak{Y}$ and $t: \mathfrak{X}\to \mathfrak{Z}$ with $t\in \mathcal{S}$ making the diagram \begin{center}\begin{tikzcd} \mathfrak{X}\arrow{r}{f_{0}} \arrow{d}{t} & \mathfrak{Y} \arrow{d}{s} \\ \mathfrak{Z} \arrow{r}{g_{0}} & \mathfrak{W}\end{tikzcd}\end{center}commute. Since $s$ is a composition of normalized formal blow-ups, $s_{\eta}: \mathfrak{Y}_{\eta}^{\ad}\to \mathfrak{W}_{\eta}^{\ad}$ is an isomorphism. Then we can apply Lemma~\ref{Fullness} to the morphism of adic spaces \begin{equation*}f=s_{\eta}^{-1}\circ g_{0\eta}: \mathfrak{Z}_{\eta}^{\ad}\to \mathfrak{Y}_{\eta}^{\ad}\end{equation*}and obtain the desired conclusion with $t$ being a normalized formal blow-up. 

RMS3: Let $f_{0}, g_{0}: \mathfrak{X}\to \mathfrak{Y}$ be a pair of morphisms in $\mathcal{C}$ and let \begin{equation*}(s: \mathfrak{Y}\to \mathfrak{Z})\in \mathcal{S}\end{equation*}such that $s\circ f_{0}=s\circ g_{0}$. We want to prove that this implies $f_{0}=g_{0}$. But $s_{\eta}$ is an isomorphism, so $f_{0\eta}=g_{0\eta}$ on the generic fibers. We conclude by applying Lemma~\ref{Faithfulness}.\end{proof}
\begin{thm}\label{Analog of Raynaud theory}Let $R$ be a complete adic ring with ideal of definition generated by a single non-zero-divisor $\varpi\in R$ and suppose that the Tate ring $R[\varpi^{-1}]$ is sheafy. The functor \begin{equation*}\mathfrak{X}\mapsto \mathfrak{X}_{\eta}^{\ad}\end{equation*}gives rise to an equivalence between \begin{enumerate}[(1)]\item the category of locally stably uniform $\varpi$-torsion-free quasi-compact quasi-separated adic formal $R$-schemes which are integrally closed in their generic fibers, localized by the right multiplicative system of compositions of normalized formal blow-ups, and 
\item the category of uniform quasi-compact quasi-separated adic spaces over $\Spa(R[\varpi^{-1}], \overline{R})$.\end{enumerate}\end{thm}
\begin{proof}That the functor takes values in the category of uniform adic spaces is the content of Lemma~\ref{Uniform generic fiber}. Lemma~\ref{Faithfulness} shows that the functor is faithful, while Lemma~\ref{Fullness} shows that it is full. Finally, Theorem~\ref{Existence of formal models} ensures that the functor is essentially surjective.\end{proof}
For adic spaces of finite type over the base $\Spa(R[\varpi^{-1}], R)$, we can also prove a more immediate analog of Raynaud's theory of formal models. This rests on the following lemma, which is a straightforward generalization of \cite{BL1}, Lemma~4.5. 
\begin{lemma}\label{Finite formal modifications are admissible formal blow-ups}Every $\varpi$-torsion-free finite formal modification $f_{0}: \mathfrak{X}'\to\mathfrak{X}$ of an affine rig-sheafy $\varpi$-torsion-free qcqs adic formal $R$-scheme $\mathfrak{X}$ is an admissible formal blow-up.\end{lemma}
\begin{proof}The proof of \cite{BL1}, Lemma~4.5, carries over almost verbatim; we include the argument for the sake of completeness. Let $f_1,\dots, f_n\in \mathcal{O}_{\mathfrak{X}'}(\mathfrak{X}')$ be elements which generate $\mathcal{O}_{\mathfrak{X}'}(\mathfrak{X}')$ as an $\mathcal{O}_{\mathfrak{X}}(\mathfrak{X})$-module. Since $\mathfrak{X}'\to\mathfrak{X}$ is a formal modification, $\mathcal{O}_{\mathfrak{X}}(\mathfrak{X})[\varpi^{-1}]=\mathcal{O}_{\mathfrak{X}'}(\mathfrak{X}')[\varpi^{-1}]$. Hence there exists an integer $r\geq 0$ such that $\varpi^{r}f_1,\dots, \varpi^{r}f_n\in \mathcal{O}_{\mathfrak{X}}(\mathfrak{X})$. Let $\mathfrak{X}''\to\mathfrak{X}$ be the admissible formal blow-up of $\mathfrak{X}$ in the ideal sheaf $\mathcal{I}$ defined by the open ideal \begin{equation*}I=(\varpi^{r},\varpi^{r}f_1,\dots, \varpi^{r}f_n)_{\mathcal{O}_{\mathfrak{X}}(\mathfrak{X})}.\end{equation*}We have \begin{equation*}\mathfrak{X}'=\Spf(\mathcal{O}_{\mathfrak{X}}(\mathfrak{X})[f_1,\dots, f_n])=\Spf(\mathcal{O}_{\mathfrak{X}}(\mathfrak{X})\langle f_1,\dots, f_n\rangle)=\Spf(\mathcal{O}_{\mathfrak{X}}(\mathfrak{X})\langle\frac{\varpi^{r},\varpi^{r}f_1,\dots, \varpi^{r}f_n}{\varpi^{r}}\rangle),\end{equation*}so $\mathfrak{X}'$ equals the open subset of $\mathfrak{X}''$ on which the stalks of the ideal sheaf $\mathcal{I}\mathcal{O}_{\mathfrak{X}''}$ are generated by $\varpi^{r}$. We claim that this open subset is equal to all of $\mathfrak{X}''$. To prove this, we have to verify the equalities \begin{equation*}IA_{i}=\varpi^{r}A_{i}\end{equation*}for $i=1,\dots, n$, where \begin{equation*}A_{i}=\mathcal{O}_{\mathfrak{X}}(\mathfrak{X})\langle\frac{\varpi^{r}}{\varpi^{r}f_{i}},\frac{\varpi^{r}f_{1}}{\varpi^{r}f_{i}},\dots,\frac{\varpi^{r}f_{n}}{\varpi^{r}f_{i}}\rangle.\end{equation*}Since the ideal $IA_{i}$ of the completed blow-up algebra $A_{i}=\mathcal{O}_{\mathfrak{X}}(\mathfrak{X})\langle\frac{I}{\varpi^{r}f_{i}}\rangle$ is generated by $\varpi^{r}f_{i}$, it suffices to prove that the element \begin{equation*}f_{i}\in A_{i}[f_{i}]\subseteq \mathcal{O}_{\mathfrak{X}'}(\mathfrak{X}')\langle\frac{\varpi^{r}}{\varpi^{r}f_{i}}, \frac{\varpi^{r}f_1}{\varpi^{r}f_{i}},\dots, \frac{\varpi^{r}f_{n}}{\varpi^{r}f_{i}}\rangle\end{equation*}belongs to $A_{i}$ (as in this case $f_{i}$ is an invertible element of $A_{i}$, with inverse $\frac{\varpi^{r}}{\varpi^{r}f_{i}}$). Since $\mathcal{O}_{\mathfrak{X}}(\mathfrak{X})[f_{i}]\subseteq \mathcal{O}_{\mathfrak{X}'}(\mathfrak{X}')$ is integral over $\mathcal{O}_{\mathfrak{X}}(\mathfrak{X})$, the element $f_{i}\in A_{i}[f_{i}]$ is, a fortiori, integral over $A_{i}$. But, at the same time, $f_{i}$ is a unit in $A_{i}[f_{i}]$ with inverse $\frac{\varpi^{r}}{\varpi^{r}f_{i}}\in A_{i}$, so we can multiply an integral equation \begin{equation*}f_{i}^{N}+a_{1}f_{i}^{N-1}+\dots+a_{N-1}f_{i}+a_{N}=0\end{equation*}of $f_{i}$ over $A_{i}$ by $f_{i}^{-(N-1)}$ to see that $f_{i}\in A_{i}$. This concludes the proof.\end{proof}
We call a morphism of adic spaces $f: X\to S$ locally of finite type if for any pair of affinoid open subspaces $V\subseteq X$, $U\subseteq S$ with $f(V)\subseteq U$ the corresponding morphism of Huber pairs $(\mathcal{O}_{S}(U), \mathcal{O}_{S}^{+}(U))\to (\mathcal{O}_{X}(V), \mathcal{O}_{X}^{+}(V))$ is of finite type. We call a morphism of adic spaces $f: X\to S$ of finite type if it is locally of finite type and quasi-compact. For adic spaces of finite type over a Tate affinoid adic space, we have the following non-Noetherian version of Raynaud's theory of formal models.
\begin{thm}\label{Analog of Raynaud theory 2}Let $R$ be a complete adic ring with ideal of definition generated by a single non-zero-divisor $\varpi\in R$. Suppose that the Tate ring $R[\varpi^{-1}]$ is sheafy and that $R$ is integrally closed in $R[\varpi^{-1}]$. The functor \begin{equation*}\mathfrak{X}\mapsto\mathfrak{X}_{\eta}^{\ad}\end{equation*}gives rise to an equivalence between \begin{enumerate}[(1)]\item the category of locally rig-sheafy $\varpi$-torsion-free quasi-separated adic formal $R$-schemes topologically of finite type over $\Spf(R)$, localized by admissible formal blow-ups, and \item the category of quasi-separated adic spaces of finite type over $\Spa(R[\varpi^{-1}], R)$.\end{enumerate}\end{thm}
\begin{proof}Faithfulness of the functor is a special case of Lemma~\ref{Faithfulness}. To prove that the functor is full, let $\mathfrak{X}$ and $\mathfrak{Z}$ be locally rig-sheafy quasi-separated $\varpi$-torsion-free adic formal $R$-schemes topologically of finite type and let $f: \mathfrak{Z}_{\eta}^{\ad}\to \mathfrak{X}_{\eta}^{\ad}$ be a morphism of adic spaces over $\Spa(R[\varpi^{-1}], R)$ between their generic fibers. By the same argument as in the second paragraph of the proof of Lemma~\ref{Fullness}, we may assume that $\mathfrak{X}$ and $\mathfrak{Z}$ are affine. Recall that $\mathfrak{Z}_{\eta}^{\ad}$ and $\mathfrak{X}_{\eta}^{\ad}$ are both of finite type over $\Spa(R[\varpi^{-1}], R)$. Set $A_{0}=\mathcal{O}_{\mathfrak{X}}(\mathfrak{X})$ and $B_{0}=\mathcal{O}_{\mathfrak{Z}}(\mathfrak{Z})$, which are rings of definition of $A=\mathcal{O}_{\mathfrak{X}_{\eta}^{\ad}}(\mathfrak{X}_{\eta}^{\ad})$ and $B=\mathcal{O}_{\mathfrak{Z}_{\eta}^{\ad}}(\mathfrak{Z}_{\eta}^{\ad})$, respectively. Note that $A_{0}$ and $B_{0}$ are topologically of finite type over $R$. Consider the continuous ring map \begin{equation*}\varphi: A\to B\end{equation*}induced by $f$ and let $f_1,\dots, f_n\in A_{0}$ be topological generators of $A_{0}$ as an $R$-algebra, i.e., $A_{0}=R\langle f_1,\dots, f_n\rangle$. Set \begin{equation*}C_{0}=B_{0}[\varphi(f_1),\dots,\varphi(f_n)]\subseteq B.\end{equation*}Since \begin{equation*}\varphi(A_{0})\subseteq \varphi(A^{+})\subseteq B^{+}=\overline{B_{0}},\end{equation*}the subring $C_{0}$ is contained in $\overline{B_{0}}$. Consequently, $C_{0}$ is integral over $B_{0}$. Since $C_{0}$ is a finitely generated $B_{0}$-algebra by definition, it is a finite $B_{0}$-algebra. By Lemma~\ref{Finite formal modifications are admissible formal blow-ups}, this entails that the morphism $\Spf(C_{0})\to \Spf(B_{0})$ is an admissible formal blow-up. But by construction, the morphism of formal schemes \begin{equation*}f_{0}: \Spf(C_{0})\to \Spf(A_{0})\end{equation*}induced by the restriction of $\varphi$ to $A_{0}$ induces the given morphism $f$ on the generic fiber. This shows that the functor in the theorem is indeed full. The analog of Lemma~\ref{Isomorphisms} can be deduced from this using the universal property of admissible formal blow-ups as in the proof of Lemma~\ref{Isomorphisms} (that is, as in the proof of \cite{BoschLectures}, \S8.4, Lemma~4(d)). The essential surjectivity of the functor can then be deduced from its fullness and the analog of Lemma~\ref{Isomorphisms} by the same argument as in the proof of \cite{BL1}, Theorem~4.1(e) (i.e., by the same argument as in the proof of Theorem~\ref{Existence of formal models} above).\end{proof}

\section*{Data Availability}No data was used in this article.

\section*{Conflict of Interest}The author states that there is no conflict of interest.

\bibliographystyle{plain} 
\bibliography{Bib}

\textsc{Department of Mathematics, University of California San Diego, La Jolla, CA 92093, United States} \newline 

E-mail address: \textsf{ddine@ucsd.edu}

\end{document}